\documentclass[a4paper,11pt]{amsart}
\usepackage{fixltx2e}
\usepackage{amsmath}
\usepackage{amssymb}
\usepackage{amsthm}
\usepackage{amsfonts}
\usepackage{amscd}
\usepackage{amsbsy}
\usepackage{psfrag}
\usepackage{geometry}
\usepackage{tikz}
\usepackage{mathalfa}
\usepackage{mathrsfs}
\usepackage{graphics}
\usepackage{graphicx}
\usepackage{dsfont}
\usepackage{euscript}
\usepackage{hyperref}
\usepackage{color}
\usepackage{xcolor}
\makeatletter
\renewcommand\subsection
{
\@startsection{subsection}{2}
\z@{0.1875\linespacing \@plus 0.1875\linespacing}
{0.125\linespacing}{\centering\normalfont\scshape}
}
\makeatother
\makeatletter
\@namedef{subjclassname@2020}{\textup{2020} Mathematics Subject Classification}
\makeatother
\theoremstyle{plain}

\newtheorem{lemma}{Lemma}
\newtheorem{theorem}{Theorem}
\newtheorem{corollary}{Corollary}
\theoremstyle{remark}
\newtheorem{remark}{Remark}

\newtheorem*{case*}{Case}

\newtheorem*{claim*}{Claim}
\newtheorem{motivation}{Motivation}
\newtheorem{question}{Question}
\begin{document}
\allowdisplaybreaks[4]
\title[Existence of Higher Extremal K\"ahler Metrics on a Minimal Ruled Surface]{Existence of Higher Extremal K\"ahler Metrics on a Minimal Ruled Surface}
\author[Rajas Sandeep Sompurkar]{Rajas Sandeep Sompurkar}
\address{Department of Mathematics, Indian Institute of Science, Bangalore - 560012, Karnataka, India}
\email{rajass@iisc.ac.in}
\thanks{The author is supported by the DST-INSPIRE Doctoral Fellowship No. IF180232 and is also grateful for the support received from IISc Bangalore.}
\date{September 15, 2023}
\begin{abstract}
In this paper we prove that on a special type of minimal ruled surface, which is an example of a `pseudo-Hirzebruch surface', every K\"ahler class admits a certain kind of `higher extremal K\"ahler metric', which is a K\"ahler metric whose corresponding top Chern form and volume form satisfy a nice equation motivated by analogy with the equation characterizing an extremal K\"ahler metric. From an already proven result, it will follow that this specific higher extremal K\"ahler metric cannot be a `higher constant scalar curvature K\"ahler (hcscK) metric', which is defined, again by analogy with the definition of a constant scalar curvature K\"ahler (cscK) metric, to be a K\"ahler metric whose top Chern form is harmonic. By doing a certain set of computations involving the top Bando-Futaki invariant we will conclude that hcscK metrics do not exist in any K\"ahler class on this surface.
\end{abstract}
\maketitle
\numberwithin{equation}{subsection}
\numberwithin{definition}{subsection}
\numberwithin{lemma}{subsection}
\numberwithin{theorem}{subsection}
\numberwithin{corollary}{subsection}
\numberwithin{remark}{subsection}
\numberwithin{case}{subsection}
\numberwithin{claim}{subsection}
\numberwithin{motivation}{subsection}
\numberwithin{question}{subsection}
\section{Introduction}\label{sec:Intro}
\subsection{Background of the Study}\label{subsec:Background}
Extremal K\"ahler metrics were introduced by Calabi \cite{Calabi:1982:eK,Calabi:1985:eK2} as candidates for `canonical metrics' (Tian \cite{Tian:2000:CanonMetr}) in a given K\"ahler class on a compact K\"ahler manifold. These are generalizations of constant scalar curvature K\"ahler (cscK) metrics which themselves include K\"ahler-Einstein metrics as a special case (see Aubin \cite{Aubin:1998:NonlinRiemGeom}, Barth, Hulek et al. \cite{Barth:2004:CmpctCmplxSurf}, Sz\'ekelyhidi \cite{Szekelyhidi:2014:eKintro} and Tian \cite{Tian:2000:CanonMetr} and the references therein). Extremal K\"ahler metrics can be equivalently defined as those K\"ahler metrics for which the gradient of the scalar curvature (which is expressed in terms of the first Chern form) is a holomorphic vector field (viz. $\nabla^{1,0} \mathrm{S} \left(\omega\right) = \nabla^{\left(1,0\right)} \left(\frac{2 n \pi c_1 \left(\omega\right) \wedge \omega^{n - 1}}{\omega^n}\right)$ is a holomorphic vector field) (see \cite{Calabi:1982:eK,Calabi:1985:eK2}, Pingali \cite{Pingali:2018:heK}, \cite{Szekelyhidi:2014:eKintro} and \cite{Tian:2000:CanonMetr}). Special cases of these are cscK metrics which can be characterized as those metrics for which the corresponding first Chern form is harmonic (viz. $\Delta c_1 \left(\omega\right) = - \bar{\partial}^* \bar{\partial} c_1 \left(\omega\right) = 0$) (see Bando \cite{Bando:2006:HarmonObstruct}, \cite{Calabi:1982:eK,Calabi:1985:eK2}, Futaki \cite{Futaki:1983:ObstructKE}, Pingali \cite{Pingali:2018:heK}, \cite{Szekelyhidi:2014:eKintro} and \cite{Tian:2000:CanonMetr}). These definitions (and even the definition of a K\"ahler-Einstein metric (\cite{Futaki:1983:ObstructKE} and \cite{Szekelyhidi:2014:eKintro})) involve the first Chern form which is the same as the Ricci form for K\"ahler manifolds \cite{Szekelyhidi:2014:eKintro}; this fact and the fact, that the first Chern class of a K\"ahler manifold does not depend on the choice of the K\"ahler metric on the manifold \cite{Szekelyhidi:2014:eKintro}, give some interesting results in the theory of extremal K\"ahler metrics (see \cite{Calabi:1982:eK,Calabi:1985:eK2}, \cite{Futaki:1983:ObstructKE}, \cite{Szekelyhidi:2014:eKintro} and Yau \cite{Yau:1977:CalabiConject}). \par
Taking the analogy of these definitions to the top Chern form, Pingali \cite{Pingali:2018:heK} defined \textit{higher extremal K\"ahler metrics} and \textit{higher constant scalar curvature K\"ahler (hcscK) metrics} by considering the following equation (where $c_n \left(\omega\right)$ is the top Chern form of a K\"ahler metric $\omega$ on a compact K\"ahler $n$-manifold $\mathrm{M}$ and $\lambda$ is a smooth real-valued function on $\mathrm{M}$):
\begin{equation}\label{eq:defheK}
c_n \left(\omega\right) = \lambda \omega^n
\end{equation}
where $\omega$ is said to be \textit{higher extremal K\"ahler} if $\nabla^{1,0} \lambda = \left(\bar{\partial} \lambda\right)^{\sharp}$ is a holomorphic vector field on $\mathrm{M}$ and $\omega$ is said to be \textit{hcscK} if $\lambda$ is a constant, which is equivalent to saying $c_n \left(\omega\right)$ is a harmonic top form on $\mathrm{M}$ (see Bando \cite{Bando:2006:HarmonObstruct}). The primary motivations for studying these types of K\"ahler metrics were that firstly the top Chern class of a K\"ahler manifold is the same as its Euler class (see Barth, Hulek et al. \cite{Barth:2004:CmpctCmplxSurf}) and secondly Yau \cite{Yau:2006:Perspectives} had stated that the behaviour of the higher Chern forms is mysterious in general and was as then unexplored, so some interesting results were expected from the study of these objects. \par
Earlier, hcscK metrics were studied by Bando \cite{Bando:2006:HarmonObstruct}, who defined obstructions for the existence of the same in a K\"ahler class on a compact K\"ahler manifold (which are now called as Bando-Futaki invariants), and also by Futaki \cite{Futaki:2006:HarmonicChern,Futaki:2008:HolomorphicVF}, whereas a version of higher extremal K\"ahler metrics called `perturbed extremal K\"ahler metrics' (and analogously a version of hcscK metrics called `perturbed constant scalar curvature K\"ahler (perturbed cscK) metrics') were studied by Futaki \cite{Futaki:2006:HarmonicChern,Futaki:2008:HolomorphicVF}. But Futaki's results \cite{Futaki:2006:HarmonicChern,Futaki:2008:HolomorphicVF} do not seem to apply in our case as mentioned by Pingali \cite{Pingali:2018:heK}. \par
This paper answers the problem of finding higher extremal K\"ahler metrics with some nice enough symmetries on a certain example of minimal ruled surface, which comes from a special family of minimal ruled surfaces called `pseudo-Hirzebruch surfaces', which like Hirzebruch surfaces are some of the most important examples of compact K\"ahler surfaces (see Barth, Hulek et al. \cite{Barth:2004:CmpctCmplxSurf} and T{\o}nnesen-Friedman \cite{Tonnesen:1998:eKminruledsurf}). The problem of constructing higher extremal K\"ahler metrics on this minimal ruled surface was started by Pingali \cite{Pingali:2018:heK} motivated by the analogy with the problem of constructing extremal K\"ahler metrics on this surface which was dealt with first by T{\o}nnesen-Friedman \cite{Tonnesen:1998:eKminruledsurf} (and later by Apostolov, Calderbank et al. \cite{Apostolov:2008:Hamiltonian}). Pingali \cite{Pingali:2018:heK} had proven the existence of a higher extremal K\"ahler metric with some required properties in one specific K\"ahler class on this K\"ahler surface, while we will prove in this paper that in every K\"ahler class on this K\"ahler surface there exists a higher extremal K\"ahler representative (which is constructed by a certain method which imposes those required nice symmetries on the metric), and it was already proven in \cite{Pingali:2018:heK} that this constructed metric, if it does exist, cannot be hcscK. Then after proving some facts about the top Bando-Futaki invariant on a general compact K\"ahler manifold we will be able to conclude that hcscK metrics do not exist in any K\"ahler class on our K\"ahler surface. We will then generalize our results for all pseudo-Hirzebruch surfaces. Finally we will do a brief comparison of our results in the higher extremal K\"ahler case with those gotten by T{\o}nnesen-Friedman \cite{Tonnesen:1998:eKminruledsurf} (and Apostolov, Calderbank et al. \cite{Apostolov:2008:Hamiltonian}) in the usual extremal K\"ahler analogue of this problem.
\subsection{Overview of the Paper}\label{subsec:Overview}
In this paper we will consider the minimal ruled surface $X := \mathbb{P} \left(\mathrm{L} \oplus \mathcal{O}\right)$ where $\mathrm{L}$ is a degree $-1$ holomorphic line bundle on a genus $2$ Riemann surface $\Sigma$, $\Sigma$ is equipped with a K\"ahler metric $\omega_\Sigma$ of constant scalar curvature $-2$ and $\mathrm{L}$ is equipped with a Hermitian metric $h$ whose curvature form is $-\omega_\Sigma$. The problem of finding extremal K\"ahler metrics on $X$ was studied by T{\o}nnesen-Friedman \cite{Tonnesen:1998:eKminruledsurf}, more general results regarding the same were already proven in \cite{Tonnesen:1998:eKminruledsurf}, then further results in more generalized settings were proven by Apostolov, Calderbank et al. \cite{Apostolov:2008:Hamiltonian} and a complete exposition about the specific case of $X$ is contained in Sz\'ekelyhidi \cite{Szekelyhidi:2014:eKintro} where the momentum construction method outlined in Hwang-Singer \cite{Hwang:2002:MomentConstruct} is used to obtain the required symmetries on the metric. In this paper we will deal with the higher extremal K\"ahler analogue of this problem and prove some important results on the same. \par
It has been proven using the Leray-Hirsch Theorem and the Nakai-Moishezon Criterion (proven by Fujiki \cite{Fujiki:1992:eKruledmani} and T{\o}nnesen-Friedman \cite{Tonnesen:1998:eKminruledsurf} and briefly explained in Subsection \ref{subsec:KCone}) that the K\"ahler cone (i.e. the set of all K\"ahler classes) of $X$ is precisely the following set:
\begin{equation}\label{eq:KConeX0}
H^{\left(1,1\right)} \left(X, \mathbb{R}\right)^+ = \left\lbrace a \mathsf{C} + b S_\infty \hspace{3pt} \vert \hspace{3pt} a, b > 0 \right\rbrace \subseteq H^{\left(1,1\right)} \left(X, \mathbb{R}\right) \subseteq H^2 \left(X, \mathbb{R}\right) = \mathbb{R} \mathsf{C} \oplus \mathbb{R} S_\infty
\end{equation}
where $\mathsf{C}$ is the Poincar\'e dual of a typical fibre of $X$ and $S_\infty$ is the image of $\Sigma$ in $X$ sitting as the `infinity divisor' of $X$. Following the usual extremal K\"ahler case as in Sz\'ekelyhidi \cite{Szekelyhidi:2014:eKintro}, Pingali \cite{Pingali:2018:heK} took $a = 2 \pi$, $b = 2 m \pi$ where $m > 0$ and posed the problem of finding a higher extremal K\"ahler metric $\omega$ on $X$, analogously using the momentum construction method of Hwang-Singer \cite{Hwang:2002:MomentConstruct} (described in Subsection \ref{subsec:MomentConstruct}), satisfying the following:
\begin{equation}\label{eq:problem}
\left[\omega\right] = 2 \pi \left(\mathsf{C} + m S_\infty\right) \hspace{1pt}, \hspace{5pt} c_2 \left(\omega\right) = \frac{\lambda}{2 \left(2 \pi\right)^2} \omega^2 \hspace{1pt}, \hspace{5pt} \nabla^{1,0} \lambda \in \mathfrak{h} \left(X\right)
\end{equation}
where $\left[\omega\right]$ denotes the K\"ahler class of $\omega$ and $\mathfrak{h} \left(X\right)$ denotes the set of all holomorphic vector fields on $X$. \par
Pingali \cite{Pingali:2018:heK} solved the problem (\ref{eq:problem}) for $m = 1$ (i.e. constructed the required higher extremal K\"ahler metric $\omega$ in the K\"ahler class $2 \pi \left(\mathsf{C} + S_\infty\right)$) and conjectured that there might exist a maximum value of $m$ beyond which there may not exist a solution to the problem (\ref{eq:problem}), meaning there perhaps may not exist higher extremal K\"ahler metrics in the K\"ahler classes $2 \pi \left(\mathsf{C} + m S_\infty\right)$ for values of $m$ larger than this maximum value. This was expected by analogy with the usual extremal K\"ahler analogue of the problem (\ref{eq:problem}) (refer to Apostolov, Calderbank et al. \cite{Apostolov:2008:Hamiltonian}, Sz\'ekelyhidi \cite{Szekelyhidi:2014:eKintro} and T{\o}nnesen-Friedman \cite{Tonnesen:1998:eKminruledsurf}). But on the contrary we will show (in Subsections \ref{subsec:AnalysisODEBVP}, \ref{subsec:Proof1} and \ref{subsec:Proof2}) that higher extremal K\"ahler metrics with the required properties exist in the K\"ahler classes $2 \pi \left(\mathsf{C} + m S_\infty\right)$ for all positive values of $m$. Since being a higher extremal K\"ahler metric is a scale-invariant property (as will be seen in Subsection \ref{subsec:AnalysisODEBVP}) an appropriate rescaling procedure applied on the K\"ahler metric and its top Chern form will enable us to construct the required higher extremal K\"ahler metrics in the K\"ahler classes $a \mathsf{C} + b S_\infty$ with $a, b > 0$, which exhaust the K\"ahler cone of $X$. \par
As will be seen in Subsections \ref{subsec:MomentConstruct} and \ref{subsec:AnalysisODEBVP}, solving the problem (\ref{eq:problem}) eventually boils down to solving an ODE BVP for the `momentum profile' $\phi$ of the K\"ahler metric $\omega$ (a smooth real-valued function of a real variable $\gamma$ which appears in the momentum construction method), depending on $m > 0$ and one more real parameter $C$, and the ODE in our case is not readily integrable and also not autonomous (it is a version of Chini's Equation) unlike the usual extremal K\"ahler case given in \cite{Szekelyhidi:2014:eKintro} and \cite{Tonnesen:1998:eKminruledsurf}, and hence requires a very delicate analysis for the existence of a solution satisfying all the boundary conditions (see Pingali \cite{Pingali:2018:heK}). Pingali \cite{Pingali:2018:heK} used some explicit numerical estimates specific to the case $m = 1$ and managed to solve it for $m = 1$, but our method of proving the existence of a solution to the required ODE BVP (explained in Subsections \ref{subsec:Proof1} and \ref{subsec:Proof2}) is completely different from that of Pingali \cite{Pingali:2018:heK} as we are doing it for an arbitrary positive value of $m$. Roughly speaking, we will fix $m > 0$ and drop the final boundary condition and look at the resultant ODE IVP, and then find the smooth solutions to the ODE IVP depending on the parameter $C$ and study the variation of the final boundary value of these solutions w.r.t. $C$, which will eventually lead us to a value of $C$ for which the required final boundary condition holds. \par
Then after constructing the higher extremal K\"ahler metric with the required nice symmetries in each K\"ahler class of $X$ we will tackle the problem of (non-)existence of hcscK metrics on $X$ in Section \ref{sec:Bando-Futaki}. We will first prove that the top Bando-Futaki invariant $\mathcal{F}_n : \mathfrak{h} \left(\mathrm{M}\right) \times H^{\left(1,1\right)} \left(\mathrm{M}, \mathbb{R}\right)^+ \to \mathbb{R}$ for a compact K\"ahler $n$-manifold $\mathrm{M}$ can be re-expressed in the following way, if the K\"ahler metric $\omega$ on $\mathrm{M}$ is higher extremal K\"ahler:
\begin{equation}\label{eq:BFexpr}
\mathcal{F}_n \left(Y, \left[\omega\right]\right) = - \int\limits_{\mathrm{M}} \left(\lambda - \lambda_0\right)^2 \omega^n =: - \lVert \lambda - \lambda_0 \rVert_{\mathcal{L}^2 \left(\mathrm{M}, \omega\right)}^2
\end{equation}
where $\lambda \in \mathcal{C}^\infty \left(\mathrm{M}, \mathbb{R}\right)$ satisfies equation (\ref{eq:defheK}) and $\lambda_0 \in \mathbb{R}$ is a constant. Then it will follow that given $\omega$ is higher extremal K\"ahler, $\mathcal{F}_n \left(\cdot, \left[\omega\right]\right) \equiv 0$ if and only if $\omega$ is hcscK, and further in the K\"ahler class of an hcscK metric on $\mathrm{M}$ every higher extremal K\"ahler representative is hcscK. This statement along with the fact, that our constructed higher extremal K\"ahler metrics on our surface $X$ are not hcscK (proven by Pingali \cite{Pingali:2018:heK} given their existence), will help us in concluding the complete non-existence of hcscK metrics (even without the nice symmetries imposed by the momentum construction method) on $X$. \par
In Section \ref{sec:Generalgd} we will take the genus of the Riemann surface $\Sigma$ to be $\mathtt{g} \geq 2$, the constant scalar curvature of the K\"ahler metric $\omega_\Sigma$ on $\Sigma$ to be $-2 \left(\mathtt{g} - 1\right)$, the degree of the holomorphic line bundle $\mathrm{L}$ to be $\mathtt{d} \neq 0$ and the curvature form of the Hermitian metric $h$ on $\mathrm{L}$ to be $\mathtt{d} \omega_\Sigma$. We will see that for the minimal ruled surface $X := \mathbb{P} \left(\mathrm{L} \oplus \mathcal{O}\right)$ in this general setting as well, the same analysis and all the arguments as in the special case of $\mathtt{g} = 2$ and $\mathtt{d} = -1$ (shown in Sections \ref{sec:ConstructheK}, \ref{sec:Proof} and \ref{sec:Bando-Futaki}) go through well and we can obtain all the results about the existence of higher extremal K\"ahler metrics and the non-existence of hcscK metrics on $X$ in the general case as well. \par
We will finally summarize our results on higher extremal K\"ahler (and hcscK) metrics and the top Bando-Futaki invariant and compare them with the analogous results of Apostolov, Calderbank et al. \cite{Apostolov:2008:Hamiltonian}, Calabi \cite{Calabi:1985:eK2}, LeBrun-Simanca \cite{LeBrun:1994:eKcmplxdeform} and T{\o}nnesen-Friedman \cite{Tonnesen:1998:eKminruledsurf} on extremal K\"ahler (and cscK) metrics and the Futaki invariant (exposition contained in Sz\'ekelyhidi \cite{Szekelyhidi:2014:eKintro}) in Section \ref{sec:Analogy}.
\numberwithin{equation}{subsection}
\numberwithin{definition}{subsection}
\numberwithin{lemma}{subsection}
\numberwithin{theorem}{subsection}
\numberwithin{corollary}{subsection}
\numberwithin{remark}{subsection}
\numberwithin{case}{subsection}
\numberwithin{claim}{subsection}
\numberwithin{motivation}{subsection}
\numberwithin{question}{subsection}
\section{Constructing Higher Extremal K\"ahler Metrics on a Minimal Ruled Surface}\label{sec:ConstructheK}
\subsection{The Minimal Ruled Surface and its K\"ahler Cone}\label{subsec:KCone}
Let $\left(\Sigma, \omega_\Sigma\right)$ be a genus $2$ (compact) Riemann surface equipped with a K\"ahler metric of constant scalar curvature $-2$ (and hence area $2 \pi$). Let $\mathrm{L}$ be a degree $-1$ holomorphic line bundle on $\Sigma$ equipped with a Hermitian metric $h$ whose curvature form is $-\omega_\Sigma$. Let $X$ be the minimal ruled (complex) surface $\mathbb{P} \left(\mathrm{L} \oplus \mathcal{O}\right)$ where $\mathcal{O}$ is the trivial line bundle on $\Sigma$ and $\mathbb{P}$ denotes vector bundle projectivization. This surface $X$ is an example of a `pseudo-Hirzebruch surface' (see definition in T{\o}nnesen-Friedman \cite{Tonnesen:1998:eKminruledsurf} and in Section \ref{sec:Generalgd}). \par
Let $\mathsf{C}$ be the Poincar\'e dual of a typical fibre of $X$ (i.e. $\mathsf{C}$ is a copy of the Riemann sphere $S^2$ sitting in $X$), $S_\infty$ (called the \textit{infinity divisor} of $X$) be the image of the subbundle $\mathrm{L} \oplus \left\lbrace 0 \right\rbrace \subseteq \mathrm{L} \oplus \mathcal{O}$ under the bundle projectivization to $X = \mathbb{P} \left(\mathrm{L} \oplus \mathcal{O}\right)$ and similarly $S_0$ (called the \textit{zero divisor} of $X$) be the image of the subbundle $\left\lbrace 0 \right\rbrace \oplus \mathcal{O} \subseteq \mathrm{L} \oplus \mathcal{O}$ under the bundle projectivization to $X = \mathbb{P} \left(\mathrm{L} \oplus \mathcal{O}\right)$, so that $S_\infty$ and $S_0$ are actually copies of $\Sigma$ sitting in $X$ as its infinity and zero divisors respectively (and $\Sigma$ is identified with $S_0$ as a (complex) curve in $X$). We have the following intersection formulae (refer to Barth, Hulek et al. \cite{Barth:2004:CmpctCmplxSurf}, Sz\'ekelyhidi \cite{Szekelyhidi:2014:eKintro} and T{\o}nnesen-Friedman \cite{Tonnesen:1998:eKminruledsurf}):
\begin{equation}\label{eq:IntersectForm}
\mathsf{C}^2 = 0 \hspace{1pt}, \hspace{5pt} S_\infty^2 = 1 \hspace{1pt}, \hspace{5pt} S_0^2 = -1 \hspace{1pt}, \hspace{5pt} \mathsf{C} \cdot S_\infty = 1 \hspace{1pt}, \hspace{5pt} \mathsf{C} \cdot S_0 = 1 \hspace{1pt}, \hspace{5pt} S_\infty \cdot S_0 = 0
\end{equation}
and also the following intersection formulae (again refer to \cite{Barth:2004:CmpctCmplxSurf}, \cite{Szekelyhidi:2014:eKintro} and \cite{Tonnesen:1998:eKminruledsurf}), where $c_1 \left(\mathrm{L}\right) \in H^{\left(1,1\right)} \left(\Sigma, \mathbb{R}\right) = H^2 \left(\Sigma, \mathbb{R}\right)$ is the first Chern class of $\mathrm{L}$, $\left[\omega_\Sigma\right] \in H^{\left(1,1\right)} \left(\Sigma, \mathbb{R}\right) = H^2 \left(\Sigma, \mathbb{R}\right)$ is the K\"ahler class of $\omega_\Sigma$ and $\left[\Sigma\right] \in H^2 \left(\Sigma, \mathbb{R}\right)$ is the fundamental class of $\Sigma$ (and $\left[\Sigma\right]$ is identified with $S_0$ in $H^2 \left(X, \mathbb{R}\right)$):
\begin{equation}\label{eq:Sigma}
c_1 \left(\mathrm{L}\right) \cdot \left[\Sigma\right] = -1 \hspace{1pt}, \hspace{5pt} \left[\omega_\Sigma\right] \cdot \left[\Sigma\right] = 2 \pi \hspace{1pt}, \hspace{5pt} \mathsf{C} \cdot \left[\Sigma\right] = 1 \hspace{1pt}, \hspace{5pt} S_\infty \cdot \left[\Sigma\right] = 0 \hspace{1pt}, \hspace{5pt} S_0 \cdot \left[\Sigma\right] = -1
\end{equation} \par
By the Leray-Hirsch Theorem we have $H^2 \left(X, \mathbb{R}\right) = \mathbb{R} \mathsf{C} \oplus \mathbb{R} S_\infty$. So starting with a general cohomology class $\alpha := a \mathsf{C} + b S_\infty \in H^{\left(1,1\right)} \left(X, \mathbb{R}\right) \subseteq H^2 \left(X, \mathbb{R}\right)$ we have the following criterion for determining when $\alpha$ will be a K\"ahler class on $X$ (integral cohomology version attributed to Nakai-Moishezon and explained in Barth, Hulek et al. \cite{Barth:2004:CmpctCmplxSurf} with all the relevant references, while real cohomology version proven in Fujiki \cite{Fujiki:1992:eKruledmani}, Lamari \cite{Lamari:1999:Kcone}, LeBrun-Singer \cite{LeBrun:1993:scalflatKcmpctcmplxsurf} and T{\o}nnesen-Friedman \cite{Tonnesen:1998:eKminruledsurf}):
\begin{theorem}[Nakai-Moishezon Criterion]\label{thm:Nakai-Moishezon}
$\alpha$ is a K\"ahler class on $X$ if and only if the following conditions are satisfied:
\begin{enumerate}
\item $\alpha^2 > 0$.
\item $\alpha \cdot \Gamma > 0$ for every irreducible (complex) curve $\Gamma \subseteq X$.
\end{enumerate}
\end{theorem} \par
It was further shown by Fujiki \cite{Fujiki:1992:eKruledmani} that it suffices to check the intersection number of $\alpha$ against $\Gamma = \mathsf{C} \hspace{1pt}, \hspace{2.5pt} S_\infty \hspace{1pt}, \hspace{2.5pt} S_0 \hspace{3.5pt} \text{and} \hspace{1pt} \left[\Sigma\right]$ only in Theorem \ref{thm:Nakai-Moishezon} (see the explanation given in T{\o}nnesen-Friedman \cite{Tonnesen:1998:eKminruledsurf}):
\begin{corollary}[Fujiki, T{\o}nnesen-Friedman]\label{cor:TFKCone}
$\alpha$ is a K\"ahler class on $X$ if and only if the following conditions are satisfied:
\begin{enumerate}
\item $\alpha^2 > 0$.
\item $\alpha \cdot \mathsf{C} > 0$, $\alpha \cdot S_\infty > 0$, $\alpha \cdot S_0 > 0$ and $\alpha \cdot \left[\Sigma\right] > 0$.
\end{enumerate}
\end{corollary} \par
From Corollary \ref{cor:TFKCone} and the intersection formulae (\ref{eq:IntersectForm}) and (\ref{eq:Sigma}) we derive the following inequalities characterizing all K\"ahler classes $\alpha = a \mathsf{C} + b S_\infty$ on $X$:
\begin{equation}\label{eq:KConeXCond}
2 a b + b^2 > 0 \hspace{1pt}, \hspace{5pt} b > 0 \hspace{1pt}, \hspace{5pt} a + b > 0 \hspace{1pt}, \hspace{5pt} a > 0 \hspace{1pt}, \hspace{5pt} a > 0
\end{equation}
which simply boil down to $a > 0$, $b > 0$. Therefore the K\"ahler cone of $X$ is precisely (\cite{Fujiki:1992:eKruledmani} and \cite{Tonnesen:1998:eKminruledsurf}):
\begin{equation}\label{eq:KConeX}
H^{\left(1,1\right)} \left(X, \mathbb{R}\right)^+ = \left\lbrace a \mathsf{C} + b S_\infty \hspace{3pt} \vert \hspace{3pt} a, b > 0 \right\rbrace
\end{equation} \par
Following Sz\'ekelyhidi \cite{Szekelyhidi:2014:eKintro} we will first consider K\"ahler classes only of the form $\beta := 2 \pi \left(\mathsf{C} + m S_\infty\right)$ with $m > 0$, and after obtaining results about the existence of higher extremal K\"ahler metrics and the non-existence of hcscK metrics in these K\"ahler classes, we will generalize these results to the general K\"ahler classes which are of the form $\alpha = a \mathsf{C} + b S_\infty$ with $a, b > 0$ by using a simple rescaling argument.
\subsection{Description of the Momentum Construction Method}\label{subsec:MomentConstruct}
We give here a brief description of the momentum construction method attributed to Hwang-Singer \cite{Hwang:2002:MomentConstruct} which was applied by Pingali \cite{Pingali:2018:heK} to define the problem of construction of higher extremal K\"ahler metrics possessing some nice symmetries along the fibres and the zero and infinity divisors of the minimal ruled surface $X$ described in Subsection \ref{subsec:KCone}. In this description we are following \cite{Pingali:2018:heK}. \par
We write an ansatz for a K\"ahler metric on $X$, which is supposed to lie in a required K\"ahler class and be higher extremal K\"ahler, in a way similar to Hwang-Singer \cite{Hwang:2002:MomentConstruct} and Sz\'ekelyhidi \cite{Szekelyhidi:2014:eKintro}. The idea is to first consider an ansatz for a metric on the total space of $\mathrm{L}$ minus its zero section and then extend this metric across the zero and infinity divisors of $X = \mathbb{P} \left(\mathrm{L} \oplus \mathcal{O}\right)$, and this is done by taking the pullback of $\mathrm{L}$ to its total space minus the zero section and then adding the curvature of the resultant bundle to the pullback of $\omega_\Sigma$ to $X$. \par
Let $\mathtt{p} : X \to \Sigma$ be the fibre bundle projection, $z$ be a coordinate on $\Sigma$, $w$ be a coordinate on the fibres of $\mathrm{L}$, $s := \ln \left| \left(z,w\right) \right|_{h}^2 = \ln \left| w \right|^2 + \ln h \left(z\right)$ be the coordinate on the total space of $\mathrm{L}$ minus the zero section, $f$ be a strictly convex smooth function of $s$ such that $s + f \left(s\right)$ is strictly increasing, and $\omega$ be a K\"ahler metric on $X$ given by the following ansatz (as in \cite{Pingali:2018:heK} and \cite{Szekelyhidi:2014:eKintro}):
\begin{equation}\label{eq:ansatz}
\omega = \mathtt{p}^* \omega_\Sigma + \sqrt{-1} \partial \bar{\partial} f \left(s\right)
\end{equation}
From the computations done in \cite{Pingali:2018:heK} and \cite{Szekelyhidi:2014:eKintro} we get the following expression for $\omega$:
\begin{equation}\label{eq:omega}
\omega = \left(1 + f' \left(s\right)\right) \mathtt{p}^* \omega_\Sigma + f'' \left(s\right) \sqrt{-1} \frac{d w \wedge d \bar{w}}{\left| w \right|^2}
\end{equation}
We want $\omega$ to be in the K\"ahler class $2 \pi \left(\mathsf{C} + m S_\infty\right)$ where $m \in \mathbb{R}_{> 0}$, and for that to happen we must have $0 \leq f' \left(s\right) \leq m$ (which we get by integrating equation (\ref{eq:omega}) over $X$ and computing the areas of $\mathsf{C}$ and $S_\infty$ using the intersection formulae (\ref{eq:IntersectForm}) and (\ref{eq:Sigma}), as done in \cite{Pingali:2018:heK} and \cite{Szekelyhidi:2014:eKintro}). \par
We compute the curvature form matrix of $\omega$ given by $\Theta \left(\omega\right) := \bar{\partial} \left( \mathrm{H}^{-1} \partial \mathrm{H} \right) \left(\omega\right)$, where $\mathrm{H} \left(\omega\right)$ is the Hermitian matrix of $\omega$, as follows:
\begin{equation}\label{eq:omega2}
\omega^2 = 2 \left(1 + f' \left(s\right)\right) f'' \left(s\right) \mathtt{p}^* \omega_\Sigma \sqrt{-1} \frac{d w \wedge d \bar{w}}{\left| w \right|^2}
\end{equation}
\begin{equation}\label{eq:Curv1}
\Theta \left(\omega\right) =
\begin{bmatrix}
- \partial \bar{\partial} \ln \left(1 + f' \left(s\right)\right) + 2 \sqrt{-1} \mathtt{p}^* \omega_\Sigma & 0 \\
0 & - \partial \bar{\partial} \ln \left(f'' \left(s\right)\right)
\end{bmatrix}
\end{equation}
This is where the conditions $1 + f' \left(s\right) > 0$ and $f'' \left(s\right) > 0$ are needed. We then use the Legendre Transform $F \left(\tau\right)$ in the variable $\tau := f' \left(s\right)$ as follows:
\begin{equation}\label{eq:LegTrans}
f \left(s\right) + F \left(\tau\right) = s \tau
\end{equation}
We then define the \textit{momentum profile} of $\omega$ as $\phi \left(\tau\right) := \frac{1}{F'' \left(\tau\right)} = f'' \left(s\right)$ (again as in \cite{Pingali:2018:heK} and \cite{Szekelyhidi:2014:eKintro}). We then write down the curvature form matrix $\sqrt{-1} \Theta \left(\omega\right)$ in terms of $\phi \left(\gamma\right)$, where $\gamma := \tau + 1 \in \left[1, m + 1\right]$, as follows:
\begin{equation}\label{eq:Curv2}
\sqrt{-1} \Theta \left(\omega\right) =
\begin{bmatrix}
\frac{\phi}{\gamma} \left( \frac{\phi}{\gamma} - \phi' \right) \sqrt{-1} \frac{d w \wedge d \bar{w}}{\left| w \right|^2} - \left(\frac{\phi}{\gamma} + 2\right) \mathtt{p}^* \omega_\Sigma & 0 \\
0 & - \phi'' \phi \sqrt{-1} \frac{d w \wedge d \bar{w}}{\left| w \right|^2} - \phi' \mathtt{p}^* \omega_\Sigma
\end{bmatrix}
\end{equation}
\par
The top Chern form of $\omega$ is given by $c_2 \left(\omega\right) := \frac{1}{\left(2 \pi\right)^2} \det \left( \sqrt{-1} \Theta \left(\omega\right) \right)$ and in terms of $\phi \left(\gamma\right)$ is given by:
\begin{equation}\label{eq:Chern}
c_2 \left(\omega\right) = \frac{1}{\left(2 \pi\right)^2} \mathtt{p}^* \omega_\Sigma \sqrt{-1} \frac{d w \wedge d \bar{w}}{\left| w \right|^2} \frac{\phi}{\gamma^2} \left( \gamma \left(\phi + 2 \gamma\right) \phi'' + \phi' \left(\phi' \gamma - \phi\right) \right)
\end{equation}
In order for $\omega$ to be a higher extremal K\"ahler metric we need to have:
\begin{equation}\label{eq:ChernheK}
c_2 \left(\omega\right) = \frac{\lambda}{2 \left(2 \pi\right)^2} \omega^2
\end{equation}
where $\nabla^{1,0} \lambda$ is a holomorphic vector field. Comparing equations (\ref{eq:omega2}), (\ref{eq:Chern}) and (\ref{eq:ChernheK}) we do the following calculation:
\begin{equation}\label{eq:gradlambda}
\nabla^{1,0} \lambda = \lambda' \nabla^{1,0} \gamma = \lambda' \nabla^{1,0} \tau = \lambda' \nabla^{\left(1,0\right)} \left(f' \left(s\right)\right) = \lambda' w \frac{\partial}{\partial w}
\end{equation}
So $\nabla^{1,0} \lambda$ is a holomorphic vector field if and only if $\lambda'$ is a constant i.e. $\lambda = A \gamma + B$ for some $A, B \in \mathbb{R}$. \par
So finally $\omega$ is a higher extremal K\"ahler metric on $X$ if and only if its momentum profile $\phi \left(\gamma\right)$, $\gamma \in \left[1, m + 1\right]$ satisfies the following ODE, for some $C \in \mathbb{R}$ (which is obtained by substituting the value of $\lambda$ in equations (\ref{eq:Chern}) and (\ref{eq:ChernheK}) and integrating w.r.t. $\gamma$, as done in \cite{Pingali:2018:heK}):
\begin{equation}\label{eq:ODE}
\left(2 \gamma + \phi\right) \phi' = A \frac{\gamma^4}{3} + B \frac{\gamma^3}{2} + C \gamma
\end{equation}
with the following boundary conditions (which are required for $\omega$ to extend smoothly to the zero and infinity divisors of $X$ as shown in \cite{Pingali:2018:heK} and \cite{Szekelyhidi:2014:eKintro}):
\begin{equation}\label{eq:BVP}
\begin{gathered}
\phi \left(1\right) = \phi \left(m + 1\right) = 0 \\
\phi' \left(1\right) = - \phi' \left(m + 1\right) = 1
\end{gathered}
\end{equation}
and with $\phi > 0$ on $\left(1, m + 1\right)$, as $\phi = f'' > 0$. \par
Now finally the problem boils down to solving the ODE (\ref{eq:ODE}) for $\phi \left(\gamma\right)$ on $\left[1, m + 1\right]$ with the boundary conditions (\ref{eq:BVP}) and with $\phi \left(\gamma\right) > 0$ on $\left(1, m + 1\right)$, for some constants $A, B, C \in \mathbb{R}$.
\subsection{Analysis of the ODE BVP for the Momentum Profile}\label{subsec:AnalysisODEBVP}
Following Pingali \cite{Pingali:2018:heK} we define the polynomial $p \left(\gamma\right) := A \frac{\gamma^3}{3} + B \frac{\gamma^2}{2} + C$ and the transformation $v := \frac{\left(2 \gamma + \phi\right)^2}{2}$, $\gamma \in \left[1, m + 1\right]$ after which the ODE BVP (\ref{eq:ODE}) and (\ref{eq:BVP}) with $\phi > 0$ on $\left(1, m + 1\right)$, obtained in Subsection \ref{subsec:MomentConstruct}, reduces to the following:
\begin{equation}\label{eq:ODEBVP1}
\begin{gathered}
v' = 2 \sqrt{2} \sqrt{v} + p \left(\gamma\right) \gamma \hspace{5pt} \text{on} \hspace{5pt} \left[1, m + 1\right] \\
v \left(1\right) = 2 \hspace{1pt}, \hspace{5pt} v \left(m + 1\right) = 2 \left(m + 1\right)^2 \\
v' \left(1\right) = 6 \hspace{1pt}, \hspace{5pt} v' \left(m + 1\right) = 2 \left(m + 1\right) \\
v \left(\gamma\right) > 2 \gamma^2 \hspace{5pt} \text{on} \hspace{5pt} \left(1, m + 1\right)
\end{gathered}
\end{equation} \par
As shown in \cite{Pingali:2018:heK}, imposing the boundary conditions on the ODE in (\ref{eq:ODEBVP1}) gives us $A$, $B$ as linear functions of $C$ as follows:
\begin{equation}\label{eq:ABC}
\begin{gathered}
A\left(C\right) = \frac{3 C}{m}\left[1 - \frac{1}{\left(m + 1\right)^2}\right] - \frac{6}{m}\left[1 + \frac{1}{\left(m + 1\right)^2}\right] \\
B\left(C\right) = - 2 C\left[1 + \frac{1}{m} - \frac{1}{m\left(m + 1\right)^2}\right] + 4\left[1 + \frac{1}{m} + \frac{1}{m\left(m + 1\right)^2}\right]
\end{gathered}
\end{equation} \par
Now we observe that if we solve the ODE in (\ref{eq:ODEBVP1}) with the boundary condition $v \left(1\right) = 2$ then equations (\ref{eq:ABC}) will imply that $v' \left(1\right) = 6$, and similarly solving the ODE in (\ref{eq:ODEBVP1}) with the boundary condition $v \left(m + 1\right) = 2 \left(m + 1\right)^2$ will give us $v' \left(m + 1\right) = 2 \left(m + 1\right)$ after applying equations (\ref{eq:ABC}). So after this the ODE BVP (\ref{eq:ODEBVP1}) reduces to the following:
\begin{equation}\label{eq:ODEBVP2}
\begin{gathered}
v' = 2 \sqrt{2} \sqrt{v} + p \left(\gamma\right) \gamma \hspace{5pt} \text{on} \hspace{5pt} \left[1, m + 1\right] \\
v \left(1\right) = 2 \hspace{1pt}, \hspace{5pt} v \left(m + 1\right) = 2 \left(m + 1\right)^2 \\
v \left(\gamma\right) > 2 \gamma^2 \hspace{5pt} \text{on} \hspace{5pt} \left(1, m + 1\right)
\end{gathered}
\end{equation}
where $p \left(\gamma\right) = A \left(C\right) \frac{\gamma^3}{3} + B \left(C\right) \frac{\gamma^2}{2} + C$ after considering the equations (\ref{eq:ABC}). \par
We then analyze the polynomials $p \left(\gamma\right)$ and $p \left(\gamma\right) \gamma$ on $\left[1, m + 1\right]$ and get the following result (proven by Pingali \cite{Pingali:2018:heK}):
\begin{lemma}[Pingali]\label{lem:polynom}
The polynomial $p \left(\gamma\right)$ satisfying $p \left(1\right) = 2$ and $p \left(m + 1\right) = -2$ has exactly one root $\gamma_0$ in $\left[1, m + 1\right]$ and has at most one critical point $\gamma' := - \frac{B}{A}$ in $\left[1, m + 1\right]$. Further the polynomial $p \left(\gamma\right) \gamma$ also has the same $\gamma_0$ as its unique root in $\left[1, m + 1\right]$ but has at most $3$ critical points in $\left[1, m + 1\right]$. In particular both $p \left(\gamma\right)$ and $p \left(\gamma\right) \gamma$ are strictly positive on $\left[1, \gamma_0\right)$ and strictly negative on $\left(\gamma_0, m + 1\right]$.
\end{lemma} \par
By using Lemma \ref{lem:polynom} we can observe that if $v$ is a smooth solution of the ODE in (\ref{eq:ODEBVP2}) on $\left[1, m + 1\right]$ satisfying both the boundary conditions viz. $v \left(1\right) = 2$ and $v \left(m + 1\right) = 2 \left(m + 1\right)^2$, then integrating the expression for $v'$ in (\ref{eq:ODEBVP2}) on $\left[1, \gamma_0\right]$ and $\left[\gamma_0, m + 1\right]$ separately and noting the sign of $p \left(\gamma\right) \gamma$ on both the subintervals will help us conclude $v \left(\gamma\right) \geq 2 \gamma^2$ on $\left[1, m + 1\right]$, and then rewriting the equation of $v'$ in (\ref{eq:ODEBVP2}) as $2 \sqrt{v} \left(\sqrt{v} - \sqrt{2} \gamma\right)' = p \left(\gamma\right) \gamma$ and using the uniqueness of $\gamma_0$ will help us conclude $v \left(\gamma\right) > 2 \gamma^2$ on $\left(1, m + 1\right)$ (see \cite{Pingali:2018:heK}). Thus the ODE BVP (\ref{eq:ODEBVP2}) further reduces to the following:
\begin{equation}\label{eq:ODEBVP3}
\begin{gathered}
v' = 2 \sqrt{2} \sqrt{v} + p \left(\gamma\right) \gamma \hspace{5pt} \text{on} \hspace{5pt} \left[1, m + 1\right] \\
v \left(1\right) = 2 \hspace{1pt}, \hspace{5pt} v \left(m + 1\right) = 2 \left(m + 1\right)^2
\end{gathered}
\end{equation} \par
Now finally our problem has been reduced to solving the ODE BVP (\ref{eq:ODEBVP3}), depending on the parameter $C \in \mathbb{R}$, for a fixed $m \in \mathbb{R}_{> 0}$ (indicating the K\"ahler class under consideration). By following the strategy of Pingali \cite{Pingali:2018:heK} we drop the final boundary condition $v \left(m + 1\right) = 2 \left(m + 1\right)^2$ in the ODE BVP (\ref{eq:ODEBVP3}) and consider the following ODE IVP:
\begin{equation}\label{eq:ODEIVP}
\begin{gathered}
v' = 2 \sqrt{2} \sqrt{v} + p \left(\gamma\right) \gamma \hspace{5pt} \text{on} \hspace{5pt} \left[1, m + 1\right] \\
v \left(1\right) = 2
\end{gathered}
\end{equation} \par
We first get a smooth solution $v \left(\cdot; C\right)$ on $\left[1, m + 1\right]$ for the ODE IVP (\ref{eq:ODEIVP}) for each $C$ in an interval of the real line. Then we find a $C_1$ in that interval such that $v \left(m + 1; C_1\right) > 2 \left(m + 1\right)^2$ and another $C_2$ in the same interval such that $v \left(m + 1; C_2\right) < 2 \left(m + 1\right)^2$, thereby proving that there exists a $C$ in the interval such that $v \left(m + 1; C\right) = 2 \left(m + 1\right)^2$. The first part was done by Pingali \cite{Pingali:2018:heK} (see the following result where $m$ is fixed) and this paper's main goal is to do the second part which requires some deeper analysis of the variation of the ODE IVP (\ref{eq:ODEIVP}) w.r.t. the parameter $C$. Note that this will be done for an arbitrary $m > 0$.
\begin{theorem}[Pingali]\label{thm:VPP}
There exists an $M > 2$ (depending only on $m$) such that for any $C < M$ there exists a unique smooth solution $v \left(\cdot; C\right)$ on $\left[1, m + 1\right]$ for the ODE IVP (\ref{eq:ODEIVP}). Further there exists a $C_1 < M$ such that $v \left(m + 1; C_1\right) > 2 \left(m + 1\right)^2$.
\end{theorem} \par
We will prove the following main result in this paper (where $m$ is fixed):
\begin{theorem}\label{thm:main}
There exists a unique $M = M \left(m\right) > 2$ such that for every $C \in \left(-\infty, M\right)$ there exists a unique smooth solution $v\left(\cdot; C\right) : \left[1, m + 1\right] \to \mathbb{R}$ for the ODE IVP (\ref{eq:ODEIVP}) such that $\lim\limits_{C \to -\infty} v\left(m + 1; C\right) = \infty$ and $\lim\limits_{C \to M^-} v\left(m + 1; C\right) = 0$, and for every $C \geq M$ there exists a unique smooth solution $v\left(\cdot; C\right) : \left[1, \gamma_{\star, C}\right) \to \mathbb{R}$ for (\ref{eq:ODEIVP}), for a unique $\gamma_{\star, C} \in \left(1, m + 1\right]$, such that $v\left(\cdot; C\right)$ cannot be continued smoothly beyond $\gamma_{\star, C}$. Further there exists a unique $C = C \left(m\right) \in \left(-\infty, M\right)$ such that $v \left(m + 1; C\right) = 2 \left(m + 1\right)^2$, and further the $C$ with this property has to be strictly greater than $2$.
\end{theorem} \par
We will then have answered the question of the existence of higher extremal K\"ahler metrics in the K\"ahler classes of the form $2 \pi \left(\mathsf{C} + m S_\infty\right)$ on $X$ affirmatively:
\begin{corollary}\label{cor:main}
For each $m > 0$ there exists a higher extremal K\"ahler metric $\omega$ on $X$ satisfying the following:
\begin{equation}\label{eq:problem1}
\left[\omega\right] = 2 \pi \left(\mathsf{C} + m S_\infty\right) \hspace{1pt}, \hspace{5pt} c_2 \left(\omega\right) = \frac{\lambda}{2 \left(2 \pi\right)^2} \omega^2 \hspace{1pt}, \hspace{5pt} \nabla^{1,0} \lambda \in \mathfrak{h} \left(X\right)
\end{equation}
\end{corollary} \par
\begin{remark}
Note that Pingali \cite{Pingali:2018:heK} had proven for $m = 1$ that there exists a $C_2 < M$ such that $v \left(m + 1; C_2\right) < 2 \left(m + 1\right)^2$, thus proving the existence of a $C < M$ satisfying $v \left(m + 1; C\right) = 2 \left(m + 1\right)^2$, and hence proving the existence of the required higher extremal K\"ahler metric in the K\"ahler class $2 \pi \left(\mathsf{C} + S_\infty\right)$. But this was done by using some explicit numerical bounds on $p \left(\gamma\right) \gamma$ and $v'$ specifically applicable for $m = 1$, and the author was not able to generalize those arguments and estimates for a general $m > 0$. So we will adopt a completely different strategy of proof for Theorem \ref{thm:main}.
\end{remark} \par
Now we observe that on a compact K\"ahler $n$-manifold $\mathrm{M}$ if we rescale a K\"ahler metric $\omega$ by a factor of $k > 0$ then the K\"ahler class and the Hermitian matrix get multiplied by $k$ whereas the connection and hence the curvature remain unchanged. So the curvature form matrix and hence the top Chern form (as defined in Subsection \ref{subsec:MomentConstruct}) of $k \omega$ are exactly the same as that of $\omega$. Thus if $\omega$ satisfying equation (\ref{eq:defheK}) is a higher extremal K\"ahler metric on $\mathrm{M}$ then so is $k \omega$ as seen by the following:
\begin{equation}\label{eq:rescalheK}
c_n \left(k \omega\right) = c_n \left(\omega\right) = \lambda \omega^n = \frac{\lambda}{k^n} \left(k \omega\right)^n
\end{equation} \par
So on our surface $X$ if $\omega \in 2 \pi \left(\mathsf{C} + m S_\infty\right)$ is a higher extremal K\"ahler metric satisfying the equations (\ref{eq:problem1}) then for any $k > 0$ the rescaled K\"ahler metric $\eta := \frac{k}{2 \pi} \omega \in k \left(\mathsf{C} + m S_\infty\right)$ is also higher extremal K\"ahler and satisfies the equation $c_2 \left(\eta\right) = \frac{\lambda}{2 k^2} \eta^2$. \par
This allows us to generalize Corollary \ref{cor:main} to the general K\"ahler classes on $X$ which are of the form $a \mathsf{C} + b S_\infty$ where $a, b > 0$:
\begin{corollary}\label{cor:rescalheK}
For all $a, b > 0$ there exists a higher extremal K\"ahler metric $\eta$ on $X$ satisfying the following:
\begin{equation}\label{eq:rescalprob}
\left[\eta\right] = a \mathsf{C} + b S_\infty \hspace{1pt}, \hspace{5pt} c_2 \left(\eta\right) = \frac{\lambda}{2 a^2} \eta^2 \hspace{1pt}, \hspace{5pt} \nabla^{1,0} \lambda \in \mathfrak{h} \left(X\right)
\end{equation}
\end{corollary}
\numberwithin{equation}{subsection}
\numberwithin{definition}{subsection}
\numberwithin{lemma}{subsection}
\numberwithin{theorem}{subsection}
\numberwithin{corollary}{subsection}
\numberwithin{remark}{subsection}
\numberwithin{case}{subsection}
\numberwithin{claim}{subsection}
\numberwithin{motivation}{subsection}
\numberwithin{question}{subsection}
\section{Proof of the Main Result}\label{sec:Proof}
\subsection{First Part of the Proof}\label{subsec:Proof1}
Throughout Subsections \ref{subsec:Proof1} and \ref{subsec:Proof2}, $m > 0$ is fixed. The final goal over here is to prove Theorem \ref{thm:main}. \par
\begin{motivation}
We will first prove that for every $C \in \mathbb{R}$ there exists a unique $\mathcal{C}^1$ solution $v$ to the ODE IVP (\ref{eq:ODEIVP}) on a non-degenerate interval containing $1$, and in fact this $v$ exists and is strictly increasing on $\left[1, \gamma_0\right]$ where $\gamma_0$ is the unique root of $p \left(\gamma\right) \gamma$ in $\left[1, m + 1\right]$. We will then prove that the $\mathcal{C}^1$ solution $v$ defined on any interval is always strictly positive on the interval, and as a consequence is smooth (i.e. $\mathcal{C}^\infty$) on the interval. We will finally prove a necessary and sufficient condition for the continuation of the solution $v$ defined a priori on $\left[1, \tilde{r}\right)$ for a given $\tilde{r} \in \left(1, m + 1\right]$.
\end{motivation} \par
As was noted by Pingali \cite{Pingali:2018:heK} if $v$ is a $\mathcal{C}^1$ solution to (\ref{eq:ODEIVP}) on any interval then substituting $\sqrt{v} < v + 1$ and $\left| p \left(\gamma\right) \gamma \right| \leq l$ (for some $l > 0$) in the expression for $v' = \left(v + 1\right)'$ in (\ref{eq:ODEIVP}) and applying Gr\"onwall's inequality will give us a $K > 0$ such that $v \left(\gamma\right) \leq K$ on the interval. \par
Thus solutions to (\ref{eq:ODEIVP}) are always bounded above (and always bounded below by $0$) on any interval on which they exist. So by standard ODE Theory the existence of a strictly positive lower bound on a solution of (\ref{eq:ODEIVP}) is a sufficient condition for the continuation of the solution beyond its prior interval of definition (Pingali \cite{Pingali:2018:heK}).
\begin{lemma}[Pingali]\label{lem:continue}
For a given $C \in \mathbb{R}$ let $v$ be a $\mathcal{C}^1$ solution to the ODE IVP (\ref{eq:ODEIVP}) existing on $\left[1, \tilde{r}\right) \subseteq \left[1, m + 1\right]$. If there exists an $\epsilon > 0$ such that $v \left(\gamma\right) \geq \epsilon$ on $\left[1, \tilde{r}\right)$ then $v$ can be continued beyond $\tilde{r}$.
\end{lemma} \par
We will prove the converse of Lemma \ref{lem:continue} (viz. Theorem \ref{thm:continue}), but before that we prove some more basic results:
\begin{lemma}\label{lem:firstexist}
For every $C \in \mathbb{R}$ there exists a unique $\mathcal{C}^1$ solution $v$ to the ODE IVP (\ref{eq:ODEIVP}) on $\left[1, r\right)$ for some $r \in \left(1, m + 1\right]$ such that $v' > 0$ on $\left[1, r\right)$. If $\left[1, r'\right) \subseteq \left[1, m + 1\right]$ is the maximal interval of existence of $v$ then $\left[1, \gamma_0\right] \subseteq \left[1, r'\right)$ and $v' > 0$ on $\left[1, \gamma_0\right]$. Similarly if $\left[1, m + 1\right]$ is the maximal interval of existence of $v$ then $\left[1, \gamma_0\right] \subseteq \left[1, m + 1\right]$ and $v' > 0$ on $\left[1, \gamma_0\right]$.
\end{lemma}
\begin{proof}
As noted in Subsection \ref{subsec:AnalysisODEBVP}, $v \left(1\right) = 2$ will imply $v' \left(1\right) = 6$ in (\ref{eq:ODEIVP}). Since $\sqrt{v \left(1\right)} = \sqrt{2} > 0$ so the RHS of (\ref{eq:ODEIVP}) is continuous in $\gamma$ and locally Lipschitz in $v$ in a neighbourhood of $1$ and by standard ODE Theory there exists a unique $\mathcal{C}^1$ solution $v$ to (\ref{eq:ODEIVP}) on $\left[1, r\right)$ for some $r \in \left(1, m + 1\right]$. Since $v' \left(1\right) = 6 > 0$ so this $r \in \left(1, m + 1\right]$ can be chosen so that $v' > 0$ on $\left[1, r\right)$. \\
Let $\left[1, r'\right) \subseteq \left[1, m + 1\right]$ be the maximal interval of existence of $v$. If $\gamma_0 \geq r'$ then by Lemma \ref{lem:polynom}, $p \left(\gamma\right) \gamma \geq 0$ on $\left[1, r'\right)$ and hence $v' \geq 0$ on $\left[1, r'\right)$. So $v \left(\gamma\right) \geq v \left(1\right) = 2 > 0$ on $\left[1, r'\right)$ and by Lemma \ref{lem:continue}, $v$ can be continued beyond $r'$ contradicting the maximality of $r'$. So $\left[1, \gamma_0\right] \subseteq \left[1, r'\right)$ and $v' \geq 0$ on $\left[1, \gamma_0\right]$, but as $\sqrt{v \left(\gamma\right)} > 0$ on $\left[1, \gamma_0\right]$ so $v' > 0$ on $\left[1, \gamma_0\right]$. \\
If $\left[1, m + 1\right]$ is the maximal interval of existence of $v$ then by Lemma \ref{lem:polynom}, $\left[1, \gamma_0\right] \subseteq \left[1, m + 1\right]$ and by the same arguments as in the above case, $v' > 0$ on $\left[1, \gamma_0\right]$.
\end{proof} \par
\begin{remark}
Note that $v' > 0$ on $\left[1, \gamma_0\right]$ actually implies $v' > 0$ on $\left[1, \gamma_0'\right) \subseteq \left[1, r'\right)$ (or $\left[1, \gamma_0'\right) \subseteq \left[1, m + 1\right]$) for some $\gamma_0' > \gamma_0$.
\end{remark} \par
Observe that if there exists a $\mathcal{C}^1$ solution $v$ to (\ref{eq:ODEIVP}) on any interval then $v \geq 0$ on the interval, but Lemma \ref{lem:positivity} is saying that $v > 0$ on the interval.
\begin{lemma}[Positivity of Solutions]\label{lem:positivity}
For a given $C \in \mathbb{R}$ let $v$ be the unique $\mathcal{C}^1$ solution to the ODE IVP (\ref{eq:ODEIVP}) existing on some interval containing $1$.
\begin{enumerate}
\item If $\left[1, r'\right) \subseteq \left[1, m + 1\right]$ is the maximal interval of existence of $v$ then $v \left(\gamma\right) > 0$ for all $\gamma \in \left[1, r'\right)$ and $\lim\limits_{\gamma \to r'} v \left(\gamma\right) = 0$ and $\lim\limits_{\gamma \to r'} v' \left(\gamma\right) < 0$. \label{itm:posit1}
\item If $\left[1, m + 1\right]$ is the maximal interval of existence of $v$ then $v \left(\gamma\right) > 0$ for all $\gamma \in \left[1, m + 1\right]$. \label{itm:posit2}
\end{enumerate}
\end{lemma}
\begin{proof}
\begin{case*}(\ref{itm:posit1}) \hspace{2pt}
If $\left[1, r'\right)$ is the maximal interval of existence of $v$ then by Lemma \ref{lem:firstexist}, $\gamma_0 \in \left[1, r'\right)$ and $v' \left(\gamma_0\right) > 0$ and also $v \left(\gamma_0\right) \geq v \left(1\right) = 2 > 0$. Since $v$ cannot be continued beyond $r'$ so by Lemma \ref{lem:continue}, $\inf\limits_{\gamma \in \left[1, r'\right)} v \left(\gamma\right) = 0$. If $t_0 \in \left(1, r'\right)$ is such that $v \left(t_0\right) = 0$ then $t_0$ must be a point of local minimum of $v$ and so $v' \left(t_0\right) = 0$ which will imply $p \left(t_0\right) t_0 = 0$ and then $t_0 = \gamma_0$ (by the uniqueness of $\gamma_0$) which contradicts the first assertion above. So $v \left(\gamma\right) > 0$ for all $\gamma \in \left[1, r'\right)$. \\
Now since $v$ and $p \left(\gamma\right) \gamma$ are bounded on $\left[1, r'\right)$, from the expression of $v'$ in (\ref{eq:ODEIVP}) we get $v'$ is bounded, thereby implying $v$ is Lipschitz on $\left[1, r'\right)$. So $\lim\limits_{\gamma \to r'} v \left(\gamma\right)$ exists and as $v > 0$ on $\left[1, r'\right)$ so $\lim\limits_{\gamma \to r'} v \left(\gamma\right) = \inf\limits_{\gamma \in \left[1, r'\right)} v \left(\gamma\right) = 0$. \\
Now as $\sqrt{\cdot}$ on $\left[0, \infty\right)$, $v$ on $\left[1, r'\right)$ and $p \left(\gamma\right) \gamma$ on $\left[1, m + 1\right]$ are uniformly continuous so from (\ref{eq:ODEIVP}), $v'$ is uniformly continuous on $\left[1, r'\right)$ and so $\lim\limits_{\gamma \to r'} v' \left(\gamma\right)$ exists. Since $\lim\limits_{\gamma \to r'} v \left(\gamma\right) = 0$ and $v > 0$ on $\left[1, r'\right)$ so $\lim\limits_{\gamma \to r'} v' \left(\gamma\right) \leq 0$. If $\lim\limits_{\gamma \to r'} v' \left(\gamma\right) = 0$ then from (\ref{eq:ODEIVP}) we will get $r'$ is a root of $p \left(\gamma\right) \gamma$ which is not possible by Lemma \ref{lem:firstexist}. So $\lim\limits_{\gamma \to r'} v' \left(\gamma\right) < 0$.
\end{case*}
\begin{case*}(\ref{itm:posit2}) \hspace{2pt}
If $\left[1, m + 1\right]$ is the maximal interval of existence of $v$ then $v \geq 0$ on $\left[1, m + 1\right]$ and $v \left(1\right) = 2 > 0$ and from Lemma \ref{lem:firstexist}, $v' \left(\gamma_0\right) > 0$ and $v \left(\gamma_0\right) > 0$. So by the same argument as in Case (\ref{itm:posit1}), there cannot exist a $t_0 \in \left(1, m + 1\right)$ such that $v \left(t_0\right) = 0$. So $v \left(\gamma\right) > 0$ for all $\gamma \in \left[1, m + 1\right)$. \\
Let if possible $v \left(m + 1\right) = 0$. Since $v \geq 0$ and $v$ is $\mathcal{C}^1$ on $\left[1, m + 1\right]$ so $v' \left(m + 1\right) \leq 0$. If $v' \left(m + 1\right) = 0$ then from (\ref{eq:ODEIVP}) we will get $m + 1$ is a root of $p \left(\gamma\right) \gamma$ which is not possible by Lemma \ref{lem:polynom}. So $v' \left(m + 1\right) < 0$ i.e. $v$ is strictly decreasing in a neighbourhood of $m + 1$. Now $v$ is the $\mathcal{C}^1$ solution of (\ref{eq:ODEIVP}) and $v > 0$ on $\left[1, m + 1\right)$, and for $v$ to be extendable as the $\mathcal{C}^1$ solution to (\ref{eq:ODEIVP}) on an interval strictly containing $\left[1, m + 1\right)$ we must have $v \geq 0$ on the larger interval which will not be possible with $v' \left(m + 1\right) < 0$. So in that case, $v$ will exist as the $\mathcal{C}^1$ solution of (\ref{eq:ODEIVP}) maximally on $\left[1, m + 1\right)$, a contradiction to the hypothesis. So $v \left(m + 1\right) > 0$ and hence $v \left(\gamma\right) > 0$ for all $\gamma \in \left[1, m + 1\right]$.
\end{case*}
\end{proof} \par
Before proceeding further let us observe the following:
\begin{corollary}[Smoothness of Solutions]\label{cor:InfinitDiff}
Let $v$ be the $\mathcal{C}^1$ solution of the ODE IVP (\ref{eq:ODEIVP}) on a non-degenerate interval $\mathsf{J}$. Then $v'$ is bounded and uniformly continuous on $\mathsf{J}$, and $v^{\left(i\right)}$ exists on $\mathsf{J}$ for all $i \in \mathbb{N}_{\geq 2}$ i.e. $v$ is $\mathcal{C}^\infty$ on $\mathsf{J}$. For all $i \in \mathbb{N}_{\geq 2}$ if $\inf\limits_{\mathsf{J}} v > 0$ then $v^{\left(i\right)}$ is bounded on $\mathsf{J}$ and if $\inf\limits_{\mathsf{J}} v = 0$ then $v^{\left(i\right)}$ is unbounded on $\mathsf{J}$.
\end{corollary}
\begin{proof}
By Lemma \ref{lem:positivity}, $v > 0$ on $\mathsf{J}$ and so $\frac{1}{\sqrt{v}}$ makes sense. Considering the ODE in (\ref{eq:ODEIVP}) and its $\left(i - 1\right)$\textsuperscript{th} derivative, substituting the implied property of $\sqrt{\cdot}$, $v$ and $p \left(\gamma\right) \gamma$ in both the cases into these ODEs and using induction will give us the required results for all $i \in \mathbb{N}_{\geq 2}$.
\end{proof} \par
Observe one more thing that the solution $v$ cannot be constant on any non-degenerate interval $\mathsf{J}$, as that would imply (again from (\ref{eq:ODEIVP})) that the polynomial $p \left(\gamma\right) \gamma$ is a constant polynomial which is not possible by Lemma \ref{lem:polynom}. \par
We now have the following necessary and sufficient condition for the continuation of the solution to (\ref{eq:ODEIVP}) defined a priori on some interval:
\begin{theorem}[Criterion for Continuation of Solutions]\label{thm:continue}
For any $C \in \mathbb{R}$ if $v$ is the smooth solution to (\ref{eq:ODEIVP}) defined on an interval $\left[1, \tilde{r}\right) \subseteq \left[1, m + 1\right]$ then:
\begin{enumerate}
\item $v$ can be continued beyond $\tilde{r}$ if and only if $\inf\limits_{\gamma \in \left[1, \tilde{r}\right)} v \left(\gamma\right) = \epsilon_1 > 0$ if and only if $\lim\limits_{\gamma \to \tilde{r}} v \left(\gamma\right) = \epsilon_2 > 0$. \label{itm:cont1}
\item $\left[1, \tilde{r}\right)$ is the maximal interval of existence of $v$ if and only if $\inf\limits_{\gamma \in \left[1, \tilde{r}\right)} v \left(\gamma\right) = 0$ if and only if $\lim\limits_{\gamma \to \tilde{r}} v \left(\gamma\right) = 0$. \label{itm:cont2}
\end{enumerate}
\end{theorem}
\begin{proof}
From Lemma \ref{lem:positivity} and Corollary \ref{cor:InfinitDiff}, $v > 0$ and $v$ is Lipschitz on $\left[1, \tilde{r}\right)$ so $\inf\limits_{\gamma \in \left[1, \tilde{r}\right)} v \left(\gamma\right)$ and $\lim\limits_{\gamma \to \tilde{r}} v \left(\gamma\right)$ both exist and are non-negative. So again by using Lemma \ref{lem:positivity} and Corollary \ref{cor:InfinitDiff} it can be easily checked that either both $\inf\limits_{\gamma \in \left[1, \tilde{r}\right)} v \left(\gamma\right)$ and $\lim\limits_{\gamma \to \tilde{r}} v \left(\gamma\right)$ are simultaneously strictly positive or both are simultaneously zero. From Lemma \ref{lem:continue} we already have that if both are simultaneously positive then $v$ can be continued beyond $\tilde{r}$. For proving the converse let $\inf\limits_{\gamma \in \left[1, \tilde{r}\right)} v \left(\gamma\right) = \lim\limits_{\gamma \to \tilde{r}} v \left(\gamma\right) = 0$. Let if possible $v$ be extendable as the smooth solution of (\ref{eq:ODEIVP}) to an interval $\left[1, r'\right)$ with $\left[1, \tilde{r}\right] \subseteq \left[1, r'\right) \subseteq \left[1, m + 1\right]$. Then by Lemma \ref{lem:positivity}, $v > 0$ on $\left[1, r'\right)$ and as $\tilde{r} \in \left[1, r'\right)$ and $v$ is continuous on $\left[1, r'\right)$ so $0 = \lim\limits_{\gamma \to \tilde{r}} v \left(\gamma\right) = v \left(\tilde{r}\right) > 0$, a contradiction. So $\left[1, \tilde{r}\right)$ is the maximal interval of existence of $v$.
\end{proof} \par
So if the smooth solution $v$ to (\ref{eq:ODEIVP}) cannot be defined on $\left[1, m + 1\right]$ then there exists a unique $\gamma_\star \in \left(1, m + 1\right]$ such that $\left[1, \gamma_\star\right)$ is the maximal interval of existence of $v$. \par
\begin{remark}
Lemma \ref{lem:positivity} and Theorem \ref{thm:continue} are together saying that the solution of (\ref{eq:ODEIVP}) continues to exist as long as it is strictly positive, but the moment it attains zero, it `breaks down' i.e. it cannot be continued further.
\end{remark} \par
So finally for every $C \in \mathbb{R}$ considering the ODE IVP (\ref{eq:ODEIVP}) depending on $C$ we have exactly one of the following two scenarios (as a consequence of Lemma \ref{lem:firstexist} and Theorem \ref{thm:continue}):
\begin{enumerate}
\item There exists a unique smooth solution $v_C := v \left(\cdot; C\right)$ on $\left[1, m + 1\right]$. \label{itm:case1}
\item There exists a unique smooth solution $v_C := v \left(\cdot; C\right)$ with maximal interval of existence $\left[1, \gamma_{\star, C}\right)$ for a unique $\gamma_{\star, C} := \gamma_\star \left(C\right) \in \left(1, m + 1\right]$. \label{itm:case2}
\end{enumerate} \par
\begin{motivation}
In order to prove Theorem \ref{thm:main} we will first show that the set of all $C \in \mathbb{R}$, for which the condition (\ref{itm:case1}) above holds true, is precisely the interval $\left(-\infty, M\right)$ for a unique $M > 2$, and then we will check the limits of $v_C \left(m + 1\right)$ as $C \to -\infty$ and as $C \to M$ respectively to conclude that the range of the function $\left(-\infty, M\right) \to \mathbb{R}$, $C \mapsto v_C \left(m + 1\right)$ is precisely the interval $\left(0, \infty\right)$. For this we will prove some preparatory results in the remainder of Subsection \ref{subsec:Proof1} and in Subsection \ref{subsec:Proof2}.
\end{motivation} \par
Let $\left( \mathcal{C} \left[1, m + 1\right], \lVert \cdot \rVert_\infty \right)$ be the Banach space of all continuous functions on $\left[1, m + 1\right]$. For each $C \in \mathbb{R}$ define $u \left(\cdot; C\right) : \left[1, m + 1\right] \to \mathbb{R}$ as follows:
\begin{enumerate}
\item If the smooth solution $v_C$ to (\ref{eq:ODEIVP}) exists on $\left[1, m + 1\right]$ then $u \left(\gamma; C\right) := v_C \left(\gamma\right)$ for all $\gamma \in \left[1, m + 1\right]$. \label{itm:defu1}
\item If the smooth solution $v_C$ to (\ref{eq:ODEIVP}) has maximal interval of existence $\left[1, \gamma_{\star, C}\right)$ then $u \left(\gamma; C\right) := v_C \left(\gamma\right)$ for all $\gamma \in \left[1, \gamma_{\star, C}\right)$ and $u \left(\gamma; C\right) := 0$ for all $\gamma \in \left[\gamma_{\star, C}, m + 1\right]$. \label{itm:defu2}
\end{enumerate}
By Lemma \ref{lem:positivity}, $u \left(\cdot; C\right)$ is continuous on $\left[1, m + 1\right]$ in the Case (\ref{itm:defu2}) above as well, and hence $u \left(\cdot; C\right) \in \mathcal{C} \left[1, m + 1\right]$ in both the Cases (\ref{itm:defu1}) and (\ref{itm:defu2}) above. Thus we get a function $\Phi : \mathbb{R} \to \mathcal{C} \left[1, m + 1\right]$ defined as $\Phi \left(C\right) := u \left(\cdot; C\right)$ for all $C \in \mathbb{R}$. It can be readily checked from (\ref{eq:ODEIVP}) that $\Phi$ is well-defined and injective. \par
\begin{motivation}
We want to prove that $\Phi$ is continuous and considering the pointwise partial order $\leq$ on $\mathcal{C} \left[1, m + 1\right]$, $\Phi$ is monotone decreasing.
\end{motivation} \par
For a given $C \in \mathbb{R}$ let $\gamma_{0, C} := \gamma_0 \left(C\right)$ be the unique root of the polynomial $p_C \left(\gamma\right) \gamma := p \left(\gamma; C\right) \gamma$ in $\left[1, m + 1\right]$ and similarly let $u_C := u \left(\cdot; C\right)$ on $\left[1, m + 1\right]$. \par
\begin{theorem}\label{thm:UnifConv}
Let $\left(C_\mathrm{n}\right) \to C_0$ and $u_\mathrm{n} := u_{C_\mathrm{n}}$ and $u_0 := u_{C_0}$. Then there exists a subsequence $\left(u_{\mathrm{n}_\mathrm{k}}\right)$ of $\left(u_\mathrm{n}\right)$ such that $\left(u_{\mathrm{n}_\mathrm{k}}\right) \to u_0$ uniformly on $\left[1, m + 1\right]$.
\end{theorem}
\begin{proof}
Let $v_\mathrm{n} := v_{C_\mathrm{n}}$ and $v_0 := v_{C_0}$, and $p_\mathrm{n} \left(\gamma\right) \gamma := p_{C_\mathrm{n}} \left(\gamma\right) \gamma$ and $p_0 \left(\gamma\right) \gamma := p_{C_0} \left(\gamma\right) \gamma$. Then from the expressions (\ref{eq:ABC}), $\left(A \left(C_\mathrm{n}\right)\right) \to A \left(C_0\right)$ and $\left(B \left(C_\mathrm{n}\right)\right) \to B \left(C_0\right)$, and hence $\left( p_\mathrm{n} \left(\gamma\right) \gamma \right) \to p_0 \left(\gamma\right) \gamma$ uniformly on $\left[1, m + 1\right]$. Let $\gamma_{0, \mathrm{n}} := \gamma_{0, C_\mathrm{n}}$ and $\gamma_{0, 0} := \gamma_{0, C_0}$ be the roots of $p_\mathrm{n} \left(\gamma\right) \gamma$ and $p_0 \left(\gamma\right) \gamma$ in $\left[1, m + 1\right]$ respectively. Then we can verify that $\gamma_{0, \mathrm{n}} \to \gamma_{0, 0}$. Also note that $u_\mathrm{n} \left(1\right) = v_\mathrm{n} \left(1\right) = 2 = v_0 \left(1\right) = u_0 \left(1\right)$ and $u_\mathrm{n}' \left(1\right) = v_\mathrm{n}' \left(1\right) = 6 = v_0' \left(1\right) = u_0' \left(1\right)$. After this the proof of Theorem \ref{thm:UnifConv} will be divided into the following $3$ Cases:
\begin{case*}($1$) \hspace{2pt}
The solution $v_\mathrm{n}$ exists on $\left[1, m + 1\right]$ for all $\mathrm{n} \in \mathbb{N}$ and $\inf\limits_{\gamma \in \left[1, m + 1\right], \mathrm{n} \in \mathbb{N}} v_\mathrm{n} \left(\gamma\right) = \epsilon > 0$.
\end{case*}
{\noindent Here $u_\mathrm{n} = v_\mathrm{n}$ for all $\mathrm{n} \in \mathbb{N}$. Since $\left( p_\mathrm{n} \left(\gamma\right) \gamma \right)$ is uniformly norm bounded, substituting $\left| p_\mathrm{n} \left(\gamma\right) \gamma \right| \leq l$ (for some $l > 0$) and $\sqrt{v_\mathrm{n}} < v_\mathrm{n} + 1$ for all $\mathrm{n} \in \mathbb{N}$ in the expression for $v_\mathrm{n}' = \left(v_\mathrm{n} + 1\right)'$ in (\ref{eq:ODEIVP}) and using Gr\"onwall's inequality will yield a $K > 0$ such that $\epsilon \leq v_\mathrm{n} \left(\gamma\right) \leq K$ for all $\gamma \in \left[1, m + 1\right]$ and for all $\mathrm{n} \in \mathbb{N}$, thereby proving that $\left(v_\mathrm{n}\right)$ is uniformly norm bounded on $\left[1, m + 1\right]$. Now substituting $v_\mathrm{n} \leq K$ and $\left| p_\mathrm{n} \left(\gamma\right) \gamma \right| \leq l$ in the expression for $v_\mathrm{n}'$ in (\ref{eq:ODEIVP}) will give us that $\left(v_\mathrm{n}'\right)$ is also uniformly norm bounded, implying that $\left(v_\mathrm{n}\right)$ is uniformly equicontinuous on $\left[1, m + 1\right]$. So by Arzel\`a-Ascoli Theorem there exists a subsequence $\left(v_{\mathrm{n}_\mathrm{k}}\right)$ of $\left(v_\mathrm{n}\right)$ converging uniformly on $\left[1, m + 1\right]$ to some $w \in \mathcal{C} \left[1, m + 1\right]$. As $\left( p_\mathrm{n} \left(\gamma\right) \gamma \right)$ and $\left(\sqrt{v_{\mathrm{n}_\mathrm{k}}}\right)$ are uniformly convergent so $\left(v_{\mathrm{n}_\mathrm{k}}'\right)$ is uniformly convergent on $\left[1, m + 1\right]$ (again from (\ref{eq:ODEIVP})). Then by standard Uniform Convergence Theory $w$ is differentiable and satisfies the ODE IVP (\ref{eq:ODEIVP}) for $C = C_0$ on $\left[1, m + 1\right)$. As $\inf\limits_{\gamma \in \left[1, m + 1\right], \mathrm{n} \in \mathbb{N}} v_\mathrm{n} \left(\gamma\right) = \epsilon > 0$ so $\inf\limits_{\gamma \in \left[1, m + 1\right)} w \left(\gamma\right) = \tilde{\epsilon} > 0$ and so by Theorem \ref{thm:continue}, $w$ is differentiable and satisfies (\ref{eq:ODEIVP}) on $\left[1, m + 1\right]$, and hence the solution $v_0$ exists on $\left[1, m + 1\right]$ and by uniqueness, $w = v_0$ and by definition, $u_0 = v_0$ as well. Thus $\left(u_{\mathrm{n}_\mathrm{k}}\right) \to u_0$ uniformly on $\left[1, m + 1\right]$.}
\begin{case*}($2$) \hspace{2pt}
The solution $v_\mathrm{n}$ exists on $\left[1, m + 1\right]$ for all $\mathrm{n} \in \mathbb{N}$ and $\inf\limits_{\gamma \in \left[1, m + 1\right], \mathrm{n} \in \mathbb{N}} v_\mathrm{n} \left(\gamma\right) = 0$.
\end{case*}
{\noindent Here also $u_\mathrm{n} = v_\mathrm{n}$ for all $\mathrm{n} \in \mathbb{N}$. The Gr\"onwall's inequality argument as in Case ($1$) above will prove $0 \leq v_\mathrm{n} \left(\gamma\right) \leq K$ i.e. $\left(v_\mathrm{n}\right)$ is uniformly norm bounded on $\left[1, m + 1\right]$. Once again by substituting this in (\ref{eq:ODEIVP}) with $C = C_\mathrm{n}$, $\left(v_\mathrm{n}'\right)$ will be uniformly norm bounded on $\left[1, m + 1\right]$ and so there will exist a subsequence $\left(v_{\mathrm{n}_\mathrm{k}}\right) \to w \in \mathcal{C} \left[1, m + 1\right]$ uniformly on $\left[1, m + 1\right]$. By the same arguments as in Case ($1$), $\left(v_{\mathrm{n}_\mathrm{k}}'\right)$ is uniformly convergent on $\left[1, m + 1\right]$ and $w$ is differentiable and satisfies the ODE IVP (\ref{eq:ODEIVP}) for $C = C_0$ on $\left[1, m + 1\right)$.
\begin{claim*}
There exists an $\epsilon_0 > 0$ such that for any sequence $\left(t_\mathrm{n}\right)$ where $t_\mathrm{n} \in \left(1, m + 1\right)$ is a local minimum of $v_\mathrm{n}$ we have $v_\mathrm{n} \left(t_\mathrm{n}\right) \geq \epsilon_0$.
\end{claim*}
{\noindent If the above Claim is false then there exists a sequence $\left(t_\mathrm{n}\right)$ of respective local minima of $v_\mathrm{n}$ such that $\lim\limits_{\mathrm{n} \to \infty} v_\mathrm{n} \left(t_\mathrm{n}\right) = 0$. Passing to a subsequence if necessary, let $\left(t_\mathrm{n}\right) \to t_0 \in \left[1, m + 1\right]$. Note also that $v_\mathrm{n}' \left(t_\mathrm{n}\right) = 0$. By the uniform norm boundedness of $\left(v_\mathrm{n}'\right)$ let $R > 0$ be the uniform Lipschitz constant for $\left(v_\mathrm{n}\right)$. For any $\mathrm{n}, \mathrm{j} \in \mathbb{N}$ considering the following estimates:
\begin{align}\label{eq:Case2}
\lvert v_\mathrm{n}\left(t_0\right) \rvert &\leq \lvert v_\mathrm{n}\left(t_0\right) - v_\mathrm{n}\left(t_\mathrm{j}\right) \rvert + \lvert v_\mathrm{n}\left(t_\mathrm{j}\right) - v_\mathrm{n}\left(t_\mathrm{n}\right) \rvert + \lvert v_\mathrm{n}\left(t_\mathrm{n}\right) \rvert \nonumber \\
&\leq \lvert v_\mathrm{n}\left(t_0\right) - v_\mathrm{n}\left(t_\mathrm{j}\right) \rvert + R \lvert t_\mathrm{j} - t_\mathrm{n} \rvert + \lvert v_\mathrm{n}\left(t_\mathrm{n}\right) \rvert
\end{align}
we get $\lim\limits_{\mathrm{n} \to \infty} v_\mathrm{n} \left(t_0\right) = 0$. Now as both $\left(\sqrt{v_\mathrm{n}}\right)$ and $\left( p_\mathrm{n} \left(\gamma\right) \gamma \right)$ are uniformly equicontinuous so $\left(v_\mathrm{n}'\right)$ is also uniformly equicontinuous on $\left[1, m + 1\right]$ (from (\ref{eq:ODEIVP})). On similar lines as the estimates (\ref{eq:Case2}), considering the following estimates:
\begin{align}\label{eq:Case2'}
\lvert v_\mathrm{n}'\left(t_0\right) \rvert &\leq \lvert v_\mathrm{n}'\left(t_0\right) - v_\mathrm{n}'\left(t_\mathrm{j}\right) \rvert + \lvert v_\mathrm{n}'\left(t_\mathrm{j}\right) - v_\mathrm{n}'\left(t_\mathrm{n}\right) \rvert + \lvert v_\mathrm{n}'\left(t_\mathrm{n}\right) \rvert \nonumber \\
&= \lvert v_\mathrm{n}'\left(t_0\right) - v_\mathrm{n}'\left(t_\mathrm{j}\right) \rvert + \lvert v_\mathrm{n}'\left(t_\mathrm{j}\right) - v_\mathrm{n}'\left(t_\mathrm{n}\right) \rvert
\end{align}
we get $\lim\limits_{\mathrm{n} \to \infty} v_\mathrm{n}' \left(t_0\right) = 0$. With both these limits, substituting $\gamma = t_0$ in the ODE in (\ref{eq:ODEIVP}) with $C = C_\mathrm{n}$ and taking limits as $\mathrm{n} \to \infty$ will imply that $t_0$ is a root of $p_0 \left(\gamma\right) \gamma$ and so $t_0 = \gamma_{0,0} \in \left(1, m + 1\right)$. Now as $w$ is a subsequential uniform limit of $\left(v_\mathrm{n}\right)$ and $w$ satisfies the ODE IVP (\ref{eq:ODEIVP}) with $C = C_0$ on $\left[1, m + 1\right)$, we will get $w \left(\gamma_{0,0}\right) = w \left(t_0\right) = 0$ and $w' \left(\gamma_{0,0}\right) = w' \left(t_0\right) = 0$ which contradicts Lemmas \ref{lem:firstexist} and \ref{lem:positivity}. Hence the Claim.} \\
As $\inf\limits_{\gamma \in \left[1, m + 1\right], \mathrm{n} \in \mathbb{N}} v_\mathrm{n} \left(\gamma\right) = 0$, we must have $\inf\limits_{\mathrm{n} \in \mathbb{N}} v_\mathrm{n} \left(m + 1\right) = 0$. As $\left(v_{\mathrm{n}_\mathrm{k}}\right) \to w$ uniformly on $\left[1, m + 1\right]$ so $w \left(m + 1\right) = 0$ and hence by Theorem \ref{thm:continue}, $\left[1, m + 1\right)$ is the maximal interval of existence of $w$ as the smooth solution to (\ref{eq:ODEIVP}) with $C = C_0$, and so $w = v_0 = u_0$ on $\left[1, m + 1\right)$ and $w = u_0$ on $\left[1, m + 1\right]$ by continuity. With this, $\left(u_{\mathrm{n}_\mathrm{k}}\right) \to u_0$ uniformly on $\left[1, m + 1\right]$.}
\begin{case*}($3$) \hspace{2pt}
The solution $v_\mathrm{n}$ has maximal interval of existence $\left[1, \gamma_{\star, \mathrm{n}}\right)$ with $\gamma_{\star, \mathrm{n}} := \gamma_{\star, C_\mathrm{n}} \in \left(1, m + 1\right]$ for all $\mathrm{n} \in \mathbb{N}$. W.l.o.g. $\left(\gamma_{\star, \mathrm{n}}\right)$ is a monotone sequence converging to some $\sigma \in \left[1, m + 1\right]$.
\end{case*}
{\noindent Here $u_\mathrm{n} = v_\mathrm{n}$ on $\left[1, \gamma_{\star, \mathrm{n}}\right)$ and $u_\mathrm{n} = 0$ on $\left[\gamma_{\star, \mathrm{n}}, m + 1\right]$ for all $\mathrm{n} \in \mathbb{N}$. As was noted in the beginning, the polynomials $\left( p_\mathrm{n} \left(\gamma\right) \gamma \right) \to p_0 \left(\gamma\right) \gamma$ uniformly on $\left[1, m + 1\right]$ and their respective roots $\gamma_{0, \mathrm{n}} \to \gamma_{0, 0}$. By Lemmas \ref{lem:polynom} and \ref{lem:firstexist}, $1 < \gamma_{0,\mathrm{n}} < \gamma_{\star,\mathrm{n}} \leq m + 1$ for all $n \in \mathbb{N}$ and taking limits as $\mathrm{n} \to \infty$ we get $1 < \gamma_{0,0} \leq \sigma \leq m + 1$ ($\gamma_{0,0} > 1$ by Lemma \ref{lem:polynom}) and specifically $\sigma \in \left(1, m + 1\right]$.
\begin{claim*}
There exists a $\tilde{K} > 0$ such that for any $\mathrm{n} \in \mathbb{N}$ and any local maximum $t \in \left(1, \gamma_{\star, \mathrm{n}}\right)$ of $v_\mathrm{n}$ we have $v_\mathrm{n} \left(t\right) \leq \tilde{K}$.
\end{claim*}
{\noindent For any $\mathrm{n} \in \mathbb{N}$ if $t \in \left(1, \gamma_{\star, \mathrm{n}}\right)$ is a local maximum of $v_\mathrm{n}$ then $v_\mathrm{n}'\left(t\right) = 0$ implies $v_\mathrm{n}\left(t\right) = \frac{p_\mathrm{n}\left(t\right)^2 t^2}{8}$ (from (\ref{eq:ODEIVP}) with $C = C_\mathrm{n}$), and as $\left(p_\mathrm{n}\left(\gamma\right)\gamma\right)$ is uniformly norm bounded on $\left[1, m + 1\right]$ in all cases so there exists a $\tilde{K} > 0$ such that $v_\mathrm{n}\left(t\right) = \frac{p_\mathrm{n}\left(t\right)^2 t^2}{8} \leq \tilde{K}$ for all $\mathrm{n} \in \mathbb{N}$.} \\
Looking at the definition of $u_\mathrm{n}$ in this case we will observe for each $\mathrm{n} \in \mathbb{N}$ that $\max\limits_{\gamma \in \left[1, m + 1\right]} u_\mathrm{n} \left(\gamma\right) = \sup\limits_{\gamma \in \left[1, \gamma_{\star, \mathrm{n}}\right)} v_\mathrm{n} \left(\gamma\right) = v_\mathrm{n} \left(t\right) \leq \tilde{K}$ for some local maximum $t \in \left(1, \gamma_{\star, \mathrm{n}}\right)$ of $v_\mathrm{n}$. Note that the supremum of $v_\mathrm{n}$ on $\left[1, \gamma_{\star, \mathrm{n}}\right)$ will be attained at an interior point only (which will then be a local maximum of $v_\mathrm{n}$), because $v_\mathrm{n} \left(1\right) = 2$ and $v_\mathrm{n}' \left(1\right) = 6$ and from Lemma \ref{lem:positivity}, $\lim\limits_{\gamma \to \gamma_{\star, \mathrm{n}}} v_\mathrm{n} \left(\gamma\right) = 0$ and $\lim\limits_{\gamma \to \gamma_{\star, \mathrm{n}}} v_\mathrm{n}' \left(\gamma\right) < 0$. Thus we see $0 \leq u_\mathrm{n} \left(\gamma\right) \leq \tilde{K}$ for all $\gamma \in \left[1, m + 1\right]$ and for all $\mathrm{n} \in \mathbb{N}$ i.e. $\left(u_\mathrm{n}\right)$ is uniformly norm bounded on $\left[1, m + 1\right]$. By substituting the uniform norm bounds on $\left(\sqrt{v_\mathrm{n}}\right)$ and $\left( p_\mathrm{n} \left(\gamma\right) \gamma \right)$ in the expression for $v_\mathrm{n}'$ in (\ref{eq:ODEIVP}) on $\left[1, \gamma_{\star, \mathrm{n}}\right)$ we get an $\tilde{R} > 0$ such that for each $\mathrm{n} \in \mathbb{N}$, $\left| v_\mathrm{n}' \left(\gamma\right) \right| \leq \tilde{R}$ for all $\gamma \in \left[1, \gamma_{\star, \mathrm{n}}\right)$, and so $\left| \lim\limits_{\gamma \to \gamma_{\star, \mathrm{n}}} v_\mathrm{n}' \left(\gamma\right) \right| \leq \tilde{R}$. So by its definition, $\left(u_\mathrm{n}\right)$ is uniformly Lipschitz on $\left[1, m + 1\right]$. So we extract a subsequence $\left(u_{\mathrm{n}_\mathrm{k}}\right) \to w \in \mathcal{C} \left[1, m + 1\right]$ uniformly on $\left[1, m + 1\right]$. As $u_{\mathrm{n}_\mathrm{k}} \geq 0$ so $w \geq 0$.
\begin{claim*}
$w \left(\sigma\right) = 0$.
\end{claim*}
{\noindent Note that $u_\mathrm{n} \left(\gamma_{\star, \mathrm{n}}\right) = 0$ for all $\mathrm{n} \in \mathbb{N}$. Then using the following estimates for $\mathrm{j}, \mathrm{k} \in \mathbb{N}$:
\begin{align}\label{eq:Case3}
\lvert w\left(\sigma\right) \rvert &\leq \lvert w\left(\sigma\right) - u_{\mathrm{n}_\mathrm{k}}\left(\sigma\right) \rvert + \lvert u_{\mathrm{n}_\mathrm{k}}\left(\sigma\right) - u_{\mathrm{n}_\mathrm{k}}\left(\gamma_{\star, \mathrm{n}_\mathrm{j}}\right) \rvert \nonumber \\
&+ \lvert u_{\mathrm{n}_\mathrm{k}}\left(\gamma_{\star, \mathrm{n}_\mathrm{j}}\right) - u_{\mathrm{n}_\mathrm{k}}\left(\gamma_{\star, \mathrm{n}_\mathrm{k}}\right) \rvert + \lvert u_{\mathrm{n}_\mathrm{k}}\left(\gamma_{\star, \mathrm{n}_\mathrm{k}}\right) \rvert \\
&\leq \lvert w\left(\sigma\right) - u_{\mathrm{n}_\mathrm{k}}\left(\sigma\right) \rvert + \lvert u_{\mathrm{n}_\mathrm{k}}\left(\sigma\right) - u_{\mathrm{n}_\mathrm{k}}\left(\gamma_{\star, \mathrm{n}_\mathrm{j}}\right) \rvert + \tilde{R} \lvert \gamma_{\star, \mathrm{n}_\mathrm{j}} - \gamma_{\star, \mathrm{n}_\mathrm{k}} \rvert \nonumber
\end{align}
it can be easily seen that $w \left(\sigma\right) = 0$.} \\
After this the proof of $w = u_0$ in Case ($3$) will depend upon whether $\left(\gamma_{\star, \mathrm{n}}\right)$ is increasing or decreasing.
\begin{case*}($3.1$) \hspace{2pt}
$\left(\gamma_{\star, \mathrm{n}}\right)$ decreases to $\sigma$.
\end{case*}
{\noindent So $\left[1, \sigma\right] = \bigcap\limits_{\mathrm{k} \in \mathbb{N}} \left[1, \gamma_{\star,\mathrm{n}_\mathrm{k}}\right)$ and so each $v_{\mathrm{n}_\mathrm{k}}$ will satisfy the ODE IVP (\ref{eq:ODEIVP}) with $C = C_{\mathrm{n}_\mathrm{k}}$ on $\left[1, \sigma\right]$. But $v_{\mathrm{n}_\mathrm{k}} = u_{\mathrm{n}_\mathrm{k}}$ on $\left[1, \gamma_{\star, \mathrm{n}_\mathrm{k}}\right)$ and $\left(u_{\mathrm{n}_\mathrm{k}}\right) \to w$ uniformly on $\left[1, m + 1\right]$. So $\left( \sqrt{v_{\mathrm{n}_\mathrm{k}}} \right) \to \sqrt{w}$ uniformly on $\left[1, \sigma\right]$ and this will imply (again from (\ref{eq:ODEIVP}) with $C = C_{\mathrm{n}_\mathrm{k}}$) that $\left(v_{\mathrm{n}_\mathrm{k}}'\right)$ is uniformly convergent on $\left[1, \sigma\right]$. Hence $w$ is differentiable and satisfies (\ref{eq:ODEIVP}) for $C = C_0$ on $\left[1, \sigma\right)$. As $w \left(\sigma\right) = 0$, from Lemma \ref{lem:positivity} and Theorem \ref{thm:continue}, $\left[1, \sigma\right)$ is the maximal interval of existence of the solution $v_0$ and $w = v_0$ on $\left[1, \sigma\right)$ and $\sigma = \gamma_{\star,0}$. Now note that $u_{\mathrm{n}_\mathrm{k}} \equiv 0$ on $\left[\gamma_{\star, \mathrm{n}_\mathrm{k}}, m + 1\right]$ for each $\mathrm{k} \in \mathbb{N}$ and so (the uniform limit) $w \equiv 0$ on $\left(\sigma, m + 1\right] = \bigcup\limits_{\mathrm{k} \in \mathbb{N}} \left[\gamma_{\star,\mathrm{n}_\mathrm{k}}, m + 1\right]$, as $\left(\gamma_{\star, \mathrm{n}}\right)$ is decreasing to $\sigma$. So by its definition, $w = u_0$ on $\left[1, m + 1\right]$.}
\begin{case*}($3.2$) \hspace{2pt}
$\left(\gamma_{\star, \mathrm{n}}\right)$ increases to $\sigma$.
\end{case*}
{\noindent So $\left[1, \sigma\right) = \bigcup\limits_{\mathrm{k} \in \mathbb{N}} \left[1, \gamma_{\star,\mathrm{n}_\mathrm{k}}\right)$ and $\left[1, \gamma_{\star,\mathrm{n}_\mathrm{k}}\right) = \bigcap\limits_{\mathrm{j} \geq \mathrm{k}} \left[1, \gamma_{\star,\mathrm{n}_\mathrm{j}}\right)$ for each $\mathrm{k} \in \mathbb{N}$. So for a fixed $\mathrm{k} \in \mathbb{N}$, $v_{\mathrm{n}_\mathrm{j}}$ satisfies the ODE IVP (\ref{eq:ODEIVP}) with $C = C_{\mathrm{n}_\mathrm{j}}$ on $\left[1, \gamma_{\star,\mathrm{n}_\mathrm{k}}\right)$ for all $\mathrm{j} \geq \mathrm{k}$. By using the same set of arguments as in Case ($3.1$) for the sequence $\left(v_{\mathrm{n}_\mathrm{j}}\right)_{\mathrm{j} \geq \mathrm{k}}$ converging uniformly to $w$ on $\left[1, \gamma_{\star,\mathrm{n}_\mathrm{k}}\right)$ we will see that $\left(v_{\mathrm{n}_\mathrm{j}}'\right)_{\mathrm{j} \geq \mathrm{k}}$ is uniformly convergent on $\left[1, \gamma_{\star,\mathrm{n}_\mathrm{k}}\right)$. So $w$ is differentiable and satisfies (\ref{eq:ODEIVP}) for $C = C_0$ on $\left[1, \gamma_{\star,\mathrm{n}_\mathrm{k}}\right)$, and as this holds true for each $\mathrm{k} \in \mathbb{N}$ so $w$ satisfies (\ref{eq:ODEIVP}) for $C = C_0$ on $\left[1, \sigma\right)$. After this, the same arguments as in Case ($3.1$) will give $\sigma = \gamma_{\star,0}$ and $w = v_0$ on $\left[1, \sigma\right)$ and $w \equiv 0$ on $\left[\sigma, m + 1\right]$, thereby giving $w = u_0$ on $\left[1, m + 1\right]$.} \\
In both the Cases ($3.1$) and ($3.2$), $\left(u_{\mathrm{n}_\mathrm{k}}\right) \to u_0$ uniformly on $\left[1, m + 1\right]$ and $\left(\gamma_{\star,\mathrm{n}_\mathrm{k}}\right) \to \gamma_{\star,0}$.} \\
Since our aim was to find only a subsequence of $\left(u_\mathrm{n}\right)$ which is uniformly convergent, the above Cases suffice.
\end{proof} \par
\begin{remark}
The Cases in Theorem \ref{thm:UnifConv} have given us a hint that $\mathbb{R}$ may be expressed as the disjoint set union of the set of all $C \in \mathbb{R}$ for which $v_C$ exists on $\left[1, m + 1\right]$ with the set of all $C \in \mathbb{R}$ for which $v_C$ breaks down at $\gamma_{\star, C}$, and that the first set is an open interval and the second one is a closed interval.
\end{remark}
\subsection{Second Part of the Proof}\label{subsec:Proof2}
We first do some calculations with the polynomial $p_C \left(\gamma\right) \gamma$ on $\left[1, m + 1\right]$ (in the following $3$ results) which will be needed further. The first one is from Pingali \cite{Pingali:2018:heK}. \par
\begin{lemma}[Pingali]\label{lem:polynom1}
For each $C \in \mathbb{R}$ define $P_C \left(\gamma\right) := P \left(\gamma; C\right) := \int\limits_1^\gamma p_C \left(y\right) y d y$, $\gamma \in \left[1, m + 1\right]$. Define $L := L \left(m\right)$ and $N := N \left(m\right)$ as follows:
\begin{equation}\label{eq:LN}
\begin{gathered}
L \left(m\right) := \frac{3}{10} \left(m + 1\right)^2 - \frac{1}{20} \left(m + 1\right)^4 - \frac{1}{4} - \frac{\left(m + 1\right)^4 - 1}{20 m} \left[1 - \frac{1}{\left(m + 1\right)^2}\right] \\
N \left(m\right) := \frac{1}{10} \left(m + 1\right)^4 - \frac{2}{5} \left(m + 1\right)^2 - \frac{1}{2} + \frac{\left(m + 1\right)^4 - 1}{10 m} \left[1 + \frac{1}{\left(m + 1\right)^2}\right]
\end{gathered}
\end{equation}
Then $L < 0$ and $N > 0$, and $P_C \left(m + 1\right) = L C + N$ and $P_C \left(\gamma\right) \geq \min \left\lbrace 0, L C + N \right\rbrace$.
\end{lemma} \par
\begin{lemma}\label{lem:derpolynom}
The polynomial $q \left(\gamma\right) := \frac{d}{dC}\left(p_C\left(\gamma\right)\gamma\right)$, $\gamma \in \left[1, m + 1\right]$ is independent of $C$, and further $q \left(\gamma\right) < 0$ for $\gamma \in \left(1, m + 1\right)$ and $q \left(1\right) = q \left(m + 1\right) = 0$.
\end{lemma}
\begin{proof}
\begin{equation}\label{eq:derAB}
\frac{d}{dC} \left(A\left(C\right)\right) = \frac{3 \left(m + 2\right)}{\left(m + 1\right)^2} \hspace{2pt}, \hspace{5pt} \frac{d}{dC} \left(B\left(C\right)\right) = - \frac{2 \left(m^2 + 3 m + 3\right)}{\left(m + 1\right)^2}
\end{equation}
\begin{align}\label{eq:qfactor}
\frac{d}{dC} \left(p_C\left(\gamma\right)\gamma\right) &= \frac{m + 2}{\left(m + 1\right)^2} \gamma^4 - \frac{m^2 + 3 m + 3}{\left(m + 1\right)^2} \gamma^3 + \gamma \nonumber \\
&= \left(\gamma - \left( - \frac{m + 1}{m + 2}\right)\right) \left(\gamma - 0\right) \left(\gamma - 1\right) \left(\gamma - \left(m + 1\right)\right)
\end{align}
So $q \left(\gamma\right) = \frac{d}{dC}\left(p_C\left(\gamma\right)\gamma\right)$ is independent of $C$, has its roots at $1$ and $m + 1$, and does not change its sign in $\left(1, m + 1\right)$. Evaluating $q \left(\gamma\right)$ at $\gamma = \frac{m + 2}{2} \in \left(1, m + 1\right)$:
\begin{align}\label{eq:qevalue}
q \left(\frac{m + 2}{2}\right) &= \frac{\left(m + 2\right)^5}{16 \left(m + 1\right)^2} - \frac{\left(m^2 + 3 m + 3\right)\left(m + 2\right)^3}{8 \left(m + 1\right)^2} + \frac{m + 2}{2} \nonumber \\
&= - \frac{m^2 \left(m + 2\right)\left(m^2 + 6 m +6\right)}{16 \left(m + 1\right)^2} \\
&< 0 \nonumber
\end{align}
So for any $\gamma \in \left(1, m + 1\right)$, $q \left(\gamma\right) < 0$ and $q \left(1\right) = q \left(m + 1\right) = 0$.
\end{proof} \par
\begin{corollary}\label{cor:derpolynom}
Define $Q \left(\gamma\right) := \int\limits_1^\gamma q \left(y\right) d y$, $\gamma \in \left[1, m + 1\right]$. Then $Q \left(\gamma\right) < 0$ for all $\gamma \in \left(1, m + 1\right]$ and $Q \left(\gamma\right)$ is strictly decreasing on $\left[1, m + 1\right]$.
\end{corollary} \par
\begin{motivation}
The following calculations and estimates are going to give us that for each $\gamma > 1$, $\frac{d}{dC} \left( v_C\left(\gamma\right) \right) < 0$ on appropriate intervals of $C$ and $\gamma$ i.e. $v \left(\gamma; C\right)$ is strictly decreasing in $C$.
\end{motivation} \par
\begin{theorem}\label{thm:StrctDecr}
Let $\mathsf{V}$ be a non-degenerate interval in $\mathbb{R}$ and $\mathsf{J}$ be a non-degenerate subinterval of $\left[1, m + 1\right]$ containing $1$ such that the smooth solution $v_C$ to the ODE IVP (\ref{eq:ODEIVP}) exists on $\mathsf{J}$ for every $C \in \mathsf{V}$. Then for any $C, D \in \mathsf{V}$ with $C < D$ we have $v\left(\gamma; C\right) \geq v\left(\gamma; D\right) - Q\left(\gamma\right) \left(D - C\right) \geq v\left(\gamma; D\right)$ for all $\gamma \in \mathsf{J}$, with the second inequality being strict if $\gamma > 1$.
\end{theorem}
\begin{proof}
Consider the following operations performed on (\ref{eq:ODEIVP}) with $'$ and $\frac{d}{dC}$ denoting derivatives w.r.t. $\gamma$ and $C$ respectively for $C \in \mathsf{V}$ and $\gamma \in \mathsf{J}$, and use $q \left(\gamma\right)$ and $Q \left(\gamma\right)$ from Lemma \ref{lem:derpolynom} and Corollary \ref{cor:derpolynom} respectively:
\begin{equation}\label{eq:ODEIVPC}
v_C'\left(\gamma\right) = 2 \sqrt{2} \sqrt{v_C\left(\gamma\right)} + p_C\left(\gamma\right)\gamma \hspace{3pt}, \hspace{5pt} v_C\left(1\right) = 2
\end{equation}
\begin{equation}\label{eq:derODEIVPC}
\frac{d}{dC} \left( v_C'\left(\gamma\right) \right) = \frac{\sqrt{2}}{\sqrt{v_C\left(\gamma\right)}} \frac{d}{dC} \left( v_C\left(\gamma\right) \right) + q\left(\gamma\right)
\end{equation}
By Lemma \ref{lem:positivity}, $\sqrt{v_C\left(\gamma\right)} > 0$ for all $\gamma \in \mathsf{J}$. Multiplying by $e^{- \int\limits_1^\gamma \frac{\sqrt{2}}{\sqrt{v\left(y; C\right)}} dy}$ and using the equality of second order mixed partial derivatives:
\begin{equation}\label{eq:intfact}
\left( \frac{d}{dC} \left( v_C\left(\gamma\right) \right) \right)' e^{- \int\limits_1^\gamma \frac{\sqrt{2}}{\sqrt{v\left(y; C\right)}} dy} + \frac{d}{dC} \left( v_C\left(\gamma\right) \right) \left( e^{- \int\limits_1^\gamma \frac{\sqrt{2}}{\sqrt{v\left(y; C\right)}} dy} \right)' = q\left(\gamma\right) e^{- \int\limits_1^\gamma \frac{\sqrt{2}}{\sqrt{v\left(y; C\right)}} dy}
\end{equation}
For $\gamma \in \mathsf{J}$ integrating on $\left[1, \gamma\right]$:
\begin{equation}\label{eq:intODEIVPC}
\frac{d}{dC} \left( v_C\left(\gamma\right) \right) e^{- \int\limits_1^\gamma \frac{\sqrt{2}}{\sqrt{v\left(y; C\right)}} dy} = \int\limits_1^\gamma q\left(y\right) e^{- \int\limits_1^y \frac{\sqrt{2}}{\sqrt{v\left(x; C\right)}} dx} dy
\end{equation}
\begin{align}\label{eq:dersol}
\frac{d}{dC} \left( v_C\left(\gamma\right) \right) &= e^{\int\limits_1^\gamma \frac{\sqrt{2}}{\sqrt{v\left(y; C\right)}} dy} \int\limits_1^\gamma q\left(y\right) e^{- \int\limits_1^y \frac{\sqrt{2}}{\sqrt{v\left(x; C\right)}} dx} dy \nonumber \\
&\leq e^{\int\limits_1^\gamma \frac{\sqrt{2}}{\sqrt{v\left(y; C\right)}} dy} \int\limits_1^\gamma q\left(y\right) e^{- \int\limits_1^\gamma \frac{\sqrt{2}}{\sqrt{v\left(x; C\right)}} dx} dy \\
&= Q\left(\gamma\right) \nonumber
\end{align}
Now for any $C, D \in \mathsf{V}$ with $C < D$ and any $\gamma \in \mathsf{J}$ we have (for some $E \in \left(C, D\right)$):
\begin{equation}\label{eq:LMVT}
v\left(\gamma; D\right) - v\left(\gamma; C\right) = \frac{d}{dC} \left( v_E\left(\gamma\right) \right) \left(D - C\right) \leq Q\left(\gamma\right) \left(D - C\right) \leq 0
\end{equation}
So $v\left(\gamma; C\right) \geq v\left(\gamma; D\right) - Q\left(\gamma\right) \left(D - C\right) \geq v\left(\gamma; D\right)$ for all $\gamma \in \mathsf{J}$ and from Corollary \ref{cor:derpolynom}, the second inequality here is strict if $\gamma > 1$.
\end{proof} \par
Define $\mathscr{C} := \left\lbrace C \in \mathbb{R} \hspace{3pt} \vert \hspace{3pt} \text{$v_C$ exists on $\left[1, m + 1\right]$} \right\rbrace \subseteq \mathbb{R}$. \par
We have the following result of Pingali \cite{Pingali:2018:heK} which along with Lemmas \ref{lem:polynom}, \ref{lem:continue} and \ref{lem:polynom1} was used by them in the proof of Theorem \ref{thm:VPP}.
\begin{lemma}\label{lem:VPP}
$\left(-\infty, 2\right] \subseteq \mathscr{C}$ and $\lim\limits_{C \to -\infty} v\left(m + 1; C\right) = \infty$.
\end{lemma}
\begin{proof}
We just mention the highlights of the proof of Lemma \ref{lem:VPP} and the detailed calculations are given in \cite{Pingali:2018:heK}. For any $C \leq 2$ by Lemma \ref{lem:firstexist}, the solution $v_C$ to (\ref{eq:ODEIVP}) a priori exists on some $\left[1, \tilde{\gamma}\right)$ with $\tilde{\gamma} > 1$. By using $L$ and $N$ of Lemma \ref{lem:polynom1}, it can be checked that if $C \leq 2$ then $L C + N > 0$. Integrating (\ref{eq:ODEIVP}) on $\left[1, \gamma\right]$ for $\gamma \in \left[1, \tilde{\gamma}\right)$ and using Lemma \ref{lem:polynom1} will give $v_C \left(\gamma\right) \geq 2 + P_C \left(\gamma\right) \geq 2$, and then by Lemma \ref{lem:continue}, $v_C$ can be continued beyond $\tilde{\gamma}$, and this will be true for every $\tilde{\gamma} > 1$. So $v_C$ exists on $\left[1, m + 1\right]$ if $C \leq 2$ i.e. $\left(-\infty, 2\right] \subseteq \mathscr{C}$. Now as $L < 0$ and $N > 0$ so $\lim\limits_{C \to -\infty} \left(L C + N\right) = \infty$. So $\lim\limits_{C \to -\infty} v\left(m + 1; C\right) \geq 2 + \lim\limits_{C \to -\infty} P_C \left(m + 1\right) = \infty$ (by Lemma \ref{lem:polynom1}). An important point to be noted here is that if $C \leq 2$ then $L C + N > 0$ i.e. $C < -\frac{N}{L}$, and so $-\frac{N}{L} > 2$.
\end{proof} \par
\begin{motivation}
In the remainder of Subsection \ref{subsec:Proof2} we will prove that $\mathscr{C} = \left(-\infty, M\right)$ and $\lim\limits_{C \to M^-} v\left(m + 1; C\right) = 0$ which will give us a $C$ satisfying the required final boundary condition viz. $v \left(m + 1; C\right) = 2 \left(m + 1\right)^2$.
\end{motivation} \par
\begin{theorem}\label{thm:scrC}
We have the following properties of $\mathscr{C}$:
\begin{enumerate}
\item $\mathscr{C}$ is an interval. \label{itm:scrCinterval}
\item $\mathscr{C}$ is open. \label{itm:scrCopen}
\item $\mathscr{C} \subsetneq \mathbb{R}$. \label{itm:scrCprprsbst}
\item There exists a unique $M = M \left(m\right) > 2$ such that $\mathscr{C} = \left(-\infty, M\right)$. \label{itm:scrCfinal}
\end{enumerate}
\end{theorem}
\begin{proof}
$ $ \vspace{5pt} \newline
(\ref{itm:scrCinterval}) Let $D_1, D_2 \in \mathscr{C}$ with $D_1 < D_2$.
\begin{claim*}
There exists a $\tilde{\gamma} \in \left(1, m + 1\right]$ such that $v_C$ exists at least on $\left[1, \tilde{\gamma}\right)$ for all $C \in \left[D_1, D_2\right]$.
\end{claim*}
{\noindent If not true then there exists a sequence $\left(C_\mathrm{n}\right)$ in $\left[D_1, D_2\right]$ such that $v_\mathrm{n}$ exists maximally on $\left[1, \gamma_{\star,\mathrm{n}}\right)$ and $\left(\gamma_{\star,\mathrm{n}}\right) \to 1$. Passing to a subsequence if necessary, assume $\left(C_\mathrm{n}\right) \to C_0 \in \left[D_1, D_2\right]$. By Theorem \ref{thm:UnifConv} Case ($3$), there exists a subsequence $\left(u_{\mathrm{n}_\mathrm{k}}\right) \to u_0$ uniformly on $\left[1, m + 1\right]$ with $\left(\gamma_{\star,\mathrm{n}_\mathrm{k}}\right) \to \gamma_{\star,0}$ where $\left[1, \gamma_{\star,0}\right)$ is the maximal interval of existence of $v_0$. By Lemma \ref{lem:firstexist}, $\gamma_{\star,0} > 1$. So $\left(\gamma_{\star,\mathrm{n}_\mathrm{k}}\right) \to \gamma_{\star,0}$ and $\left(\gamma_{\star,\mathrm{n}}\right) \to 1$, a contradiction. Hence the Claim.} \\
Take any $\tilde{\gamma} > 1$ with the property mentioned in the Claim above. Applying Theorem \ref{thm:StrctDecr} with $\mathsf{V} = \left[D_1, D_2\right]$ and $\mathsf{J} = \left[1, \tilde{\gamma}\right)$ we will get for any $C \in \left[D_1, D_2\right]$ and for all $\gamma \in \left[1, \tilde{\gamma}\right)$, $v\left(\gamma; C\right) \geq v\left(\gamma; D_2\right)$. Since $D_2 \in \mathscr{C}$ so $v_{D_2}$ exists on $\left[1, m + 1\right]$ and by Theorem \ref{thm:continue}, there exists an $\epsilon > 0$ (depending only on $D_2$) such that $v\left(\gamma; D_2\right) \geq \epsilon$ for all $\gamma \in \left[1, m + 1\right]$ and hence in particular for all $\gamma \in \left[1, \tilde{\gamma}\right)$. So $v\left(\gamma; C\right) \geq \epsilon$ for all $\gamma \in \left[1, \tilde{\gamma}\right)$ and so by Theorem \ref{thm:continue}, $v_C$ can be continued beyond $\tilde{\gamma}$ for all $C \in \left[D_1, D_2\right]$. Since this holds true for any $\tilde{\gamma} > 1$ with the property mentioned in the Claim above and the lower bound $\epsilon > 0$ on $v_C$ does not depend on $C \in \left[D_1, D_2\right]$ as well as on $\tilde{\gamma} > 1$ so $v_C$ has to exist on $\left[1, m + 1\right]$ for all $C \in \left[D_1, D_2\right]$ i.e. $\left[D_1, D_2\right] \subseteq \mathscr{C}$. So $\mathscr{C}$ is an interval. \vspace{5pt} \\
(\ref{itm:scrCopen}) Take a sequence $\left(C_\mathrm{n}\right) \to C_0 \in \mathbb{R}$ of points in $\mathscr{C}^c := \mathbb{R} \smallsetminus \mathscr{C}$. Then $v_\mathrm{n}$ has maximal interval of existence $\left[1, \gamma_{\star,\mathrm{n}}\right)$ and we are in Theorem \ref{thm:UnifConv} Case ($3$). So there exists a subsequence $\left(u_{\mathrm{n}_\mathrm{k}}\right) \to u_0$ uniformly on $\left[1, m + 1\right]$ with $\left(\gamma_{\star,\mathrm{n}_\mathrm{k}}\right) \to \gamma_{\star,0}$. So $\left[1, \gamma_{\star,0}\right)$ is the maximal interval of existence of $v_0$ and so by definition, $C_0 \in \mathscr{C}^c$. So $\mathscr{C}$ is open. \vspace{5pt} \\
(\ref{itm:scrCprprsbst}) Let if possible $v_C$ exist on $\left[1, m + 1\right]$ for all $C \in \mathbb{R}$. Taking $\mathsf{V} = \mathbb{R}$ and $\mathsf{J} = \left[1, m + 1\right]$ in Theorem \ref{thm:StrctDecr} we get for a fixed $C_0 \in \mathbb{R}$ and for any $C > C_0$, and with $\gamma = m + 1$, $v\left(m + 1; C\right) \leq v\left(m + 1; C_0\right) + Q\left(m + 1\right) \left(C - C_0\right)$. Taking $C_\mathrm{n} := C_0 + \mathrm{n}$ for $\mathrm{n} \in \mathbb{N}$ we get $v\left(m + 1; C_\mathrm{n}\right) \leq v\left(m + 1; C_0\right) + \mathrm{n} Q\left(m + 1\right)$. Since $v\left(m + 1; C_\mathrm{n}\right), \hspace{2.5pt} v\left(m + 1; C_0\right) > 0$ and $Q\left(m + 1\right) < 0$ (from Lemma \ref{lem:positivity} and Corollary \ref{cor:derpolynom} respectively) we have $\mathrm{n} < \frac{-v\left(m + 1; C_0\right)}{Q\left(m + 1\right)}$ for all $\mathrm{n} \in \mathbb{N}$, a contradiction. So there exists a $C \in \mathbb{R}$ such that $v_C$ has maximal interval of existence $\left[1, \gamma_{\star,C}\right)$ for some $\gamma_{\star,C} \in \left(1, m + 1\right]$ i.e. $C \in \mathscr{C}^c$. So $\mathscr{C} \subsetneq \mathbb{R}$. \vspace{5pt} \\
(\ref{itm:scrCfinal}) From Lemma \ref{lem:VPP} and (\ref{itm:scrCinterval}), (\ref{itm:scrCopen}) and (\ref{itm:scrCprprsbst}), there exists a unique $M = M \left(m\right) > 2$ such that $\mathscr{C} = \left(-\infty, M\right)$.
\end{proof} \par
For each $C \in \mathbb{R}$ let $I_C \subseteq \left[1, m + 1\right]$ denote the maximal interval of existence of the solution $v_C$ to the ODE IVP (\ref{eq:ODEIVP}). Let $\mathscr{P} \left(\left[1, m + 1\right]\right)$ denote the power set of $\left[1, m + 1\right]$. By Lemma \ref{lem:firstexist} and Theorem \ref{thm:continue} we get the set map $\mathbb{R} \to \mathscr{P} \left(\left[1, m + 1\right]\right)$, $C \mapsto I_C$. Then using Theorems \ref{thm:UnifConv}, \ref{thm:StrctDecr} and \ref{thm:scrC} and the definitions of $u_C$, $\Phi$ and $\mathscr{C}$ we get the following $2$ results:
\begin{corollary}\label{cor:StrctDecr}
The set map $\mathbb{R} \to \mathscr{P} \left(\left[1, m + 1\right]\right)$, $C \mapsto I_C$ and $\Phi$ are monotone decreasing:
\begin{enumerate}
\item If $C_1, C_2 \in \mathscr{C}^c$, $C_1 < C_2$ then $\left[C_1, C_2\right] \subseteq \mathscr{C}^c$ and $\gamma_{\star,C_2} < \gamma_{\star,C_1}$ i.e. $I_{C_2} \subsetneq I_{C_1}$. In general if $C_1, C_2 \in \mathbb{R}$, $C_1 < C_2$ then $I_{C_2} \subseteq I_{C_1}$, with the set containment being strict if $C_2 \in \mathscr{C}^c$ and it being set equality otherwise. \label{itm:IC}
\item If $C_1, C_2 \in \mathscr{C}$, $C_1 < C_2$ then $\left[C_1, C_2\right] \subseteq \mathscr{C}$ and $v\left(\gamma; C_2\right) < v\left(\gamma; C_1\right)$ i.e. $u\left(\gamma; C_2\right) < u\left(\gamma; C_1\right)$ for all $\gamma \in \left(1, m + 1\right]$. In general if $C_1, C_2 \in \mathbb{R}$, $C_1 < C_2$ then $u\left(\gamma; C_2\right) \leq u\left(\gamma; C_1\right)$ for all $\gamma \in \left(1, m + 1\right]$, with the inequality being strict if $\gamma \in I_{C_1}$ and it being equality otherwise. \label{itm:uC}
\end{enumerate}
\end{corollary}
\begin{proof}
$ $ \vspace{5pt} \newline
(\ref{itm:IC}) Given $C_1, C_2 \in \mathscr{C}^c$, $C_1 < C_2$ then as $\mathscr{C}^c = \left[M, \infty\right)$ so clearly $\left[C_1, C_2\right] \subseteq \mathscr{C}^c$. So for every $C \in \left[C_1, C_2\right]$, $I_C = \left[1, \gamma_{\star, C}\right)$.
\begin{claim*}
There exists a $\tilde{\gamma} \in \left(1, m + 1\right]$ such that $v_C$ exists at least on $\left[1, \tilde{\gamma}\right)$ i.e. $\left[1, \tilde{\gamma}\right) \subseteq I_C$ for all $C \in \left[C_1, C_2\right]$.
\end{claim*}
{\noindent The proof of the above Claim is exactly the same as that of the Claim in Theorem \ref{thm:scrC} (\ref{itm:scrCinterval}) with $D_1, D_2$ being replaced by $C_1, C_2$.} \\
Take any $\tilde{\gamma} > 1$ with the property mentioned in the above Claim. Then clearly $\tilde{\gamma} \leq \gamma_{\star, C_2}$. Taking any such $\tilde{\gamma} < \gamma_{\star, C_2}$ and applying Theorem \ref{thm:StrctDecr} with $\mathsf{V} = \left[C_1, C_2\right]$ and $\mathsf{J} = \left[1, \tilde{\gamma}\right)$ we get $v\left(\gamma; C\right) \geq v\left(\gamma; C_2\right)$ for all $\gamma \in \left[1, \tilde{\gamma}\right)$ and for any $C \in \left[C_1, C_2\right]$. Since $\left[1, \tilde{\gamma}\right) \subsetneq I_{C_2}$ so by Lemma \ref{lem:positivity} and Theorem \ref{thm:continue}, there exists an $\epsilon_{\tilde{\gamma}} > 0$ (depending only on $\tilde{\gamma}$ and $C_2$) such that $v\left(\gamma; C_2\right) \geq \epsilon_{\tilde{\gamma}}$ for all $\gamma \in \left[1, \tilde{\gamma}\right)$. So for each $C \in \left[C_1, C_2\right]$, $v\left(\gamma; C\right) \geq \epsilon_{\tilde{\gamma}}$ for all $\gamma \in \left[1, \tilde{\gamma}\right)$ and so by Theorem \ref{thm:continue}, $v_C$ can be continued beyond $\tilde{\gamma}$ for all $C \in \left[C_1, C_2\right]$. Since this holds true for any $1 < \tilde{\gamma} < \gamma_{\star, C_2}$ with the property mentioned in the above Claim and the lower bound $\epsilon_{\tilde{\gamma}} > 0$ on $v_C$ does not depend on $C \in \left[C_1, C_2\right]$ so $v_C$ has to exist on $\left[1, \gamma_{\star, C_2}\right)$ i.e. $I_{C_2} \subseteq I_C$ for all $C \in \left[C_1, C_2\right]$. \\
So in particular $I_{C_2} \subseteq I_{C_1}$. Let if possible $I_{C_1} = I_{C_2}$ i.e. $\gamma_{\star, C_1} = \gamma_{\star, C_2}$. So by Lemma \ref{lem:positivity} we have $\lim\limits_{\gamma \to \gamma_{\star, C_2}} v \left(\gamma; C_2\right) = 0 = \lim\limits_{\gamma \to \gamma_{\star, C_1}} v \left(\gamma; C_1\right)$. Taking $\mathsf{V} = \left[C_1, C_2\right]$ and $\mathsf{J} = \left[1,
\gamma_{\star, C_2}\right)$ in Theorem \ref{thm:StrctDecr} we get $v\left(\gamma; C_1\right) \geq v\left(\gamma; C_2\right) - Q\left(\gamma\right) \left(C_2 - C_1\right)$ for all $\gamma \in \left[1, \gamma_{\star, C_2}\right)$. Applying limits as $\gamma \to \gamma_{\star, C_2}$ we get $Q\left(\gamma_{\star, C_2}\right) \geq 0$ where $\gamma_{\star, C_2} > 1$ which is a contradiction to Corollary \ref{cor:derpolynom}. So $I_{C_2} \subsetneq I_{C_1}$ i.e. $\gamma_{\star,C_2} < \gamma_{\star,C_1}$. \\
If $C_2 \in \mathscr{C}^c$ and $C_1 \in \mathscr{C}$ then $C_1 < C_2$ and $I_{C_2} = \left[1, \gamma_{\star, C_2}\right) \subsetneq \left[1, m + 1\right] = I_{C_1}$. If $C_1, C_2 \in \mathscr{C}$, $C_1 < C_2$ then $I_{C_1} = \left[1, m + 1\right] = I_{C_2}$. Thus the general statement for $C_1, C_2 \in \mathbb{R}$, $C_1 < C_2$ holds true. \vspace{5pt} \\
(\ref{itm:uC}) Given $C_1, C_2 \in \mathscr{C}$, $C_1 < C_2$ then from Theorem \ref{thm:scrC} (\ref{itm:scrCinterval}), $\left[C_1, C_2\right] \subseteq \mathscr{C}$. So for every $C \in \left[C_1, C_2\right]$, $I_C = \left[1, m + 1\right]$ and $u_C = v_C$ on $\left[1, m + 1\right]$. Applying Theorem \ref{thm:StrctDecr} with $\mathsf{V} = \left[C_1, C_2\right]$ and $\mathsf{J} = \left[1, m + 1\right]$ we get $v\left(\gamma; C_2\right) < v\left(\gamma; C_1\right)$ i.e. $u\left(\gamma; C_2\right) < u\left(\gamma; C_1\right)$ for all $\gamma \in \left(1, m + 1\right]$. \\
If $C_1 \in \mathscr{C}$ and $C_2 \in \mathscr{C}^c$ then $C_1 < C_2$ and $u_{C_1} = v_{C_1}$ on $I_{C_1} = \left[1, m + 1\right]$ and $u_{C_2} = v_{C_2}$ on $I_{C_2} = \left[1, \gamma_{\star, C_2}\right)$ and $u_{C_2} \equiv 0$ on $\left[\gamma_{\star, C_2}, m + 1\right]$. From (\ref{itm:IC}), $I_{C_2} \subseteq I_C$ for all $C \in \left[C_1, C_2\right]$ and so by Theorem \ref{thm:StrctDecr} with $\mathsf{V} = \left[C_1, C_2\right]$ and $\mathsf{J} = \left[1, \gamma_{\star, C_2}\right)$ we get $v\left(\gamma; C_2\right) < v\left(\gamma; C_1\right)$ for all $\gamma \in \left[1, \gamma_{\star, C_2}\right)$. On $\left[\gamma_{\star, C_2}, m + 1\right]$ by Lemma \ref{lem:positivity}, $u_{C_2} = 0 < v_{C_1} = u_{C_1}$. \\
If $C_1, C_2 \in \mathscr{C}^c$, $C_1 < C_2$ then by (\ref{itm:IC}), $u_C = v_C$ on $I_C = \left[1, \gamma_{\star, C}\right)$ and $u_C \equiv 0$ on $\left[\gamma_{\star, C}, m + 1\right]$ for all $C \in \left[C_1, C_2\right]$. Also by (\ref{itm:IC}), $I_{C_2} \subseteq I_C$ for all $C \in \left[C_1, C_2\right]$ and so again using Theorem \ref{thm:StrctDecr} with $\mathsf{V} = \left[C_1, C_2\right]$ and $\mathsf{J} = \left[1, \gamma_{\star, C_2}\right)$ we get $v\left(\gamma; C_2\right) < v\left(\gamma; C_1\right)$ for all $\gamma \in \left[1, \gamma_{\star, C_2}\right)$. As $\gamma_{\star, C_2} < \gamma_{\star, C_1}$ (again by (\ref{itm:IC})) so on $\left[\gamma_{\star, C_2}, \gamma_{\star, C_1}\right)$ by Lemma \ref{lem:positivity} and Theorem \ref{thm:continue}, $u_{C_2} = 0 < v_{C_1} = u_{C_1}$. On $\left[\gamma_{\star, C_1}, m + 1\right]$, $u_{C_1} = 0 = u{C_2}$ as $I_{C_2} \subsetneq I_{C_1}$. \\
Thus the general statement for $C_1, C_2 \in \mathbb{R}$, $C_1 < C_2$ holds true.
\end{proof} \par
Because of monotonicity we can get the uniform convergence of the whole sequence $\left(u_\mathrm{n}\right)$ instead of just a subsequence $\left(u_{\mathrm{n}_\mathrm{k}}\right)$ in Theorem \ref{thm:UnifConv}:
\begin{corollary}\label{cor:Phicont}
If $\left(C_\mathrm{n}\right) \uparrow C_0$ then $\left(u_\mathrm{n}\right) \downarrow u_0$ and more generally if $\left(C_\mathrm{n}\right) \to C_0$ then $\left(u_\mathrm{n}\right) \to u_0$ uniformly on $\left[1, m + 1\right]$. Thus $\Phi$ is continuous.
\end{corollary} \par
We now have the final result of Subsections \ref{subsec:Proof1} and \ref{subsec:Proof2} which will prove Theorem \ref{thm:main} and as a consequence Corollary \ref{cor:main}:
\begin{corollary}\label{cor:final}
$\lim\limits_{C \to M^-} v\left(m + 1; C\right) = 0$. There exists a unique $C = C \left(m\right) \in \left(-\infty, M\right)$ such that $v \left(m + 1; C\right) = 2 \left(m + 1\right)^2$ and the $C$ with this property has to be strictly greater than $2$.
\end{corollary}
\begin{proof}
By Theorem \ref{thm:scrC}, $\mathscr{C}^c = \left[M, \infty\right)$. Let $\left(C_\mathrm{n}\right) \downarrow M$ be a sequence of points in $\mathscr{C}^c$ then from Corollary \ref{cor:Phicont}, $\left(u_\mathrm{n}\right) \uparrow u_M$ uniformly on $\left[1, m + 1\right]$ and by definition, $I_\mathrm{n} := I_{C_\mathrm{n}} = \left[1, \gamma_{\star, \mathrm{n}}\right)$ and so by Theorem \ref{thm:UnifConv} Case ($3.2$) and Corollaries \ref{cor:StrctDecr} and \ref{cor:Phicont}, $I_M = \left[1, \gamma_{\star, M}\right)$ where $\left(\gamma_{\star, \mathrm{n}}\right) \uparrow \gamma_{\star, M}$ i.e. $\left(I_\mathrm{n}\right)$ is monotone increasing in $\mathscr{P} \left(\left[1, m + 1\right]\right)$ and $I_M = \bigcup\limits_{\mathrm{n} \in \mathbb{N}} I_\mathrm{n}$. \\
Let $\left(E_\mathrm{n}\right) \uparrow M \in \mathscr{C}^c$ be a sequence of points in $\mathscr{C} = \left(-\infty, M\right)$ then $\left(u_{E_\mathrm{n}}\right) \downarrow u_M$ uniformly on $\left[1, m + 1\right]$ and by definition, $I_{E_\mathrm{n}} = \left[1, m + 1\right]$ and $u_{E_\mathrm{n}} = v_{E_\mathrm{n}}$ on $\left[1, m + 1\right]$ for all $\mathrm{n} \in \mathbb{N}$. So we will land up in either of the Cases ($1$) or ($2$) in Theorem \ref{thm:UnifConv}. In Case ($1$), $\inf\limits_{\gamma \in \left[1, m + 1\right], \mathrm{n} \in \mathbb{N}} u_{E_\mathrm{n}} \left(\gamma\right) = \epsilon > 0$ and hence the uniform limit $u_M = v_M \geq \epsilon > 0$ on $\left[1, m + 1\right]$, thereby implying (by Lemma \ref{lem:continue}) that $M \in \mathscr{C}$, a contradiction. So we are in Theorem \ref{thm:UnifConv} Case ($2$) and hence $I_M = \left[1, m + 1\right)$ (i.e. $\gamma_{\star, M} = m + 1$) and $u_M = v_M > 0$ (by Lemma \ref{lem:positivity}) on $\left[1, m + 1\right)$ and $u_M \left(m + 1\right) = 0$. So by pointwise convergence, $\left(u_{E_\mathrm{n}} \left(m + 1\right) \right) \downarrow u_M \left(m + 1\right)$ i.e. $\left(v_{E_\mathrm{n}} \left(m + 1\right)\right) \downarrow 0$ and so we get $\lim\limits_{C \to M^-} v\left(m + 1; C\right) = 0$. \\
From this and Lemma \ref{lem:VPP} there exists a $C = C \left(m\right) \in \mathscr{C} = \left(-\infty, M\right)$ such that $v \left(m + 1; C\right) = 2 \left(m + 1\right)^2$ and by the strictness of the inequalities in Theorem \ref{thm:StrctDecr} and Corollary \ref{cor:StrctDecr}, this $C$ has to be unique. \\
Now to show this $C$ is strictly greater than $2$ we recollect from Subsection \ref{subsec:AnalysisODEBVP} that if there exists a smooth solution $v_C$ to the ODE IVP (\ref{eq:ODEIVP}) satisfying both the boundary conditions viz. $v_C \left(1\right) = 2$ and $v_C \left(m + 1\right) = 2 \left(m + 1\right)^2$ then $v_C \left(\gamma\right) > 2 \gamma^2$ for all $\gamma \in \left(1, m + 1\right)$. Substituting all this and $P_C \left(m + 1\right) = L C + N$ (from Lemma \ref{lem:polynom1}) in (\ref{eq:ODEIVP}) and integrating it over $\left[1, m + 1\right]$ we get $2 \left(m + 1\right)^2 - 2 > 2 \left( \left(m + 1\right)^2 - 1 \right) + L C + N$ which implies $L C + N < 0$ i.e. $C > -\frac{N}{L}$. As was noted in the proof of Lemma \ref{lem:VPP}, $-\frac{N}{L} > 2$. So if $v_C \left(m + 1\right) = 2 \left(m + 1\right)^2$ then $C > 2$.
\end{proof}
\numberwithin{equation}{section}
\numberwithin{definition}{section}
\numberwithin{lemma}{section}
\numberwithin{theorem}{section}
\numberwithin{corollary}{section}
\numberwithin{remark}{section}
\numberwithin{case}{section}
\numberwithin{claim}{section}
\numberwithin{motivation}{section}
\numberwithin{question}{section}
\section{Bando-Futaki Invariants and hcscK Metrics}\label{sec:Bando-Futaki}
Let $\mathrm{M}$ be a compact K\"ahler $n$-manifold. Given a K\"ahler form $\omega$ on $\mathrm{M}$ there exists a $\lambda \in \mathcal{C}^\infty \left(\mathrm{M}, \mathbb{R}\right)$ such that $c_n \left(\omega\right) = \lambda \omega^n \in \Omega^{\left(n,n\right)} \left(\mathrm{M}\right) = \Omega^{2 n} \left(\mathrm{M}\right)$ where $c_j \left(\omega\right)$ is the $j$\textsuperscript{th} Chern form of $\omega$, $\Omega^{\left(i,j\right)} \left(\mathrm{M}\right)$ is the set of all real $\left(i,j\right)$-forms on $\mathrm{M}$ and $\Omega^\mathrm{r} \left(\mathrm{M}\right)$ is the set of all real $\mathrm{r}$-forms on $\mathrm{M}$. By using Hodge Theory (see Aubin \cite{Aubin:1998:NonlinRiemGeom}) we get:
\begin{equation}\label{eq:Hodge1}
c_n \left(\omega\right) - \mathtt{H} c_n \left(\omega\right) = \sqrt{-1} \partial \bar{\partial} \varphi
\end{equation}
where $\mathtt{H}$ denotes harmonic projection and $\varphi \in \Omega^{\left(n-1,n-1\right)} \left(\mathrm{M}\right) \subseteq \Omega^{2 n - 2} \left(\mathrm{M}\right)$. Furthermore $\mathtt{H} c_n \left(\omega\right) = \lambda_0 \omega^n$ for some $\lambda_0 \in \mathbb{R}$. So equation (\ref{eq:Hodge1}) becomes:
\begin{equation}\label{eq:Hodge2}
\sqrt{-1} \partial \bar{\partial} \varphi = \left(\lambda - \lambda_0\right) \omega^n
\end{equation}
The $n$\textsuperscript{th} \textit{Bando-Futaki invariant} for the K\"ahler class $\left[\omega\right]$ on $\mathrm{M}$ is defined as (by Bando \cite{Bando:2006:HarmonObstruct}):
\begin{equation}\label{eq:BFdef}
\mathcal{F}_n \left(Y, \left[\omega\right]\right) := \int\limits_{\mathrm{M}} \mathcal{L}_Y \varphi \wedge \omega \hspace{1pt}, \hspace{5pt} Y \in \mathfrak{h} \left(\mathrm{M}\right)
\end{equation}
where $\mathcal{L}_Y$ denotes Lie derivative w.r.t. $Y$ and $\mathfrak{h} \left(\mathrm{M}\right)$ is the set of all holomorphic vector fields on $\mathrm{M}$. It has been proven by Bando \cite{Bando:2006:HarmonObstruct} that $\mathcal{F}_n$ is a function of the K\"ahler class $\left[\omega\right]$ alone and does not depend on the choice of the K\"ahler metric $\omega$ in $\left[\omega\right]$ and the $\left(n-1,n-1\right)$-form $\varphi$ satisfying equation (\ref{eq:Hodge1}), so that we have $\mathcal{F}_n : \mathfrak{h} \left(\mathrm{M}\right) \times H^{\left(1,1\right)} \left(\mathrm{M}, \mathbb{R}\right)^+ \to \mathbb{R}$. The top Bando-Futaki invariant for a K\"ahler class provides an obstruction to the existence of hcscK metrics in it, as a K\"ahler metric is hcscK if and only if its top Chern form is harmonic (\cite{Bando:2006:HarmonObstruct} and Pingali \cite{Pingali:2018:heK}).
\begin{theorem}[Bando]\label{thm:BFhcscK}
If $\omega$ is hcscK then $\mathcal{F}_n \left(\cdot, \left[\omega\right]\right) \equiv 0$ on $\mathfrak{h} \left(\mathrm{M}\right)$.
\end{theorem} \par
Now by using standard facts about Lie derivatives and harmonicity of differential forms, and equations (\ref{eq:Hodge1}), (\ref{eq:Hodge2}) and (\ref{eq:BFdef}) we do the following computations for the top Bando-Futaki invariant of a general K\"ahler class $\left[\omega\right]$ on $\mathrm{M}$:
\begin{align}\label{eq:BFcalc}
\mathcal{F}_n \left(Y, \left[\omega\right]\right) &= \int\limits_{\mathrm{M}} \mathcal{L}_Y \varphi \wedge \omega \nonumber \\
&= - \int\limits_{\mathrm{M}} \varphi \wedge \mathcal{L}_Y \omega \nonumber \\
&= - \int\limits_{\mathrm{M}} \varphi \wedge \left( d \left(\iota_Y \omega\right) + \iota_Y \left(d \omega\right) \right) && \left(\text{Where $\iota_Y$ denotes interior product w.r.t. $Y$}\right) \nonumber \\
&= - \int\limits_{\mathrm{M}} \varphi \wedge d \left(\iota_Y \omega\right) \\
&= - \int\limits_{\mathrm{M}} \varphi \wedge \sqrt{-1} \partial \bar{\partial} \mathtt{f} && \left(\text{For some $\mathtt{f} \in \mathcal{C}^\infty \left(\mathrm{M}, \mathbb{R}\right)$}\right) \nonumber \\
&= - \int\limits_{\mathrm{M}} \sqrt{-1} \partial \bar{\partial} \varphi \wedge \mathtt{f} \nonumber \\
&= - \int\limits_{\mathrm{M}} \mathtt{f} \left(\lambda - \lambda_0\right) \omega^n \nonumber
\end{align}
Actually we can make one more observation from equation (\ref{eq:Hodge2}) viz. that $\lambda_0 = \frac{\int\limits_{\mathrm{M}} \lambda \omega^n}{\int\limits_{\mathrm{M}} \omega^n}$ is the average value of $\lambda$ on $\mathrm{M}$. We will do one more set of calculations before proceeding to prove our main result of this section: \par
Let $J$ denote the almost complex structure induced by the complex structure on $\mathrm{M}$, $g$ denote the K\"ahler metric associated with the K\"ahler form $\omega$, and $\flat : \mathfrak{X}^{\left(1,0\right)} \left(\mathrm{M}\right) \to \Omega^{\left(0,1\right)} \left(\mathrm{M}\right)$ and $\sharp : \Omega^{\left(0,1\right)} \left(\mathrm{M}\right) \to \mathfrak{X}^{\left(1,0\right)} \left(\mathrm{M}\right)$ be the `musical isomorphisms' induced by $g$, where $\mathfrak{X}^{\left(1,0\right)} \left(\mathrm{M}\right)$ denotes the set of all $\left(1,0\right)$-vector fields on $\mathrm{M}$. For any $Y \in \mathfrak{h} \left(\mathrm{M}\right) \subseteq \mathfrak{X}^{\left(1,0\right)} \left(\mathrm{M}\right)$ we have:
\begin{equation}\label{eq:IntProd}
\sqrt{-1} Y^{\flat} \left(Z\right) = J \left(Y\right)^{\flat} \left(Z\right) = g \left( J \left(Y\right), Z \right) = \omega \left(Y, Z\right) = \iota_Y \omega \left(Z\right)
\end{equation}
for all $Z \in T \mathrm{M}$, the tangent bundle of $\mathrm{M}$, so that $\iota_Y \omega = \sqrt{-1} Y^{\flat}$. Finally we can see that in a K\"ahler class where higher extremal K\"ahler metrics exist, the top Bando-Futaki invariant characterizes the obstruction to the existence of hcscK metrics:
\begin{theorem}\label{thm:BFheK}
Let $\omega$ be a higher extremal K\"ahler metric on $\mathrm{M}$. Then $\omega$ is hcscK if and only if $\mathcal{F}_n \left(\cdot, \left[\omega\right]\right) \equiv 0$ on $\mathfrak{h} \left(\mathrm{M}\right)$. Furthermore let $\omega$ be hcscK. Then every higher extremal K\"ahler metric in the K\"ahler class $\left[\omega\right]$ is hcscK.
\end{theorem}
\begin{proof}
Since $\omega$ is higher extremal K\"ahler, $\nabla^{1,0} \lambda = \left(\bar{\partial} \lambda\right)^{\sharp} \in \mathfrak{h} \left(X\right)$. Substituting $Y = \left(\bar{\partial} \lambda\right)^{\sharp}$ in the calculations (\ref{eq:BFcalc}) and (\ref{eq:IntProd}) and using equation (\ref{eq:Hodge2}) we get $\mathtt{f} = \left(\lambda - \lambda_0\right)$ as follows:
\begin{equation}\label{eq:BFcalcfinal1}
\iota_Y \omega = \sqrt{-1} \bar{\partial} \lambda
\end{equation}
\begin{equation}\label{eq:BFcalcfinal2}
\mathcal{L}_Y \omega = d \left(\iota_Y \omega\right) = d \left(\sqrt{-1} \bar{\partial} \lambda\right) = \sqrt{-1} \partial \bar{\partial} \lambda = \sqrt{-1} \partial \bar{\partial} \left(\lambda - \lambda_0\right)
\end{equation}
\begin{align}\label{eq:BFcalcfinal3}
\mathcal{F}_n \left(Y, \left[\omega\right]\right) &= - \int\limits_{\mathrm{M}} \varphi \wedge \mathcal{L}_Y \omega \nonumber \\
&= - \int\limits_{\mathrm{M}} \varphi \wedge \sqrt{-1} \partial \bar{\partial} \left(\lambda - \lambda_0\right) \nonumber \\
&= - \int\limits_{\mathrm{M}} \sqrt{-1} \partial \bar{\partial} \varphi \wedge \left(\lambda - \lambda_0\right) \\
&= - \int\limits_{\mathrm{M}} \left(\lambda - \lambda_0\right)^2 \omega^n \nonumber \\
&=: - \lVert \lambda - \lambda_0 \rVert_{\mathcal{L}^2 \left(\mathrm{M}, \omega\right)}^2 \nonumber
\end{align}
Now if $\omega$ is higher extremal K\"ahler and $\mathcal{F}_n \left(\cdot, \left[\omega\right]\right) \equiv 0$ then $0 = \mathcal{F}_n \left(\left(\bar{\partial} \lambda\right)^{\sharp}, \left[\omega\right]\right) = - \lVert \lambda - \lambda_0 \rVert_{\mathcal{L}^2 \left(\mathrm{M}, \omega\right)}^2$, and hence $\lambda = \lambda_0$ i.e. $\omega$ is hcscK. The converse was already done by Bando \cite{Bando:2006:HarmonObstruct} in Theorem \ref{thm:BFhcscK}. The second assertion of Theorem \ref{thm:BFheK} follows directly from the first assertion.
\end{proof} \par
Now returning back to the example of our surface $X$ we state the following result of Pingali \cite{Pingali:2018:heK} saying that the higher extremal K\"ahler metric on $X$, obtained by the momentum construction method in Subsections \ref{subsec:MomentConstruct} and \ref{subsec:AnalysisODEBVP}, cannot be hcscK:
\begin{theorem}[Pingali]\label{thm:nonhcscK}
For any $m > 0$ if there exists a higher extremal K\"ahler metric $\omega$ in the K\"ahler class $2 \pi \left(\mathsf{C} + m S_\infty\right)$ on $X$ satisfying the ansatz (\ref{eq:ansatz}) and the conditions (\ref{eq:problem1}) with its momentum profile $\phi \left(\gamma\right)$ satisfying the ODE (\ref{eq:ODE}) on $\left[1, m + 1\right]$ with the boundary conditions (\ref{eq:BVP}) and with $\phi \left(\gamma\right) > 0$ on $\left(1, m + 1\right)$, then $\nabla^{1,0} \lambda \neq 0$ i.e. $\omega$ is not hcscK.
\end{theorem}
\begin{remark}
Note that even though the existence of a higher extremal K\"ahler metric by the momentum construction method in the K\"ahler class $2 \pi \left(\mathsf{C} + m S_\infty\right)$ was not proven by Pingali \cite{Pingali:2018:heK} for a general $m$, but the fact, that such a metric if it exists is not hcscK, was proven by them for a general $m$.
\end{remark} \par
Now just like being a higher extremal K\"ahler metric is a scale-invariant property, in exactly the same way being an hcscK metric is also a scale-invariant property. Indeed from the rescaling arguments in Subsection \ref{subsec:AnalysisODEBVP}, on a compact K\"ahler $n$-manifold $\mathrm{M}$ if $\omega$ satisfying equation (\ref{eq:defheK}) is an hcscK metric then so is $k \omega$ for any $k > 0$, as seen from equation (\ref{eq:rescalheK}). \par
Thus on our surface $X$ for any $k, m > 0$ if $\omega \in 2 \pi \left(\mathsf{C} + m S_\infty\right)$ is higher extremal K\"ahler but not hcscK and satisfies the equations (\ref{eq:problem1}) then $\eta := \frac{k}{2 \pi} \omega \in k \left(\mathsf{C} + m S_\infty\right)$ is also higher extremal K\"ahler but not hcscK and satisfies the equation $c_2 \left(\eta\right) = \frac{\lambda}{2 k^2} \eta^2$. From this and from Theorem \ref{thm:nonhcscK} and Corollaries \ref{cor:main} and \ref{cor:rescalheK} we get the following result:
\begin{corollary}\label{cor:rescalhcscK}
For all $a, b > 0$ there exists a K\"ahler metric $\eta$ on $X$, which is higher extremal K\"ahler but not hcscK, satisfying the following:
\begin{equation}\label{eq:probheKnonhcscK}
\left[\eta\right] = a \mathsf{C} + b S_\infty \hspace{1pt}, \hspace{5pt} c_2 \left(\eta\right) = \frac{\lambda}{2 a^2} \eta^2 \hspace{1pt}, \hspace{5pt} \nabla^{1,0} \lambda \neq 0 \in \mathfrak{h} \left(X\right)
\end{equation}
\end{corollary} \par
We finally conclude the non-existence of any hcscK metrics on $X$ from Theorem \ref{thm:BFheK} and Corollary \ref{cor:rescalhcscK}:
\begin{corollary}\label{cor:nonexisthcscK}
For all $a, b > 0$ there does not exist an hcscK metric in the K\"ahler class $a \mathsf{C} + b S_\infty$ on $X$. Thus there do not exist any hcscK metrics on $X$.
\end{corollary}
\numberwithin{equation}{subsection}
\numberwithin{definition}{subsection}
\numberwithin{lemma}{subsection}
\numberwithin{theorem}{subsection}
\numberwithin{corollary}{subsection}
\numberwithin{remark}{subsection}
\numberwithin{case}{subsection}
\numberwithin{claim}{subsection}
\numberwithin{motivation}{subsection}
\numberwithin{question}{subsection}
\section{Generalization of the Results to All Pseudo-Hirzebruch Surfaces}\label{sec:Generalgd}
\subsection{The K\"ahler Cone of a Pseudo-Hirzebruch Surface}\label{subsec:KConegd}
Following the terminology of T{\o}nnesen-Friedman \cite{Tonnesen:1998:eKminruledsurf} we define a \textit{pseudo-Hirzebruch surface} to be a compact complex surface of the form $X := \mathbb{P} \left(\mathrm{L} \oplus \mathcal{O}\right)$ where $\mathrm{L}$ is a holomorphic line bundle of degree $\mathtt{d} \neq 0$ on a (compact) Riemann surface $\Sigma$ of genus $\mathtt{g} \geq 2$. Note that $X$ is a minimal ruled (complex) surface. We equip $\Sigma$ with a K\"ahler metric $\omega_\Sigma$ of constant scalar curvature $\mathrm{S}_\Sigma := -2 \left(\mathtt{g} - 1\right)$ and we equip $\mathrm{L}$ with a Hermitian metric $h$ of curvature form $\mathrm{F} \left(h\right) := \mathtt{d} \omega_\Sigma$ so that the area of $\Sigma$ w.r.t. $\omega_\Sigma$ and the Ricci curvature form of $\omega_\Sigma$ are respectively the following (where $\chi = -2 \left(\mathtt{g} - 1\right)$ is the Euler characteristic of $\Sigma$):
\begin{equation}\label{eq:ArhoSigma}
\mathrm{A}_\Sigma = \int\limits_{\Sigma} \omega_\Sigma = 2 \pi \frac{\chi}{\mathrm{S}_\Sigma} = 2 \pi \hspace{1pt}, \hspace{5pt} \rho_\Sigma = \mathrm{S}_\Sigma \omega_\Sigma = -2 \left(\mathtt{g} - 1\right) \omega_\Sigma
\end{equation} \par
As in Subsection \ref{subsec:KCone} let $\mathsf{C}$ be the Poincar\'e dual of a typical fibre of $X$, $S_\infty = \mathbb{P} \left(\mathrm{L} \oplus \left\lbrace 0 \right\rbrace\right) \subseteq \mathbb{P} \left(\mathrm{L} \oplus \mathcal{O}\right) = X$ be the infinity divisor of $X$ and $S_0 = \mathbb{P} \left(\left\lbrace 0 \right\rbrace \oplus \mathcal{O}\right) \subseteq \mathbb{P} \left(\mathrm{L} \oplus \mathcal{O}\right) = X$ be the zero divisor of $X$ (and $\Sigma$ be identified with $S_0$ as a (complex) curve in $X$), and further let $c_1 \left(\mathrm{L}\right)$ be the first Chern class of $\mathrm{L}$, $\left[\omega_\Sigma\right]$ be the K\"ahler class of $\omega_\Sigma$ and $\left[\Sigma\right]$ be the fundamental class of $\Sigma$ (and $\left[\Sigma\right]$ be identified with $S_0$ in $H^2 \left(X, \mathbb{R}\right)$). In this general case we have the following intersection formulae similar to the intersection formulae (\ref{eq:IntersectForm}) and (\ref{eq:Sigma}) in the special case in Subsection \ref{subsec:KCone} (Barth, Hulek et al. \cite{Barth:2004:CmpctCmplxSurf}, Sz\'ekelyhidi \cite{Szekelyhidi:2014:eKintro} and T{\o}nnesen-Friedman \cite{Tonnesen:1998:eKminruledsurf}):
\begin{equation}\label{eq:IntersectFormgd}
\mathsf{C}^2 = 0 \hspace{1pt}, \hspace{5pt} S_\infty^2 = -\mathtt{d} \hspace{1pt}, \hspace{5pt} S_0^2 = \mathtt{d} \hspace{1pt}, \hspace{5pt} \mathsf{C} \cdot S_\infty = 1 \hspace{1pt}, \hspace{5pt} \mathsf{C} \cdot S_0 = 1 \hspace{1pt}, \hspace{5pt} S_\infty \cdot S_0 = 0
\end{equation}
\begin{equation}\label{eq:Sigmagd}
c_1 \left(\mathrm{L}\right) \cdot \left[\Sigma\right] = \mathtt{d} \hspace{1pt}, \hspace{5pt} \left[\omega_\Sigma\right] \cdot \left[\Sigma\right] = 2 \pi \hspace{1pt}, \hspace{5pt} \mathsf{C} \cdot \left[\Sigma\right] = 1 \hspace{1pt}, \hspace{5pt} S_\infty \cdot \left[\Sigma\right] = 0 \hspace{1pt}, \hspace{5pt} S_0 \cdot \left[\Sigma\right] = \mathtt{d}
\end{equation} \par
Now let us first assume that $\mathtt{d} < 0$ which will be the case in the remainder of Subsection \ref{subsec:KConegd} as well as in Subsection \ref{subsec:MomentConstructgd}. By using the Leray-Hirsch Theorem and the Nakai-Moishezon Criterion in the real cohomology case (Fujiki \cite{Fujiki:1992:eKruledmani}, Lamari \cite{Lamari:1999:Kcone}, LeBrun-Singer \cite{LeBrun:1993:scalflatKcmpctcmplxsurf} and T{\o}nnesen-Friedman \cite{Tonnesen:1998:eKminruledsurf}) we compute the K\"ahler cone of $X$ as follows (attributed to Fujiki \cite{Fujiki:1992:eKruledmani} and given in T{\o}nnesen-Friedman \cite{Tonnesen:1998:eKminruledsurf}):
\begin{equation*}
H^2 \left(X, \mathbb{R}\right) = \mathbb{R} \mathsf{C} \oplus \mathbb{R} S_\infty
\end{equation*}
\begin{align}\label{eq:KConeXgd}
H^{\left(1,1\right)} \left(X, \mathbb{R}\right)^+ &= \left\lbrace a \mathsf{C} + b S_\infty \hspace{3pt} \vert \hspace{3pt} 2 a b - \mathtt{d} b^2 > 0 \hspace{1pt}, \hspace{5pt} b > 0 \hspace{1pt}, \hspace{5pt} a - \mathtt{d} b > 0 \hspace{1pt}, \hspace{5pt} a > 0 \hspace{1pt}, \hspace{5pt} a > 0 \right\rbrace \\
&= \left\lbrace a \mathsf{C} + b S_\infty \hspace{3pt} \vert \hspace{3pt} a, b > 0 \right\rbrace \nonumber
\end{align}
where the inequalities characterizing a general K\"ahler class $a \mathsf{C} + b S_\infty$ on $X$ are obtained by substituting the intersection formulae (\ref{eq:IntersectFormgd}) and (\ref{eq:Sigmagd}) in Corollary \ref{cor:TFKCone} which follows from Theorem \ref{thm:Nakai-Moishezon}, both of which are applicable in this general setting of the pseudo-Hirzebruch surface $X = \mathbb{P} \left(\mathrm{L} \oplus \mathcal{O}\right)$ as well (see \cite{Fujiki:1992:eKruledmani} and \cite{Tonnesen:1998:eKminruledsurf}). \par
Over here also we consider the K\"ahler classes $2 \pi \left(\mathsf{C} + m S_\infty\right)$ where $m > 0$ only and obtain the expected results on higher extremal K\"ahler and hcscK metrics in these K\"ahler classes, and since being a higher extremal K\"ahler metric (and an hcscK metric respectively) is a scale-invariant property as seen in Subsection \ref{subsec:AnalysisODEBVP} (and in Section \ref{sec:Bando-Futaki} respectively), we can rescale the constructed metrics and generalize the results to all the K\"ahler classes $a \mathsf{C} + b S_\infty$ where $a, b > 0$.
\subsection{The Momentum Construction Method Applied to a Pseudo-Hirzebruch Surface}\label{subsec:MomentConstructgd}
We now use the momentum construction method attributed to Hwang-Singer \cite{Hwang:2002:MomentConstruct} (described briefly in Subsection \ref{subsec:MomentConstruct}) to construct a higher extremal K\"ahler metric $\omega$ in the K\"ahler class $2 \pi \left(\mathsf{C} + m S_\infty\right)$ with $m > 0$ whose top Chern form $c_2 \left(\omega\right)$ satisfies the following equation:
\begin{equation}\label{eq:ChernheKgd}
c_2 \left(\omega\right) = \frac{\mathtt{d}^2 \lambda}{2 \left(2 \pi\right)^2} \omega^2
\end{equation}
where $\nabla^{1,0} \lambda$ is a holomorphic vector field. \par
We carefully go through the calculations involving holomorphic coordinates on the surface $X$ done in Pingali \cite{Pingali:2018:heK} and Sz\'ekelyhidi \cite{Szekelyhidi:2014:eKintro} by following all their conventions and observe where the factors containing the genus $\mathtt{g}$ and the degree $\mathtt{d}$ appear in the expressions of $\omega$, $\omega^2$, the curvature form matrix $\Theta \left(\omega\right)$ and $c_2 \left(\omega\right)$. Note that $\mathtt{d} < 0$ throughout this discussion. \par
As in Subsection \ref{subsec:MomentConstruct} let $\mathtt{p} : X \to \Sigma$ be the fibre bundle projection, $z$ be a local holomorphic coordinate on $\Sigma$, $w$ be a holomorphic fibre coordinate on $\mathrm{L}$ corresponding to a local holomorphic trivialization around $z$, $s := \ln \left| \left(z,w\right) \right|_{h}^2 = \ln \left| w \right|^2 + \ln h \left(z\right)$ be the fibrewise coordinate on the total space of $\mathrm{L}$ minus the zero section, $f \left(s\right)$ be strictly convex and $s - \mathtt{d} f \left(s\right)$ be strictly increasing, and let $\omega$ satisfy the following ansatz (as in \cite{Pingali:2018:heK} and \cite{Szekelyhidi:2014:eKintro}):
\begin{equation}\label{eq:ansatzgd}
\omega = \mathtt{p}^* \omega_\Sigma + \sqrt{-1} \partial \bar{\partial} f \left(s\right)
\end{equation} \par
We have the following coordinate equations (\cite{Pingali:2018:heK} and \cite{Szekelyhidi:2014:eKintro}):
\begin{align}\label{eq:del2bars}
\sqrt{-1} \partial \bar{\partial} s &= \sqrt{-1} \partial \bar{\partial} \ln h \left(z\right) \\
&= -\mathtt{d} \mathtt{p}^* \omega_\Sigma \nonumber
\end{align}
\begin{align}\label{eq:del2barfs}
\sqrt{-1} \partial \bar{\partial} f \left(s\right) &= f' \left(s\right) \sqrt{-1} \partial \bar{\partial} \ln h \left(z\right) + f'' \left(s\right) \sqrt{-1} \frac{d w \wedge d \bar{w}}{\left| w \right|^2} \\
&= -\mathtt{d} f' \left(s\right) \mathtt{p}^* \omega_\Sigma + f'' \left(s\right) \sqrt{-1} \frac{d w \wedge d \bar{w}}{\left| w \right|^2} \nonumber
\end{align}
where we are using the facts that $\mathrm{F} \left(h\right) = - \sqrt{-1} \partial \bar{\partial} \ln h \left(z\right)$ and $\mathrm{F} \left(h\right) = \mathtt{d} \mathtt{p}^* \omega_\Sigma$ (as defined in Subsection \ref{subsec:KConegd}). \par
We write down the expression for $\omega$ as follows:
\begin{equation}\label{eq:omegagd}
\omega = \left(1 - \mathtt{d} f' \left(s\right)\right) \mathtt{p}^* \omega_\Sigma + f'' \left(s\right) \sqrt{-1} \frac{d w \wedge d \bar{w}}{\left| w \right|^2}
\end{equation} \par
We want $\left[\omega\right] = 2 \pi \left(\mathsf{C} + m S_\infty\right)$ and the intersection formulae (\ref{eq:IntersectFormgd}) and (\ref{eq:Sigmagd}) and equation (\ref{eq:omegagd}) help us in computing the following integrals:
\begin{align}\label{eq:areaC}
2 \pi m = \left[\omega\right] \cdot \mathsf{C} := \int\limits_{\mathsf{C}} \omega &= \int\limits_{\mathbb{C} \setminus \left\lbrace 0 \right\rbrace} f'' \left(s\right) \sqrt{-1} \frac{d w \wedge d \bar{w}}{\left| w \right|^2} \\
&= 2 \pi \left( \lim\limits_{s \to \infty} f' \left(s\right) - \lim\limits_{s \to -\infty} f' \left(s\right) \right) \nonumber
\end{align}
\begin{align}\label{eq:areaSinfty}
2 \pi \left(1 - \mathtt{d} m\right) = \left[\omega\right] \cdot S_\infty := \int\limits_{S_\infty} \omega &= \int\limits_{\Sigma} \lim\limits_{s \to \infty} \left(1 - \mathtt{d} f' \left(s\right)\right) \omega_\Sigma \\
&= 2 \pi \left(1 - \mathtt{d} \lim\limits_{s \to \infty} f' \left(s\right)\right) \nonumber
\end{align}
which give us $0 \leq f' \left(s\right) \leq m$. \par
We then compute the curvature form matrix $\Theta \left(\omega\right)$ as follows:
\begin{equation}\label{eq:omega2gd}
\omega^2 = 2 \left(1 - \mathtt{d} f' \left(s\right)\right) f'' \left(s\right) \mathtt{p}^* \omega_\Sigma \sqrt{-1} \frac{d w \wedge d \bar{w}}{\left| w \right|^2}
\end{equation}
\begin{equation}\label{eq:Curv1gd}
\Theta \left(\omega\right) =
\begin{bmatrix}
- \partial \bar{\partial} \ln \left(1 - \mathtt{d} f' \left(s\right)\right) + 2 \left(\mathtt{g} - 1\right) \sqrt{-1} \mathtt{p}^* \omega_\Sigma & 0 \\
0 & - \partial \bar{\partial} \ln \left(f'' \left(s\right)\right)
\end{bmatrix}
\end{equation}
where we are using $\rho_\Sigma = -2 \left(\mathtt{g} - 1\right) \omega_\Sigma$ from the equations (\ref{eq:ArhoSigma}), $\rho_\Sigma$ being the Ricci curvature form of $\omega_\Sigma$. \par
Again as in Subsection \ref{subsec:MomentConstruct} we take the Legendre Transform $F \left(\tau\right)$ in the variable $\tau := f' \left(s\right) \in \left[0, m\right]$ as $f \left(s\right) + F \left(\tau\right) = s \tau$. Then the momentum profile of $\omega$ is $\phi \left(\tau\right) := \frac{1}{F'' \left(\tau\right)} = f'' \left(s\right)$. We further have $f''' \left(s\right) = \phi' \left(\tau\right) \phi \left(\tau\right)$ (\cite{Pingali:2018:heK}). We take the new variable $\gamma := -\mathtt{d} \tau + 1 \in \left[1, -\mathtt{d} m + 1\right]$ after which we get $\gamma = 1 - \mathtt{d} f' \left(s\right)$, $\phi \left(\gamma\right) = f'' \left(s\right)$ and $f''' \left(s\right) = -\mathtt{d} \phi' \left(\gamma\right) \phi \left(\gamma\right)$. We will also need the following coordinate equation (\cite{Pingali:2018:heK} and \cite{Szekelyhidi:2014:eKintro}):
\begin{equation}\label{eq:2delbars}
\sqrt{-1} \partial s \bar{\partial} s = \sqrt{-1} \frac{d w \wedge d \bar{w}}{\left| w \right|^2}
\end{equation} \par
Using the coordinate equations (\ref{eq:del2bars}), (\ref{eq:del2barfs}) and (\ref{eq:2delbars}) we write down the curvature form matrix $\sqrt{-1} \Theta \left(\omega\right)$ in terms of $\phi \left(\gamma\right)$ as follows (\cite{Pingali:2018:heK}):
\begingroup
\addtolength{\jot}{2em}
\begin{gather}\label{eq:Curv2gd}
\sqrt{-1} \Theta \left(\omega\right) =
\begin{bmatrix}
- \sqrt{-1} \partial \bar{\partial} \ln \left(\gamma\right) - 2 \left(\mathtt{g} - 1\right) \mathtt{p}^* \omega_\Sigma & 0 \\
0 & - \sqrt{-1} \partial \bar{\partial} \ln \left(\phi\right)
\end{bmatrix} \\
=
\begin{bmatrix}
\sqrt{-1} \frac{\partial \gamma \bar{\partial} \gamma}{\gamma^2} - \frac{1}{\gamma} \sqrt{-1} \partial \bar{\partial} \gamma - 2 \left(\mathtt{g} - 1\right) \mathtt{p}^* \omega_\Sigma & 0 \\
0 & \frac{\left(\phi'\right)^2 - \phi \phi''}{\phi^2} \sqrt{-1} \partial \gamma \bar{\partial} \gamma - \frac{\phi'}{\phi} \sqrt{-1} \partial \bar{\partial} \gamma
\end{bmatrix} \nonumber \\
=
\begin{bmatrix}
\mathtt{d}^2 \frac{\phi}{\gamma} \left( \frac{\phi}{\gamma} - \phi' \right) \sqrt{-1} \frac{d w \wedge d \bar{w}}{\left| w \right|^2} - \left(\mathtt{d}^2 \frac{\phi}{\gamma} + 2 \left(\mathtt{g} - 1\right)\right) \mathtt{p}^* \omega_\Sigma & 0 \\
0 & - \mathtt{d}^2 \phi'' \phi \sqrt{-1} \frac{d w \wedge d \bar{w}}{\left| w \right|^2} - \mathtt{d}^2 \phi' \mathtt{p}^* \omega_\Sigma
\end{bmatrix} \nonumber
\end{gather}
\endgroup
\par
The top Chern form of $\omega$ is given in terms of $\phi \left(\gamma\right)$ by:
\begin{equation}\label{eq:Cherngd}
c_2 \left(\omega\right) = \frac{1}{\left(2 \pi\right)^2} \mathtt{p}^* \omega_\Sigma \sqrt{-1} \frac{d w \wedge d \bar{w}}{\left| w \right|^2} \mathtt{d}^2 \frac{\phi}{\gamma^2} \left( \gamma \left(\mathtt{d}^2 \phi + 2 \left(\mathtt{g} - 1\right) \gamma\right) \phi'' + \mathtt{d}^2 \phi' \left(\phi' \gamma - \phi\right) \right)
\end{equation} \par
Comparing equations (\ref{eq:ChernheKgd}), (\ref{eq:omega2gd}) and (\ref{eq:Cherngd}) and since $\nabla^{1,0} \lambda = - \mathtt{d} \lambda' w \frac{\partial}{\partial w}$ (from equation (\ref{eq:gradlambda})) is a holomorphic vector field if and only if $\lambda = A \gamma + B$ for some $A, B \in \mathbb{R}$, we derive the following ODE for $\phi \left(\gamma\right)$ on $\left[1, -\mathtt{d} m + 1\right]$ (for some $C \in \mathbb{R}$):
\begin{equation*}
\begin{gathered}
\gamma \left(\mathtt{d}^2 \phi + 2 \left(\mathtt{g} - 1\right) \gamma\right) \phi'' + \mathtt{d}^2 \phi' \left(\phi' \gamma - \phi\right) = \left(A \gamma + B\right) \gamma^3 \\
2 \left(\mathtt{g} - 1\right) \phi'' + \mathtt{d}^2 \left( \frac{\phi \phi'}{\gamma} \right)' = A \gamma^2 + B \gamma \\
2 \left(\mathtt{g} - 1\right) \phi' + \mathtt{d}^2 \frac{\phi \phi'}{\gamma} = A \frac{\gamma^3}{3} + B \frac{\gamma^2}{2} + C
\end{gathered}
\end{equation*}
\begin{equation}\label{eq:ODEgd}
\left( 2 \left(\mathtt{g} - 1\right) \gamma + \mathtt{d}^2 \phi \right) \phi' = A \frac{\gamma^4}{3} + B \frac{\gamma^3}{2} + C \gamma
\end{equation} \par
As explained in \cite{Szekelyhidi:2014:eKintro}, for $\omega$ to extend smoothly across the zero and infinity divisors of $X$ we should have the following boundary conditions on $\phi$:
\begin{equation}\label{eq:BVP0gd}
\begin{gathered}
\phi \left(1\right) = \lim\limits_{\gamma \to 1} \phi \left(\gamma\right) = \lim\limits_{s \to -\infty} f'' \left(s\right) = 0 \\
\phi' \left(1\right) = -\frac{1}{\mathtt{d}} \lim\limits_{\gamma \to 1} \left(-\mathtt{d} \phi' \left(\gamma\right)\right) = -\frac{1}{\mathtt{d}} \lim\limits_{s \to -\infty} \frac{f''' \left(s\right)}{f'' \left(s\right)} = -\frac{1}{\mathtt{d}}
\end{gathered}
\end{equation}
\begin{equation}\label{eq:BVPinftygd}
\begin{gathered}
\phi \left(-\mathtt{d} m + 1\right) = \lim\limits_{\gamma \to -\mathtt{d} m + 1} \phi \left(\gamma\right) = \lim\limits_{s \to \infty} f'' \left(s\right) = 0 \\
\phi' \left(-\mathtt{d} m + 1\right) = -\frac{1}{\mathtt{d}} \lim\limits_{\gamma \to -\mathtt{d} m + 1} \left(-\mathtt{d} \phi' \left(\gamma\right)\right) = -\frac{1}{\mathtt{d}} \lim\limits_{s \to \infty} \frac{f''' \left(s\right)}{f'' \left(s\right)} = \frac{1}{\mathtt{d}}
\end{gathered}
\end{equation}
where we are using the relations between $\phi \left(\gamma\right)$ and $f \left(s\right)$ (\cite{Pingali:2018:heK}). Also as $f$ was required to be strictly convex we must have $\phi > 0$ on $\left(1, -\mathtt{d} m + 1\right)$ in addition to the boundary conditions (\ref{eq:BVP0gd}) and (\ref{eq:BVPinftygd}). \par
Now as in Subsection \ref{subsec:AnalysisODEBVP} we define $p \left(\gamma\right) := \mathtt{d}^2 \left( A \frac{\gamma^3}{3} + B \frac{\gamma^2}{2} + C \right)$ and do the change of variables $v := \frac{\left( 2 \left(\mathtt{g} - 1\right) \gamma + \mathtt{d}^2 \phi \right)^2}{2}$, $\gamma \in \left[1, -\mathtt{d} m + 1\right]$ and obtain the following ODE BVP:
\begin{equation}\label{eq:ODEBVPgd}
\begin{gathered}
v' = 2 \left(\mathtt{g} - 1\right) \sqrt{2} \sqrt{v} + p \left(\gamma\right) \gamma \hspace{5pt} \text{on} \hspace{5pt} \left[1, -\mathtt{d} m + 1\right] \\
v \left(1\right) = 2 \left(\mathtt{g} - 1\right)^2 \hspace{1pt}, \hspace{5pt} v \left(-\mathtt{d} m + 1\right) = 2 \left(\mathtt{g} - 1\right)^2 \left(-\mathtt{d} m + 1\right)^2 \\
v' \left(1\right) = 2 \left(\mathtt{g} - 1\right) \left[ 2 \left(\mathtt{g} - 1\right) - \mathtt{d} \right] \hspace{1pt}, \hspace{5pt} v' \left(-\mathtt{d} m + 1\right) = 2 \left(\mathtt{g} - 1\right) \left(-\mathtt{d} m + 1\right) \left[ 2 \left(\mathtt{g} - 1\right) + \mathtt{d} \right] \\
v \left(\gamma\right) > 2 \left(\mathtt{g} - 1\right)^2 \gamma^2 \hspace{5pt} \text{on} \hspace{5pt} \left(1, -\mathtt{d} m + 1\right)
\end{gathered}
\end{equation} \par
After this point the entire analysis developed in Subsections \ref{subsec:AnalysisODEBVP}, \ref{subsec:Proof1} and \ref{subsec:Proof2} with all the bounds and estimates can be worked out in this general setting as well and in exactly the same way it can be proven that for each $m > 0$ there exist unique $A, B, C \in \mathbb{R}$ such that there exists a unique smooth solution $v$ to the ODE BVP (\ref{eq:ODEBVPgd}) on $\left[1, -\mathtt{d} m + 1\right]$ satisfying all the required conditions. \par
Hence given a pseudo-Hirzebruch surface $X$ of genus $\mathtt{g} \geq 2$ and degree $\mathtt{d} < 0$ we can construct higher extremal K\"ahler metrics $\omega$ satisfying equation (\ref{eq:ChernheKgd}) in the K\"ahler classes $2 \pi \left(\mathsf{C} + m S_\infty\right)$ for all $m > 0$, and by rescaling these constructed metrics we will get higher extremal K\"ahler metrics in all the K\"ahler classes on $X$. Then we can verify that these constructed metrics $\omega$ cannot be hcscK i.e. $\nabla^{1,0} \lambda \neq 0$ just like Pingali \cite{Pingali:2018:heK} did in the proof of Theorem \ref{thm:nonhcscK} and by the arguments in Section \ref{sec:Bando-Futaki} it will follow that hcscK metrics do not exist on the surface $X$.
\subsection{The Case of Positive Degree}\label{subsec:DegPos}
Now we will see what happens when $\mathtt{d} > 0$. All the other things about $X = \mathbb{P} \left(\mathrm{L} \oplus \mathcal{O}\right)$ are exactly the same as in Subsections \ref{subsec:KConegd} and \ref{subsec:MomentConstructgd}. Let us first observe that if $\mathrm{L}$ has degree $\mathtt{d}$ over $\Sigma$ then the dual vector bundle $\mathrm{L}^*$ is a holomorphic line bundle of degree $-\mathtt{d}$ over $\Sigma$ and there is a canonical isomorphism of complex manifolds $\mathbb{P} \left(\mathrm{L} \oplus \mathcal{O}\right) \cong \mathbb{P} \left(\mathrm{L}^* \oplus \mathcal{O}\right)$ given in local holomorphic coordinates as $\left(z, w\right) \mapsto \left(z, w^{-1}\right)$ which maps the zero and infinity divisors of $\mathbb{P} \left(\mathrm{L} \oplus \mathcal{O}\right)$ to the infinity and zero divisors of $\mathbb{P} \left(\mathrm{L}^* \oplus \mathcal{O}\right)$ respectively. Just as in Subsection \ref{subsec:KConegd} we use Theorem \ref{thm:Nakai-Moishezon} and Corollary \ref{cor:TFKCone} and the intersection formulae (\ref{eq:IntersectFormgd}) and (\ref{eq:Sigmagd}) and describe the K\"ahler cone of $X$ as follows (Fujiki \cite{Fujiki:1992:eKruledmani} and T{\o}nnesen-Friedman \cite{Tonnesen:1998:eKminruledsurf}):
\begin{equation*}
H^2 \left(X, \mathbb{R}\right) = \mathbb{R} \mathsf{C} \oplus \mathbb{R} S_0
\end{equation*}
\begin{align}\label{eq:KConeXgdpos}
H^{\left(1,1\right)} \left(X, \mathbb{R}\right)^+ &= \left\lbrace a \mathsf{C} + b S_0 \hspace{3pt} \vert \hspace{3pt} 2 a b + \mathtt{d} b^2 > 0 \hspace{1pt}, \hspace{5pt} b > 0 \hspace{1pt}, \hspace{5pt} a + \mathtt{d} b > 0 \hspace{1pt}, \hspace{5pt} a > 0 \hspace{1pt}, \hspace{5pt} a + \mathtt{d} b > 0 \right\rbrace \\
&= \left\lbrace a \mathsf{C} + b S_0 \hspace{3pt} \vert \hspace{3pt} a, b > 0 \right\rbrace \nonumber
\end{align}
Note that in the description of the K\"ahler cone over here we are using the zero divisor of $X$ instead of the infinity divisor. \par
We want to construct a higher extremal K\"ahler metric $\omega$ in the K\"ahler class $2 \pi \left(\mathsf{C} + m S_0\right)$ with $m > 0$ satisfying the equation (\ref{eq:ChernheKgd}). Following the calculations done in Subsection \ref{subsec:MomentConstructgd} we write the ansatz (\ref{eq:ansatzgd}), and then the expression (\ref{eq:omegagd}) holds true in this case as well along with all the coordinate equations. Only the computation of the integrals in equations (\ref{eq:areaC}) and (\ref{eq:areaSinfty}) changes in the following way:
\begin{align}\label{eq:areaCpos}
2 \pi m = \left[\omega\right] \cdot \mathsf{C} := \int\limits_{\mathsf{C}} \omega &= \int\limits_{\mathbb{C} \setminus \left\lbrace 0 \right\rbrace} f'' \left(s\right) \sqrt{-1} \frac{d w \wedge d \bar{w}}{\left| w \right|^2} \\
&= 2 \pi \left( \lim\limits_{s \to \infty} f' \left(s\right) - \lim\limits_{s \to -\infty} f' \left(s\right) \right) \nonumber
\end{align}
\begin{align}\label{eq:areaSinftypos}
2 \pi \left(1 + \mathtt{d} m\right) = \left[\omega\right] \cdot S_0 := \int\limits_{S_0} \omega &= \int\limits_{\Sigma} \lim\limits_{s \to -\infty} \left(1 - \mathtt{d} f' \left(s\right)\right) \omega_\Sigma \\
&= 2 \pi \left(1 - \mathtt{d} \lim\limits_{s \to -\infty} f' \left(s\right)\right) \nonumber
\end{align}
which gives us $-m \leq f' \left(s\right) \leq 0$. We take the variable of the Legendre Transform as $\tau := f' \left(s\right) \in \left[-m, 0\right]$ and the new variable as $\gamma := -\mathtt{d} \tau + 1 \in \left[1, \mathtt{d} m + 1\right]$ while $F$ and $\phi$ remain the same as in Subsection \ref{subsec:MomentConstructgd}. The expressions for $\Theta \left(\omega\right)$ and $c_2 \left(\omega\right)$ also remain unchanged and hence we will obtain the following ODE for $\phi \left(\gamma\right)$, $\gamma \in \left[1, \mathtt{d} m + 1\right]$:
\begin{equation}\label{eq:ODEgdpos}
\left( 2 \left(\mathtt{g} - 1\right) \gamma + \mathtt{d}^2 \phi \right) \phi' = A \frac{\gamma^4}{3} + B \frac{\gamma^3}{2} + C \gamma
\end{equation}
The appropriate boundary conditions for $\phi$ in this case are the following (Pingali \cite{Pingali:2018:heK} and Sz\'ekelyhidi \cite{Szekelyhidi:2014:eKintro}):
\begin{equation}\label{eq:BVP0gdpos}
\begin{gathered}
\phi \left(1\right) = \lim\limits_{s \to \infty} f'' \left(s\right) = 0 \\
\phi' \left(1\right) = -\frac{1}{\mathtt{d}} \lim\limits_{s \to \infty} \frac{f''' \left(s\right)}{f'' \left(s\right)} = \frac{1}{\mathtt{d}}
\end{gathered}
\end{equation}
\begin{equation}\label{eq:BVPinftygdpos}
\begin{gathered}
\phi \left(\mathtt{d} m + 1\right) = \lim\limits_{s \to -\infty} f'' \left(s\right) = 0 \\
\phi' \left(\mathtt{d} m + 1\right) = -\frac{1}{\mathtt{d}} \lim\limits_{s \to -\infty} \frac{f''' \left(s\right)}{f'' \left(s\right)} = -\frac{1}{\mathtt{d}}
\end{gathered}
\end{equation}
Again as in Subsection \ref{subsec:MomentConstructgd}, $\phi > 0$ on $\left(1, \mathtt{d} m + 1\right)$. \par
Comparing the equations (\ref{eq:ODEgdpos}), (\ref{eq:BVP0gdpos}) and (\ref{eq:BVPinftygdpos}) with the equations (\ref{eq:ODEgd}), (\ref{eq:BVP0gd}) and (\ref{eq:BVPinftygd}) we observe that in the case when $\mathtt{d} > 0$ we get the same ODE BVP for the momentum profile $\phi$ of the K\"ahler metric $\omega$ as we would have gotten for a holomorphic line bundle of degree $-\mathtt{d}$ over $\Sigma$, except that the to be constructed metric $\omega$ now belongs to the K\"ahler class $2 \pi \left(\mathsf{C} + m S_0\right)$ on $X = \mathbb{P} \left(\mathrm{L} \oplus \mathcal{O}\right)$. \par
So for this case as well we obtain the same results on the existence of higher extremal K\"ahler metrics and the non-existence of hcscK metrics as those gotten in Subsection \ref{subsec:MomentConstructgd}. \par
We finally summarize the work done in Subsections \ref{subsec:KConegd}, \ref{subsec:MomentConstructgd} and \ref{subsec:DegPos} in the following result:
\begin{theorem}\label{thm:Generalgd}
Given a pseudo-Hirzebruch surface $X := \mathbb{P} \left(\mathrm{L} \oplus \mathcal{O}\right)$ where $\mathrm{L}$ is a holomorphic line bundle of degree $\mathtt{d} \neq 0$ on a compact Riemann surface $\Sigma$ of genus $\mathtt{g} \geq 2$, for all $a, b > 0$ there exists a K\"ahler metric $\eta$ on $X$, which is higher extremal K\"ahler but not hcscK, satisfying the following:
\begin{equation}\label{eq:probheKnonhcscKgd}
\left[\eta\right] = a \mathsf{C} + b S_\infty \hspace{1pt}, \hspace{5pt} c_2 \left(\eta\right) = \frac{\mathtt{d}^2 \lambda}{2 a^2} \eta^2 \hspace{1pt}, \hspace{5pt} \nabla^{1,0} \lambda \neq 0 \in \mathfrak{h} \left(X\right) \hspace{7pt} \left(\text{If $\mathtt{d} < 0$}\right)
\end{equation}
\begin{equation}\label{eq:probheKnonhcscKgdpos}
\left[\eta\right] = a \mathsf{C} + b S_0 \hspace{1pt}, \hspace{5pt} c_2 \left(\eta\right) = \frac{\mathtt{d}^2 \lambda}{2 a^2} \eta^2 \hspace{1pt}, \hspace{5pt} \nabla^{1,0} \lambda \neq 0 \in \mathfrak{h} \left(X\right) \hspace{7pt} \left(\text{If $\mathtt{d} > 0$}\right)
\end{equation}
Further hcscK metrics do not exist in any K\"ahler class on $X$.
\end{theorem}
\numberwithin{equation}{section}
\numberwithin{definition}{section}
\numberwithin{lemma}{section}
\numberwithin{theorem}{section}
\numberwithin{corollary}{section}
\numberwithin{remark}{section}
\numberwithin{case}{section}
\numberwithin{claim}{section}
\numberwithin{motivation}{section}
\numberwithin{question}{section}
\section{Analogy with the Extremal K\"ahler Setup}\label{sec:Analogy}
In T{\o}nnesen-Friedman \cite{Tonnesen:1998:eKminruledsurf} and Hwang-Singer \cite{Hwang:2002:MomentConstruct} where the usual extremal K\"ahler analogue of this problem was studied, the following $2$ questions were asked regarding the existence and uniqueness of extremal K\"ahler metrics:
\begin{question}\label{ques:q1}
Given a compact K\"ahler manifold having an extremal K\"ahler metric in a given K\"ahler class, does it have an extremal K\"ahler metric in each K\"ahler class?
\end{question}
\begin{question}\label{ques:q2}
Given a compact K\"ahler manifold having an extremal K\"ahler metric in a given K\"ahler class, is this extremal K\"ahler metric unique in its K\"ahler class modulo the group action of the maximal connected group of automorphisms of the K\"ahler manifold?
\end{question} \par
T{\o}nnesen-Friedman \cite{Tonnesen:1998:eKminruledsurf} had proven that for a general pseudo-Hirzebruch surface $X$ of genus $\mathtt{g} \geq 2$ and degree $\mathtt{d} \neq 0$, both \textit{Question} \ref{ques:q1} and \textit{Question} \ref{ques:q2} cannot have an affirmative answer simultaneously i.e. if every K\"ahler class on $X$ has an extremal K\"ahler metric then there exists a K\"ahler class in which the uniqueness condition of extremal K\"ahler metrics (even upto automorphisms of $X$) fails, and vice versa if every K\"ahler class on $X$, which can be represented by an extremal K\"ahler metric, has a unique such extremal K\"ahler representative (upto automorphisms of $X$) then there exists a K\"ahler class in which extremal K\"ahler metrics do not exist (see \cite{Tonnesen:1998:eKminruledsurf} for the details). Specifically from the work of T{\o}nnesen-Friedman \cite{Tonnesen:1998:eKminruledsurf} it followed that the momentum construction method of Hwang-Singer \cite{Hwang:2002:MomentConstruct} yields extremal K\"ahler metrics which are not cscK only in the K\"ahler classes $a \mathsf{C} + b S_\infty$ if $\mathtt{d} < 0$ (and $a \mathsf{C} + b S_0$ if $\mathtt{d} > 0$) with $0 < \frac{b}{a} < k_1$ for a unique $k_1 = k_1 \left(\mathtt{g}, \mathtt{d}\right) = k_1 \left(\mathtt{g}, -\mathtt{d}\right) \in \mathbb{R}_{> 0}$. It later followed from the work of Apostolov, Calderbank et al. \cite{Apostolov:2008:Hamiltonian} that there do not exist extremal K\"ahler metrics (even without the symmetries of the momentum construction method) in the K\"ahler classes $a \mathsf{C} + b S_\infty$ if $\mathtt{d} < 0$ (and $a \mathsf{C} + b S_0$ if $\mathtt{d} > 0$) with $\frac{b}{a} \geq k_1$. In the special case where $\mathtt{g} = 2$ and $\mathtt{d} = -1$ whose exposition is given in Sz\'ekelyhidi \cite{Szekelyhidi:2014:eKintro} (and whose higher extremal K\"ahler analogue is studied in Pingali \cite{Pingali:2018:heK} and in this paper), it is seen that $k_1 = k_1 \left(2, -1\right) \approx 18.889$ (\cite{Szekelyhidi:2014:eKintro}). But in complete contrast to that, we have been able to construct higher extremal K\"ahler metrics which are not hcscK in all the K\"ahler classes $a \mathsf{C} + b S_\infty$ if $\mathtt{d} < 0$ (and $a \mathsf{C} + b S_0$ if $\mathtt{d} > 0$) with $a, b > 0$, which precisely constitute the K\"ahler cone of $X$ (Fujiki \cite{Fujiki:1992:eKruledmani} and T{\o}nnesen-Friedman \cite{Tonnesen:1998:eKminruledsurf}). So \textit{Question} \ref{ques:q1} has negative answer for all pseudo-Hirzebruch surfaces whereas we have answered the higher extremal K\"ahler version of \textit{Question} \ref{ques:q1} affirmatively for all pseudo-Hirzebruch surfaces. Regarding \textit{Question} \ref{ques:q2}, it was only recently proved by Berman-Berndtsson \cite{Berman:2017:UniqueeK} that on any compact K\"ahler $n$-manifold $\mathrm{M}$, cscK metrics (and even extremal K\"ahler metrics) in a given K\"ahler class are unique modulo the group action of $\operatorname{Aut}_0 \left(\mathrm{M}\right)$ (if they exist at all) where $\operatorname{Aut}_0 \left(\mathrm{M}\right)$ denotes the maximal connected group of automorphisms of $\mathrm{M}$. We hope to explore the higher extremal K\"ahler version of \textit{Question} \ref{ques:q2} at least for pseudo-Hirzebruch surfaces in our future works. \par
Besides these, T{\o}nnesen-Friedman \cite{Tonnesen:1998:eKminruledsurf} had already proven that cscK metrics do not exist on a general pseudo-Hirzebruch surface by actually computing the Futaki invariant on an arbitrary K\"ahler class and for a specific holomorphic vector field using a certain formula proven by LeBrun-Simanca \cite{LeBrun:1994:eKcmplxdeform}, and observing that it turns out to be non-zero. But for proving the non-existence of hcscK metrics on pseudo-Hirzebruch surfaces we have instead proven a result about the top Bando-Futaki invariant (viz. Theorem \ref{thm:BFheK}) analogous to the following result about the Futaki invariant proven by Calabi \cite{Calabi:1985:eK2} and LeBrun-Simanca \cite{LeBrun:1994:eKcmplxdeform} (and mentioned in \cite{Szekelyhidi:2014:eKintro} and \cite{Tonnesen:1998:eKminruledsurf}) and then used the fact that the momentum construction method in our higher extremal K\"ahler case yields a higher extremal K\"ahler metric which is not hcscK in each K\"ahler class on the surface.
\begin{theorem}[Calabi, LeBrun-Simanca]\label{thm:FeK}
Let $\mathrm{M}$ be a compact K\"ahler $n$-manifold. Let $\mathcal{F}_1 : \mathfrak{h} \left(\mathrm{M}\right) \times H^{\left(1,1\right)} \left(\mathrm{M}, \mathbb{R}\right)^+ \to \mathbb{R}$ be the Futaki invariant on $\mathrm{M}$. Let $\omega$ be an extremal K\"ahler metric on $\mathrm{M}$. Then $\omega$ is cscK if and only if $\mathcal{F}_1 \left(\cdot, \left[\omega\right]\right) \equiv 0$ on $\mathfrak{h} \left(\mathrm{M}\right)$. Furthermore let $\omega$ be cscK. Then every extremal K\"ahler metric in the K\"ahler class $\left[\omega\right]$ is cscK.
\end{theorem} \par
We hope to work out in the future a formula for the calculation of the top Bando-Futaki invariant on a compact K\"ahler surface with some nice properties which would be analogous to the formula of LeBrun-Simanca \cite{LeBrun:1994:eKcmplxdeform} (used by T{\o}nnesen-Friedman \cite{Tonnesen:1998:eKminruledsurf}) for the calculation of the Futaki invariant on such a surface, as it seems to be an exercise of independent importance. \par
Finally coming to the case when the Riemann surface $\Sigma$ has genus $\mathtt{g} = 0$ i.e. $\Sigma \cong \mathbb{C} \mathbb{P}^1$ and the holomorphic line bundle $\mathrm{L}$ has degree $\mathtt{d} \neq 0$ i.e. $\mathrm{L} \cong \mathcal{O} \left(\mathtt{d}\right)$, then the ruled surface $X \cong \mathbb{P} \left(\mathcal{O} \left(\mathtt{d}\right) \oplus \mathcal{O}\right)$ i.e. $X$ is a Hirzebruch surface. Calabi \cite{Calabi:1982:eK} had constructed an extremal K\"ahler metric which is not cscK in each K\"ahler class on $X$ and then Theorem \ref{thm:FeK} gave the non-existence of any cscK metrics on $X$ (see \cite{Tonnesen:1998:eKminruledsurf}). The analogous problem of constructing higher extremal K\"ahler metrics on Hirzebruch surfaces would be interesting but is completely out of the scope of this paper. In the case when $\mathtt{g} = 1$ i.e. $\Sigma$ is a complex elliptic curve representing an embedding of the complex torus $S^1 \times S^1$ into $\mathbb{C} \mathbb{P}^2$, the questions about the existence of higher extremal K\"ahler and hcscK metrics on $X := \mathbb{P} \left(\mathrm{L} \oplus \mathcal{O}\right)$ seem to be even intriguing and mysterious. \par
It would also be interesting to see in the future if there are some nice properties of the set of all K\"ahler classes of higher extremal K\"ahler metrics (along with the subset of all K\"ahler classes of hcscK metrics) analogous to those of the set of all K\"ahler classes of extremal K\"ahler metrics (along with the subset of all K\"ahler classes of cscK metrics) proven by LeBrun-Simanca \cite{LeBrun:1994:eKcmplxdeform,LeBrun:1994:eKclass}, and if a `Deformation Theory' can be developed for higher extremal K\"ahler and hcscK metrics analogous to the one for extremal K\"ahler and cscK metrics developed by LeBrun-Simanca \cite{LeBrun:1994:eKcmplxdeform,LeBrun:1994:eKclass}. In a different direction, writing down a `higher Calabi functional' and a `higher Mabuchi functional', which can characterize higher extremal K\"ahler and hcscK metrics in ways analogous to the characterization of extremal K\"ahler and cscK metrics by the Calabi functional and the Mabuchi functional, will be an important area of exploration in our future works. Understanding the properties of these functionals on the space of all K\"ahler metrics in a given K\"ahler class on a compact K\"ahler manifold may possibly yield us some results about the uniqueness of hcscK (and even higher extremal K\"ahler) metrics modulo automorphisms similar to those about the uniqueness of cscK (and also extremal K\"ahler) metrics modulo automorphisms recently obtained in Berman-Berndtsson \cite{Berman:2017:UniqueeK}. More generally we would be interested in building a theory of higher extremal K\"ahler metrics in our future works by looking at analogous results from the theory of extremal K\"ahler metrics given in Calabi \cite{Calabi:1982:eK,Calabi:1985:eK2}, Sz\'ekelyhidi \cite{Szekelyhidi:2014:eKintro} and Tian \cite{Tian:2000:CanonMetr} and the relatively recent work of Berman-Berndtsson \cite{Berman:2017:UniqueeK}. \par
Just like Pingali's higher extremal K\"ahler and hcscK metrics \cite{Pingali:2018:heK} are a generalization (or an extension) of extremal K\"ahler and cscK metrics to the level of the top cohomology of a compact K\"ahler manifold, along a similar way but in a different direction one can view Maschler's `central K\"ahler metrics' whose special cases are `constant central curvature K\"ahler metrics' \cite{Maschler:2003:centK}. The definitions of these metrics mimic those of extremal K\"ahler and cscK metrics with scalar curvature being replaced by something called as `central curvature' \cite{Maschler:2003:centK}. Even these metrics involve going up to the top cohomology and their theory was built by Maschler \cite{Maschler:2003:centK} by treating central curvature as analogous to the usual scalar curvature and then trying to find the counterparts of analogous results from the theory of the usual extremal K\"ahler and cscK metrics. Even we hope to do something along similar lines for higher extremal K\"ahler and hcscK metrics.
\section*{Acknowledgements}
The author is greatly indebted to his Ph.D. supervisor Prof. Vamsi Pritham Pingali and his Ph.D. co-supervisor Prof. Ved Datar for their continuous guidance and support throughout his doctoral research work and for their kindness and helpful attitude in the whole process. The author thanks Prof. Vamsi Pritham Pingali for guiding him through some of their own research work, for suggesting him this research problem, for having fruitful discussions with him on multiple occasions and for giving him deep insights into this research topic. The author also thanks Prof. Ved Datar for having important discussions with him, for mentioning some key points and providing some good references needed to tackle this problem and for giving him helpful suggestions at different times. Lastly thanks go to both of them for introducing the author to this novel research area and also for checking the author's research work, giving crucial feedback and suggesting essential improvements in the same. The author highly appreciates the extremely important and useful comments and suggestions made by the reviewer which helped in the improvement of this paper and also thanks them for the same.

\begin{thebibliography}{10}
%
\bibitem{Apostolov:2008:Hamiltonian}
Vestislav Apostolov, David M. J. Calderbank, Paul Gauduchon and Christina W. T{\o}nnesen-Friedman.
\newblock {\em Hamiltonian 2-Forms in K\"ahler Geometry. III. Extremal Metrics and Stability}.
\newblock Invent. Math. {\bf 173} (2008), no. 3, 547-601, DOI 10.1007/s00222-008-0126-x.
%
\bibitem{Aubin:1998:NonlinRiemGeom}
Thierry Aubin.
\newblock {\em Some Nonlinear Problems in Riemannian Geometry}.
\newblock Springer Monographs in Mathematics.
\newblock Springer, 1998, ISBN 978-3-662-13006-3, DOI 10.1007/978-3-662-13006-3.
%
\bibitem{Bando:2006:HarmonObstruct}
Shigetoshi Bando.
\newblock {\em An Obstruction for Chern Class Forms to be Harmonic}.
\newblock Kodai Math. Jour. {\bf 29} (2006), no. 3, 337-345, DOI 10.2996/kmj/1162478766.
%
\bibitem{Barth:2004:CmpctCmplxSurf}
Wolf P. Barth, Klaus Hulek, Chris A. M. Peters and Antonius Van de Ven.
\newblock {\em Compact Complex Surfaces}.
\newblock Ergebnisse der Mathematik und ihrer Grenzgebiete. 3. Folge. A Series of Modern Surveys in Mathematics.
\newblock Springer, 2004, ISBN 978-3-540-00832-2, DOI 10.1007/978-3-642-57739-0.
%
\bibitem{Berman:2017:UniqueeK}
Robert J. Berman and Bo Berndtsson.
\newblock {\em Convexity of the $K$-Energy on the Space of K\"ahler Metrics and Uniqueness of Extremal Metrics}.
\newblock Jour. Amer. Math. Soc. {\bf 30} (2017), no. 4, 1165-1196, DOI 10.1090/jams/880.
%
\bibitem{Calabi:1982:eK}
Eugenio Calabi.
\newblock {\em Extremal K\"ahler Metrics}.
\newblock Seminar on Differential Geometry (ed. Shing-Tung Yau).
\newblock Ann. Math. Stud. vol. 102, Princeton Univ. Press, 1982, p. 259-290, DOI 10.1515/9781400881918.
%
\bibitem{Calabi:1985:eK2}
Eugenio Calabi.
\newblock {\em Extremal K\"ahler Metrics II}.
\newblock Differential Geometry and Complex Analysis (eds. Isaac Chavel and Hershel M. Farkas).
\newblock Springer, 1985, p. 95-114, DOI 10.1007/978-3-642-69828-6\_8.
%
\bibitem{Fujiki:1992:eKruledmani}
Akira Fujiki.
\newblock {\em Remarks on Extremal K\"ahler Metrics on Ruled Manifolds}.
\newblock Nagoya Math. Jour. {\bf 126} (1992), 89-101, DOI 10.1017/S0027763000004001.
%
\bibitem{Futaki:1983:ObstructKE}
Akito Futaki.
\newblock {\em An Obstruction to the Existence of Einstein K\"ahler Metrics}.
\newblock Invent. Math. {\bf 73} (1983), no. 3, 437-443, DOI 10.1007/BF01388438.
%
\bibitem{Futaki:2006:HarmonicChern}
Akito Futaki.
\newblock {\em Harmonic Total Chern Forms and Stability}.
\newblock Kodai Math. Jour. {\bf 29} (2006), no. 3, 346-369, DOI 10.2996/kmj/1162478767.
%
\bibitem{Futaki:2008:HolomorphicVF}
Akito Futaki.
\newblock {\em Holomorphic Vector Fields and Perturbed Extremal K\"ahler Metrics}.
\newblock Jour. Sympl. Geom. {\bf 6} (2008), no. 2, 127-138, DOI 10.4310/JSG.2008.v6.n2.a1.
%
\bibitem{Hwang:2002:MomentConstruct}
Andrew D. Hwang and Michael A. Singer.
\newblock {\em A Momentum Construction for Circle-Invariant K\"ahler Metrics}.
\newblock Trans. Amer. Math. Soc. {\bf 354} (2002), no. 6, 2285-2325, DOI 10.1090/S0002-9947-02-02965-3.
%
\bibitem{Lamari:1999:Kcone}
Ahc\`ene Lamari.
\newblock {\em Le C\^one K\"ahl\'erien d'une Surface} (French) [The K\"ahler Cone of a Surface].
\newblock Jour. Math. Pur. Appl. {\bf 78} (1999), no. 3, 249-263, DOI 10.1016/S0021-7824(98)00005-1.
%
\bibitem{LeBrun:1994:eKcmplxdeform}
Claude LeBrun and Santiago R. Simanca.
\newblock {\em Extremal K\"ahler Metrics and Complex Deformation Theory}.
\newblock Geom. Func. Anal. {\bf 4} (1994), no. 3, 298-336, DOI 10.1007/BF01896244.
%
\bibitem{LeBrun:1994:eKclass}
Claude LeBrun and Santiago R. Simanca.
\newblock {\em On the K\"ahler Classes of Extremal Metrics}.
\newblock Geometry and Global Analysis (Sendai, 1994) (eds. Takeshi Kotake, Seiki Nishikawa and Richard Schoen).
\newblock Tohoku Univ., 1994, p. 255-271.
%
\bibitem{LeBrun:1993:scalflatKcmpctcmplxsurf}
Claude LeBrun and Michael Singer.
\newblock {\em Existence and Deformation Theory for Scalar-Flat K\"ahler Metrics on Compact Complex Surfaces}.
\newblock Invent. Math. {\bf 112} (1993), no. 2, 273-313, DOI 10.1007/BF01232436.
%
\bibitem{Maschler:2003:centK}
Gideon Maschler.
\newblock {\em Central K\"ahler Metrics}.
\newblock Trans. Amer. Math. Soc. {\bf 355} (2003), no. 6, 2161-2182, DOI 10.1090/S0002-9947-03-03161-1.
%
\bibitem{Pingali:2018:heK}
Vamsi Pritham Pingali.
\newblock {\em A Note on Higher Extremal Metrics}.
\newblock Trans. Amer. Math. Soc. {\bf 370} (2018), no. 10, 6995-7010, DOI 10.1090/tran/7416.
%
\bibitem{Szekelyhidi:2014:eKintro}
G\'abor Sz\'ekelyhidi.
\newblock {\em An Introduction to Extremal K\"ahler Metrics}.
\newblock Grad. Stud. Math. vol. 152, Amer. Math. Soc., 2014, ISBN 978-1-4704-1047-6, DOI 10.1090/gsm/152.
%
\bibitem{Tian:2000:CanonMetr}
Gang Tian.
\newblock {\em Canonical Metrics in K\"ahler Geometry}.
\newblock Lectures in Mathematics, ETH Z\"urich.
\newblock Birkh\"auser, 2000, ISBN 978-3-0348-8389-4, DOI 10.1007/978-3-0348-8389-4.
%
\bibitem{Tonnesen:1998:eKminruledsurf}
Christina W. T{\o}nnesen-Friedman.
\newblock {\em Extremal K\"ahler Metrics on Minimal Ruled Surfaces}.
\newblock Jour. Rein. Angew. Math. {\bf 502} (1998), 175-197, DOI 10.1515/crll.1998.086.
%
\bibitem{Yau:1977:CalabiConject}
Shing-Tung Yau.
\newblock {\em Calabi's Conjecture and Some New Results in Algebraic Geometry}.
\newblock Proc. Nat. Acad. Sci. U.S.A. {\bf 74} (1977), no. 5, 1798-1799, DOI 10.1073/pnas.74.5.1798.
%
\bibitem{Yau:2006:Perspectives}
Shing-Tung Yau.
\newblock {\em Perspectives on Geometric Analysis}.
\newblock Surveys in Differential Geometry. Vol. X.
\newblock Surv. Differ. Geom. vol. 10, Int. Press, 2006, p. 275-379, DOI 10.4310/SDG.2005.v10.n1.a8.
%
\end{thebibliography}
\end{document}